\documentclass[11pt,a4paper]{article} 
\usepackage[T1]{fontenc}
\usepackage[english]{babel}
\usepackage[utf8]{inputenc}
\usepackage{amsfonts}
\usepackage{amsmath}
\usepackage{amssymb}
\usepackage{amsthm}
\usepackage{enumerate}
\usepackage{graphicx}
\usepackage{dsfont}
\usepackage{array}
\usepackage{url}
\usepackage{float}
\usepackage{tikz}
\usepackage{caption}
\usepackage{subcaption}
\usepackage[export]{adjustbox}
\usepackage[all]{xy}
\usepackage{faktor}
\usepackage{geometry}
\author{Arielle MARC-ZWECKER}
\title{Horn's problem in PU(n,1).}
\begin{document}
	\maketitle
	
	\theoremstyle{plain}
	\newtheorem{Theo}{Theorem}[section]
	\newtheorem*{Theo*}{Theorem}
	\newtheorem{Prop}[Theo]{Proposition}
	\newtheorem*{Prop*}{Proposition}
	\newtheorem{Prob}[Theo]{Problem}
	\newtheorem*{Prob*}{Problem}
	\newtheorem{Lemm}[Theo]{Lemma}  
	\newtheorem{Coro}[Theo]{Corollary}
	\newtheorem{Propr}[Theo]{Propri\'et\'e}
	\newtheorem{Conj}[Theo]{Conjecture}
	\newtheorem*{Conj*}{Conjecture}
	\newtheorem{Aff}[Theo]{Affirmation}
	\newtheorem{thm}{Theorem}[section]
	\renewcommand{\thethm}{\empty{}} 
	\newtheorem{Def}{Definition}[section]
	\renewcommand{\theDef}{\empty{}}
	\newtheorem{conj}{Conjecture}[section]
	\renewcommand{\theconj}{\empty{}}
	
	\theoremstyle{definition}
	\newtheorem{Defi}[Theo]{Definition}
	\newtheorem*{Defi*}{Definition}
	\newtheorem{Exem}[Theo]{Example}
	\newtheorem{Nota}[Theo]{Notation}
	
	\theoremstyle{remark}
	\newtheorem{Rema}[Theo]{Remark}
	\newtheorem{NB}[Theo]{N.B.}
	\newtheorem{Comm}[Theo]{Commentaire}
	\newtheorem{question}[Theo]{Question}
	\newtheorem{exer}[Theo]{Exercice}
	
	\newcommand{\pu}{\operatorname{PU}}
	\newcommand{\su}{\operatorname{SU}}
	\newcommand{\U}{\operatorname{U}}

	\def\emptyset{\varnothing}
	
	\def\NN{{\mathbb N}}    %naturels
	\def\ZZ{{\mathbb Z}}     %entiers relatifs
	\def\RR{{\mathbb R}}    %réels
	\def\QQ{{\mathbb Q}}    %réels
	\def\CC{{\mathbb C}}    %complexes
	\def\HH{{\mathbb H}}    %quaternions / espace hyperbolique
	\def\AA{{\mathbb P}}     %espace projectif
	\def\KK{{\mathbb K}}     %corps quelconques
	\def\SU{{\operatorname{SU}(2,1)}}
	
	%%%%%%raccourcis lettres calligraphiées
	\def\cA{{\mathcal A}}  \def\cG{{\mathcal G}} \def\cM{{\mathcal M}} \def\cS{{\mathcal S}} \def\cB{{\mathcal B}}  \def\cH{{\mathcal H}} \def\cN{{\mathcal N}} \def\cT{{\mathcal T}} \def\cC{{\mathcal C}}  \def\cI{{\mathcal I}} \def\cO{{\mathcal O}} \def\cU{{\mathcal U}} \def\cD{{\mathcal D}}  \def\cJ{{\mathcal J}} \def\cP{{\mathcal P}} \def\cV{{\mathcal V}} \def\cE{{\mathcal E}}  \def\cK{{\mathcal K}} \def\cQ{{\mathcal Q}} \def\cW{{\mathcal W}} \def\cF{{\mathcal F}}  \def\cL{{\mathcal L}} \def\cR{{\mathcal R}} \def\cX{{\mathcal X}} \def\cY{{\mathcal Y}}  \def\cZ{{\mathcal Z}}
	
	%%%%%%raccourcis lettres gothiques
	
	\def\mfA{{\mathfrak A}} \def\mfA{{\mathfrak P}} \def\mfS{{\mathfrak S}}\def\mfZ{{\mathfrak Z}} \def\mfM{{\mathfrak M}} \def\mfQ{{\mathfrak Q}} \def\mfE{{\mathfrak E}} \def\mfL{{\mathfrak L}} \def\mfW{{\mathfrak W}} \def\mfR{{\mathfrak R}} \def\mfK{{\mathfrak K}} \def\mfX{{\mathfrak X}} \def\mfT{{\mathfrak T}} \def\mfJ{{\mathfrak J}} \def\mfC{{\mathfrak C}} \def\mfY{{\mathfrak Y}} \def\mfH{{\mathfrak H}} \def\mfV{{\mathfrak V}}\def\mfU{{\mathfrak U}}\def\mfG{{\mathfrak G}} \def\mfB{{\mathfrak B}} \def\mfI{{\mathfrak I}} \def\mfF{{\mathfrak F}} \def\mfN{{\mathfrak N}} \def\mfO{{\mathfrak O}} \def\mfD{{\mathfrak D}} 
	
	\def\mfa{{\mathfrak a}} \def\mfp{{\mathfrak p}} \def\mfs{{\mathfrak s}}  \def\mfz{{\mathfrak z}} \def\mfm{{\mathfrak m}} \def\mfq{{\mathfrak q}}  \def\mfe{{\mathfrak e}} \def\mfl{{\mathfrak l}} \def\mfw{{\mathfrak w}} \def\mfr{{\mathfrak r}} \def\mfk{{\mathfrak k}} \def\mfx{{\mathfrak x}} \def\mft{{\mathfrak t}} \def\mfj{{\mathfrak j}} \def\mfc{{\mathfrak c}} \def\mfy{{\mathfrak y}} \def\mfh{{\mathfrak h}} \def\mfv{{\mathfrak v}} \def\mfu{{\mathfrak u}} \def\mfg{{\mathfrak g}} \def\mfb{{\mathfrak b}} \def\mfi{{\mathfrak i}} \def\mff{{\mathfrak f}} \def\mfn{{\mathfrak n}} \def\mfo{{\mathfrak o}} \def\mfd{{\mathfrak d}} 
	
\begin{abstract}
	The multiplicative Horn problem is the following question: given three conjugacy classes $\cC_1,\cC_2,\cC_3$ in a Lie group $G$, do there exist elements $(A,B,C)\in\cC_1\times\cC_2\times\cC_3$ such that $ABC=\operatorname{Id}$? 
	In this paper, we study the multiplicative Horn problem restricted to the elliptic classes of the group $G=\pu(n,1)$ for $n\geq 1$, which is the isometry group of the $n$-dimensional complex hyperbolic space. We show that the solution set of Horn's problem in $\pu(n,1)$ is a finite union of convex polytopes in the space of elliptic conjugacy classes. We give a complete description of these polytopes when $n=2$. 
\end{abstract}
	
\section{Introduction}

In this work, we are interested with the multiplicative Horn problem, also known as the Deligne-Simpson problem. Historically, this question first appeared under its additive form.

\begin{Prob*}
	\emph{(Horn, additive version)} Let $\mfg$ be a Lie algebra. For which triples of conjugacy classes $(\cC_{1},\cC_2,\cC_3)$ in $\mfg$ can we find three elements $(A,B,C)\in\cC_1\times\cC_2\times\cC_3$ such that $A+B+C=0$ ?
\end{Prob*}

We will say that a triple of classes $(\cC_1,\cC_2,\cC_3)$ satisfying Horn's condition is a \textit{solution} to Horn's problem. For reasons related to physics, Horn focused on Hermitian matrices (the vector space of size $n$ Hermitian matrices $\operatorname{Herm}(n)$ does not form a Lie algebra, but $\operatorname{Herm}(n)=i\mfs\mfu(n)$). He obtained in \cite{Hor} (1962) necessary conditions that should be satisfied by the solutions in $\operatorname{Herm}(n)$; these conditions materialize as a system of inequalities. He also conjectured in this paper that the set of solutions was a convex in a hyperplane of the space of triples of classes, for which he gave an explicit formula. This conjecture was later solved in various ways. We can cite the works of Klyachko \cite{Kly}, Knutson and Tao \cite{KnuTao} published at the end of the 1990's; other methods have been used a little more recently by Belkale \cite{Bel}, Ressayre \cite{Res}. This list of papers is far from being exhaustive, and we advise the interested reader to look at the surveys by Fulton \cite{Ful} or Brion \cite{Bri} which give a precise overview of this subject.

Let us now state the problem under its multiplicative form, which is the main interest of this paper.

\begin{Prob*}
	\emph{(Horn, multiplicative version, or Deligne-Simpson)} Let $G$ be a Lie group. For which triples of conjugacy classes $(\cC_{1},\cC_2,\cC_3)$ in $G$ can we find three elements $(A,B,C)\in\cC_1\times\cC_2\times\cC_3$ such that $ABC=\operatorname{Id}$ ?
\end{Prob*}

Deligne first asked this question in the case where $G=\operatorname{GL}_n(\CC)$ or $\operatorname{SL}_n(\CC)$ in the 1980's, and Simpson was the first to bring elements of answer in \cite{Sim}. In particular, he showed for $\operatorname{SL}_3(\CC)$ that all triples of semi-simple with distinct eigenvalues conjugacy classes were solutions. See Kostov's survey \cite{Kos} for a description of these results. The case of the unitary groups $G=\U(n)$ or $\su(n)$, and more generally that of compact and connected Lie groups, has been solved: see Biswas \cite{Bis}, Belkale \cite{Bel2}, Belkale-Kumar \cite{BelKum}, Ressayre \cite{Res2}. These authors obtained a minimal list of inequalities that should be satisfied by the solutions. By opposition to $\operatorname{SL}_3(\CC)$, in the unitary case there exist triples of classes that are not solutions.

In this paper, we are interested in the group $G=\pu(n,1)$, which is non-compact of rank one. Its key feature, which will enable us to tackle the question from a geometrical point of view, is that it is the isometry group of the $n$-dimensional complex hyperbolic space, denoted $\mathbb{H}_{\CC}^n$. This space, which is the natural generalization to the complex setting of the real hyperbolic space, was already present in Poincaré's work \cite{Poi} from 1907. It was also studied by Cartan \cite{Car} in the 1930's. The complex hyperbolic space is endowed with a very rich geometry, in particular because of its non-constant negative curvature. The main reference for this field of geometry is Goldman's book \cite{Gol}. 

Let us come back to Horn's problem in $\pu(n,1)$. Note that the existence of a solution $(\cC_1,\cC_2,\cC_3)$ is equivalent to the existence of a representation $\rho$ of the group $$\Gamma=\langle \texttt{a},\texttt{b},\texttt{c} \ | \ \texttt{a}\texttt{b}\texttt{c}=1\rangle$$ into $\pu(n,1)$, with $(\rho(\texttt{a}),\rho(\texttt{b}),\rho(\texttt{c}))\in\cC_1\times\cC_2\times\cC_3$ (see Section \ref{cadre} for more detail). This "representation theory" point of view reveals to be an interesting one to tackle the question.

Restricting to the $n=2$ case, Falbel and Wentworth proved in \cite{FalWen} that all triples of loxodromic classes were solutions (a loxodromic element is semi-simple with distinct eigenvalues, and geometrically, fixes two points in $\partial\mathbb{H}_{\CC}^n$ and no point in $\mathbb{H}_{\CC}^n$). In \cite{ParWil2}, Parker and Will gave elements of response again for $n=2$, but considering this time triples of parabolic classes (a parabolic element is not semi-simple, and acts geometrically by fixing a unique point in $\partial\mathbb{H}_{\CC}^n$ and no point in $\mathbb{H}_{\CC}^n$). 

In this manuscript, we will rather restrict to triples of elliptic conjugacy classes, i.e.\ classes whose elements fix a point in $\mathbb{H}_{\CC}^n$ (their lifts to $\U(n,1)$ are diagonalizable with modulus one eigenvalues). We may identify the space of such classes with a polyhedral subspace $\cT(G)\subset\RR^n$ (see Section \ref{poly}). We show the following result (Theorem \ref{generalsolution}):

\begin{Theo*}
	The solution set of the elliptic multiplicative Horn problem in $G=\pu(n,1)$ is a finite union of polytopes in $\cT(G)^3\subset(\RR^n)^3$.
\end{Theo*}

The faces of this polytope are pieces of hyperplanes composed of \textit{reducible} solutions $(\cC_{1},\cC_2,\cC_3)$, i.e.\ such that there exists a reducible representation $\rho$ of $\Gamma$ (this means that lifts of $\rho(\texttt{a}),\rho(\texttt{b}),\rho(\texttt{c})$ to $\U(n,1)$ have a common invariant subspace) such that $(\rho(\texttt{a}),\rho(\texttt{b}),\rho(\texttt{c}))$ belongs to $\cC_1\times\cC_2\times\cC_3$. We call these faces the \textit{reducible walls}. The following subtle fact will be of interest for us: there may exist an irreducible representation $\rho'$ such that the conjugacy classes of $\rho'(\texttt{a}),\rho'(\texttt{b}),\rho'(\texttt{c})$ lie on a reducible wall.

\begin{Defi*}
	We call \textit{cell} a connected component of the complement of the set of reducible walls.
\end{Defi*}

Let $C$ be a cell. The proof of the Theorem relies on the following dichotomy (Theorem \ref{emptyfulltheorem}), which is identified by Paupert in \cite{Pau} (see also Paupert-Will \cite{PauWil}):
\begin{itemize}
	\item either $C$ is contained in the solution set (in which case we say that $C$ is "full");
	\item or $C$ is contained in the complement of the solution set (in which case we say that $C$ is "empty").
\end{itemize}
This dichotomy comes from the combination of the following results (Propositions \ref{open} and \ref{closed}):
\begin{Prop*}
	The solution set of the Horn problem satisfies the following two properties:
	\begin{enumerate}[i)]
		\item The subset of irreducible solutions is open in the solution set;
		\item The solution set is closed in the interior of the space of triples of conjugacy classes $\mathring{\mathcal{T}(G)}^3$.
	\end{enumerate}
\end{Prop*}
The openness property is relatively classical for a Lie group with trivial center. However, the closedness property comes specifically from the $\delta$-hyperbolicity in Gromov's sense of $\mathbb{H}_{\CC}^n$, which enables us to use a result obtained separately by Bestvina (Theorem 3.9 in \cite{Bes}) and Paulin (Theorem 6.6 in \cite{Paulin}) in order to overcome the non-compactness of the group $\pu(n,1)$ (see Falbel-Wentworth \cite{FalWen}).

This leads to the following strategy to obtain an explicit solution to Horn's problem:
\begin{enumerate}[1.]
	\item Describe the set of reducible walls by giving equations;
	\item Describe explicitly all the cells by a set of inequalities;
	\item Determine for each cell if it is "full" or "empty".
\end{enumerate}

A triple $(A,B,C)\in\pu(n,1)^3$ is reducible if there exist $d_1,d_2\in\NN$ with $d_1+d_2=n$ such that $A,B,C$ belong to a same subgroup of $\pu(n,1)$ isomorphic to $\mathbb{P}(\U(d_1)\times\U(d_2,1))\cong\U(d_1)\times\pu(d_2,1)$. Point 1.\ therefore requires to know the explicit solutions to Horn's problem for all the groups $\U(d_1)$ with $d_1\leq n$, and $\pu(d_2,1)$ with $d_2<n$. 

Point 2.\ requires to understand the combinatorial arrangements of the reducible walls, whose complexity grows exponentially fast with the value of $n$: we will see for instance that we obtain two reducible walls and four cells for $\pu(1,1)$, against twenty-seven reducible walls and twenty-eight cells for $\pu(2,1)$. These first two points are therefore hard to obtain for technical and computational reasons. We will see in Section \ref{liftingsection} how a lifting property of representations to $\su(n,1)$ enables us to diminish the complexity of these combinatorics. It relies on the following observation: if $ABC=\operatorname{Id}$ in $\pu(n,1)$, then lifts $\tilde{A},\tilde{B},\tilde{C}\in\su(n,1)$ satisfy $\tilde{A}\tilde{B}\tilde{C}=u\operatorname{Id}$, where $u$ is an $(n+1)$-root of unity. Fixing a choice of lift for $A,B,C$ (the so-called \textit{standard lift}, see Definition \ref{standardlift}) yields a partition of the solution set according to the value of $u$.

For point 3., it suffices to find in each cell a triple of classes for which we can determine if it is a solution or not. According to the dichotomy stated above, this enables us to conclude that the entire cell is either "full" or "empty". However, this point remains the most difficult to obtain, because there exists no general method to construct examples of solutions. We will use the results of \cite{Mar}, where we have explained how to construct families of irreducible representations of $\Gamma$ such that the images of the generators belong to the same triples of classes as reducible representations of $\Gamma$. These constructions will be particularly useful to carry out point 3., for the presence of an irreducible solution on a reducible wall enables us to "fill" the cells adjacent to this wall.

In Sections \ref{horn} and \ref{proof}, we apply the resolution strategy described above to the case $n=2$, and we show the following result (Theorem \ref{solution}):

\begin{Theo*}
	The solution set to the elliptic multiplicative Horn problem in $\pu(2,1)$ is the reunion of five explicit convex polytopes in $(\RR^2)^3$.
\end{Theo*}

\begin{figure}
	\centering
	\scalebox{0.48}{\includegraphics{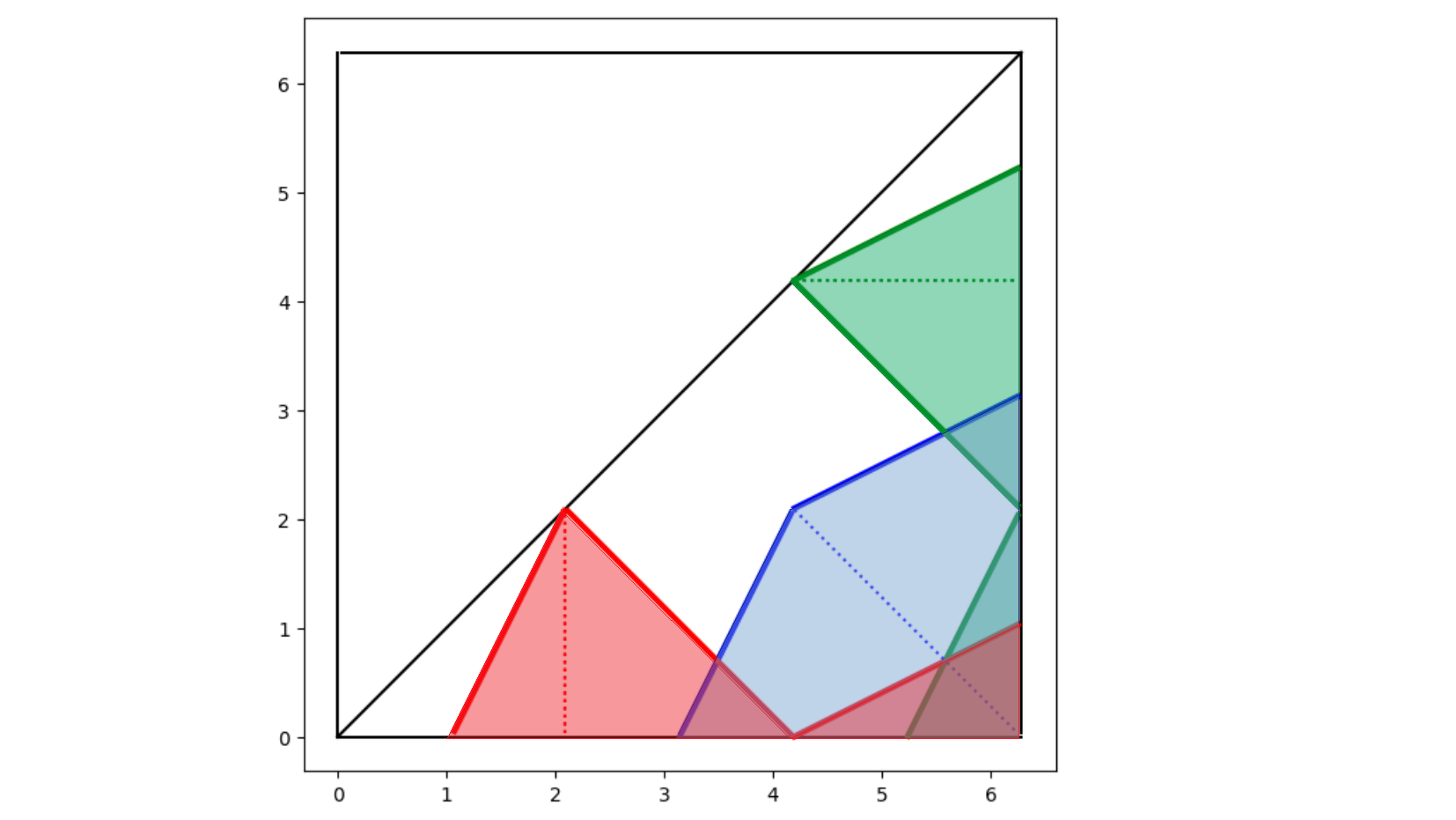}}
	\caption{Case of $\pu(2,1)$: the slice $\cC_1=\cC_2=\cC_3$. \label{tranche_sym_intro}}
\end{figure}

Let us give an illustration of this result. In the case where $G=\pu(2,1)$, the space of regular elliptic conjugacy classes is identified with the triangle $$\cT(G)=\{(\alpha_1,\alpha_2)\in[0,2\pi[^2 \ | \ \alpha_1\geq \alpha_2\}$$ (see Section \ref{poly}). Taking $\cC_1=\cC_2=\cC_3$ yields a two-dimensional slice of $\cT(G)^3$ identified with $\cT(G)$. We have depicted on Figure \ref{tranche_sym_intro} the intersection of this slice with the polytopes describing the solution set. The five polytopes that were announced in the above Theorem appear here (there are two red ones, two green ones, and one blue one). The three colors correspond to the three roots of unity. More precisely, the blue polytope corresponds to the solutions where the standard lifts (Definition \ref{standardlift}) satisfy $\tilde{A}\tilde{B}\tilde{C}=\operatorname{Id}$, the red ones correspond to $\tilde{A}\tilde{B}\tilde{C}=e^{2i\pi/3}\operatorname{Id}$, and the green ones to $\tilde{A}\tilde{B}\tilde{C}=e^{-2i\pi/3}\operatorname{Id}$. See Section \ref{figures}, where we give a precise description of this Figure.

Let us conclude this Introduction by a quick comparison of our results with Paupert's work in \cite{Pau}. In this paper, he gives fundamental arguments toward a resolution of the elliptic multiplicative Horn problem in $\pu(2,1)$, but rather by a case-by-case analysis. More precisely, he fixes two elliptic conjugacy classes $\cC_1,\cC_2$ in $\pu(2,1)$, and studies the image of the momentum map
$$\begin{array}{ccccc}
\mu_{\cC_1,\cC_2} & : & \cC_1\times \cC_2 & \to & \operatorname{Conj}(G) \\
& & (A,B) & \mapsto & [AB] \\
\end{array}$$
where $[AB]$ designates the conjugacy class of the element $AB$. Comparing with the point of view adopted here, the action of fixing two of the classes amounts to considering a two-dimensional slice of $\cT(G)^3\subset(\RR^2)^3$ (see Section \ref{figures}). The notions of reducible walls and of "full" or "empty" cells are introduced in his work. However, the action of taking slices breaks the symmetry, and makes more difficult the task of describing the solution set by a system of inequalities. One of his main results (Theorem 1.2 in \cite{Pau}) is a necessary and sufficient condition on the classes $\cC_1,\cC_2$ to obtain the surjectivity of $\mu_{\cC_1,\cC_2}$. We will be able to deduce this condition as a consequence of our description of the solution set (Remark \ref{paupert}).

The paper is organized as follows. Section \ref{preliminaries} is dedicated to preliminaries in the field of complex hyperbolic geometry: we introduce there all the notions that will be used later on. Section \ref{horn_qual} is dedicated to the elliptic multiplicative Horn problem in $\pu(n,1)$, where $n$ is any integer. We build there the setting in which we prove our results, which is that of representation theory. We give a detailed description of the identification between the space of elliptic conjugacy classes and the polyhedral subspace $\cT(G)\subset\RR^n$, which enables us to see the solution set as a reunion of polytopes. We define the notion of "full" or "empty" cell. In Section \ref{horn_u}, we give the solutions to Horn's problem for the groups $\U(2)$ and $\pu(1,1)$. Sections \ref{horn} and \ref{proof} are dedicated to the proof of the explicit solution to the elliptic multiplicative Horn problem in $\pu(2,1)$. We describe the reducible walls together with the cells, and we give arguments to "fill" or "empty" every cell. In Section \ref{examples}, we give examples of two-dimensional slices of our polytopes, like Figure \ref{tranche_sym_intro} or Paupert's slices, in order to enable the reader to visualize the solution set.

\textbf{Acknowledgments.} The results of this paper are part of a PhD thesis under the supervision of Antonin Guilloux and Pierre Will. I am most grateful to my advisors for the numerous discussions we had together, as well as for their useful advice and comments on this piece of writing. I also thank Michel Brion, Elisha Falbel, Arnaud Maret, John Parker, Julien Paupert, Nicolas Ressayre and Maxime Wolff for the useful discussions. This work is supported by the French National Research Agency in the framework of the « France 2030 » program  (ANR-15-IDEX-0002) and by the LabEx PERSYVAL-Lab (ANR-11-LABX-0025-01).

\section{The complex hyperbolic space}\label{preliminaries}

\subsection{Description of $\mathbb{H}_{\CC}^n$}\label{gene}

We start by reviewing classical notions about the complex hyperbolic space $\mathbb{H}_\CC^n$, highlighting the objects we will be studying next. General references are Goldman's book \cite{Gol} or Parker's notes \cite{Par}. 

\subsubsection{Main definitions}

We consider a Hermitian form $\langle\cdot,\cdot\rangle$ of signature $(n,1)$ over $\CC^{n+1}$. The negative, null and positive cones of $\langle\cdot,\cdot\rangle$ are
$$V^-=\{z\in\CC^{n+1} \ | \ \langle z,z\rangle<0\}, \ V^0=\{z\in\CC^{n+1} \ | \ \langle z,z\rangle=0\},$$ $$V^+=\{z\in\CC^{n+1} \ | \ \langle z,z\rangle>0\}.$$

\begin{Defi}\label{type}
	We say that a vector of $V^-$ (respectively $V^+$, respectively $V^0$) is of \textit{negative} (respectively \textit{positive}, respectively \textit{null}) \textit{type}.
\end{Defi}

\begin{Defi}
	The complex hyperbolic space $\mathbb{H}_\CC^n$ is the projectivization of the negative cone $V^-$ to $\CC\mathbb{P}^n$:
	$$\mathbb{H}_\CC^n=\mathbb{P}(V^-)\subset\CC\mathbb{P}^n,$$
	endowed with the metric $d$ defined as follows:
	\begin{equation*}
	\forall z,w \in \mathbb{H}^n_\CC, \  \cosh^2\left(\frac{d(z,w)}{2}\right)=\frac{\langle\tilde{z},\tilde{w}\rangle\langle\tilde{w},\tilde{z}\rangle}{\langle\tilde{z},\tilde{z}\rangle\langle\tilde{w},\tilde{w}\rangle},
	\end{equation*}
	where $\tilde{z},\tilde{w}\in\CC^{n+1}$ are any lifts of $z,w$ (see Section 3.1.6 in \cite{Gol}). Its boundary $\partial\mathbb{H}_\CC^n$ is the projectivization of the null cone $V^0$:
	$$\partial\mathbb{H}_\CC^n=\mathbb{P}(V^0)\subset\CC\mathbb{P}^n.$$
	We denote by $\overline{\mathbb{H}_\CC^n}$ the union $\mathbb{H}_\CC^n\cup\partial\mathbb{H}_\CC^n$.
\end{Defi}

Let us consider the Hermitian form of signature $(n,1)$ defined for any $z,w$ in $\CC^{n+1}$ by
$$\langle z,w\rangle=z_1\overline{w_1}+\ldots+z_n\overline{w_n}-z_{n+1}\overline{w_{n+1}}.$$ 
Its matrix in the canonical basis is
\begin{equation}\label{J}
J=\begin{pmatrix}
I_n & 0 \\
0 & -1 
\end{pmatrix}.
\end{equation}
The space obtained with the choice of this Hermitian form is called the \textit{ball model} for $\mathbb{H}_\CC^n$. Indeed, the affine chart $\{z_{n+1}=1\}$ identifies $\mathbb{H}_\CC^n$ with the unit ball of $\CC^n$, and $\partial\mathbb{H}_\CC^n$ with the sphere $\mathbb{S}^{2n-1}$. 

\begin{Nota}
	For $z=(z_1,\ldots,z_n)$ in $\mathbb{H}_\CC^n$, we shall denote by $\tilde{z}=(z_1,\ldots,z_n,1)\in\CC^{n+1}$ its lift to this affine chart.
\end{Nota}

Let us recall that $\operatorname{U}(n,1)$ is the subgroup of $\operatorname{GL}_{n+1}(\CC)$ that preserves the Hermitian form $\langle\cdot,\cdot\rangle$. We denote by $\operatorname{PU}(n,1)$ the quotient of this group by its center (which is the subset of homotheties), and by $\mathbb{P}:\U(n,1)\to\pu(n,1)$ the projection. One can easily check that the action of $\pu(n,1)$ on $\CC\mathbb{P}^n$ restricts to $\mathbb{H}_\CC^n$. The following classical theorem describes the main properties of this action.

\begin{Theo} \label{isomclassif}
	\emph{\cite{Par}}  The group of holomorphic isometries of $\mathbb{H}_\CC^n$ is $\operatorname{PU}(n,1)$. Its action on $\mathbb{H}_\CC^n$ is transitive, and stabilizers of points are isomorphic to the subgroup $\mathbb{P}(\operatorname{U}(n)\times\operatorname{U}(1))$.
\end{Theo}

\subsubsection{Complex lines in the complex hyperbolic plane}

In this section, we fix $n=2$. Let us develop some notions and facts about complex lines of $\mathbb{H}_\CC^2$ that we will use later on. For more details, see Section 5.2 in \cite{Par}.

\begin{Defi}\label{comprefl}
	Let $W$ be a complex plane in $\CC^3$ that intersects $V^-$. We call \textit{complex line} the intersection $\mathbb{P}(W)\cap\mathbb{H}_\CC^2$. 
	
	The  one-dimensional subspace $W^\perp\subset\CC^3$ is called the polar line to $W$, consisting of \textit{polar vectors}. These vectors have positive type (see Definition \ref{type}). We say that such a polar vector $c$ is \textit{normalized} if $\langle c,c\rangle=1$. 
\end{Defi}

\begin{Rema}
	\begin{itemize}
		\item A projective line in $\CC\mathbb{P}^2$ may be disjoint from $\overline{\mathbb{H}_\CC^2}$ (respectively, tangent to $\mathbb{H}_\CC^2$). It also has a line of polar vectors, which have negative (respectively null) type. In the negative case, we say that a polar vector $c$ is normalized if $\langle c,c\rangle=-1$.
		\item Given two points $p_1,p_2\in\mathbb{H}_\CC^2$, they are contained in a unique complex line, which we denote by $(p_1p_2)$.
	\end{itemize}
\end{Rema}

\begin{Prop}\label{CP1stab}
	\emph{\cite{Par}} The group $\operatorname{PU}(2,1)$ acts transitively on the set of complex lines. The stabilizer of a complex line under the action of $\operatorname{PU}(2,1)$ is conjugate to $\mathbb{P}(\operatorname{U}(1,1)\times\operatorname{U}(1))$ (compare with Theorem \ref{isomclassif}).
\end{Prop}

A complex line is supported by a projective line in the projective plane. We therefore have three possibilities for the relative position of two distinct complex lines:
	\begin{itemize}
		\item If they intersect in a point in $\mathbb{H}_\CC^2$, we say that they are \textit{intersecting};
		\item If their supporting projective lines intersect in a point outside $\overline{\mathbb{H}_\CC^2}$, and therefore are disjoint in $\overline{\mathbb{H}_\CC^2}$, we call them \textit{ultra-parallel};
		\item If their closures intersect in a point in $\partial\mathbb{H}_\CC^2$, we call them \textit{asymptotic}.
	\end{itemize}

\subsubsection{$\delta$-hyperbolicity}\label{deltahyp}

In the 1980's, Gromov generalized in \cite{Gro} the hyperbolic space $\mathbb{H}_\RR^n$ in the following way. Let $\delta\geq0$ and $(X,d)$ be a metric space. We say that $X$ is \textit{$\delta$-hyperbolic} if all its geodesic triangles are $\delta$-slim (i.e.\ if $T$ is a triangle, then each side of $T$ is contained in a $\delta$-neighborhood of the union of the two other sides). A general reference for the theory of $\delta$-hyperbolic spaces is Bridson and Haefliger's book \cite{BriHae}. In particular, any metric space with negative curvature is $\delta$-hyperbolic for some $\delta$ (see for instance Chapter III, Proposition 1.2 in \cite{BriHae}). This ensures that the complex hyperbolic space $\mathbb{H}_\CC^n$ is $\delta$-hyperbolic. 

This fact reveals to be very useful to overcome the non-compactness of $\pu(n,1)$ (see the Introduction). In particular, we can use the following result, due to Bestvina (Theorem 3.9 in \cite{Bes}) and Paulin (Theorem 6.6 in \cite{Paulin}). We formulate it here in the same way as Falbel and Wentworth (see Proposition 2 in \cite{FalWen}). If $g\in\operatorname{Isom}(X)$, we denote by
$$|g|_X=\inf_{x\in X}d(x,gx)$$
the \textit{translation length} of $g$.

\begin{Theo}\label{falbwent}
	Let $X$ be a $\delta$-hyperbolic metric space, and let $(g_j)_j,(h_j)_j$ be a pair of sequences of semi-simple isometries of $X$. Assume that there exists $C>0$ such that $|g_j|_X<C$ and $|h_j|_X<C$ for all $j$. Then one of the following must hold:
	\begin{enumerate}[i)]
		\item There exists a sub-sequence $(j_k)_k$ and elements $f_k\in G$ such that the two sequences $(f_kg_{j_k}f_k^{-1})_k$ and $(f_kh_{j_k}f_k^{-1})_k$ converge in $G$.
		\item The sequence of translation lengths $(|g_jh_j|_X)_j$ is unbounded.
	\end{enumerate}
\end{Theo} 

In Section \ref{horn_qual}, this result will be a key argument. We will use it through the following Corollary:

\begin{Coro}\label{bestpau}
	Let $(A_j)_j,(B_j)_j$ be a pair of sequences of elliptic elements in $\pu(n,1)$, such that $A_jB_j$ is elliptic for every $j$ (see Section \ref{conjell} below). Then up to extracting a sub-sequence and global conjugation, there exist $A,B\in\pu(n,1)$ such that $(A_j,B_j)_j$ converges to $(A,B)$.
\end{Coro}

\begin{proof}
	Since $A_jB_j$ is elliptic for every $j$, the sequence of corresponding translation lengths is constant equal to zero; in particular, it is bounded. Therefore, Item i) of Theorem \ref{falbwent} holds, which yields the result.
\end{proof}

\subsection{Elliptic isometries of $\mathbb{H}_\CC^n$}

In this section, we give a detailed description of elliptic elements and elliptic conjugacy classes. This is the setting in which we will solve the elliptic Horn problem.

\subsubsection{Conjugacy classes of elliptic elements}\label{conjell}

Just like in real hyperbolic geometry, any isometry of $\mathbb{H}_\CC^n$ belongs to one of the three (mutually exclusive) following classes: it is \textit{loxodromic} if it fixes exactly two points in $\partial\mathbb{H}_\CC^n$ and none in $\mathbb{H}_\CC^n$, \textit{parabolic} if it fixes exactly one point in $\partial\mathbb{H}_\CC^n$ and none in $\mathbb{H}_\CC^n$, and \textit{elliptic} if it fixes at least one point in $\mathbb{H}_\CC^n$.

One can refine the elliptic case by looking at the possibility of a repeated eigenvalue.

\begin{Defi}\label{regspe}
	Whenever an isometry $g\in\pu(n,1)$ has a diagonalizable lift with pairwise distinct eigenvalues, we call it \textit{regular}; otherwise we call it \textit{special} (see Section 6.2.1 in \cite{Gol}).
\end{Defi}

\begin{Rema}\label{ellip}
	This notion is well defined because it does not depend on the lift of $g$. Any elliptic element has all its lifts diagonalizable, and admits a lift in $\operatorname{U}(n,1)$ conjugate to a unique matrix of the form 
	$$E(\alpha_1,\ldots,\alpha_n)=\begin{pmatrix}
	e^{i\alpha_1} & 0 & 0 & 0\\
	0 & \ddots & 0 & 0 \\
	0 & 0 & e^{i\alpha_n} & 0 \\
	0 & 0 & 0 & 1 
	\end{pmatrix}, \ \mbox{with} \ 0\leq\alpha_n\leq\ldots\leq\alpha_1<2\pi.$$
	Note that the $n$ positive type eigenvalues $e^{i\alpha_1},\ldots, e^{i\alpha_n}$ (see Definition \ref{type}) may be exchanged by a conjugation in $\operatorname{PU}(n,1)$, which is why we impose an order between the angles to obtain uniqueness. The element $E(\alpha_1,\ldots,\alpha_n)$ is regular elliptic when $0<\alpha_n<\ldots<\alpha_1<2\pi$, and special elliptic otherwise. 
\end{Rema}

\begin{Defi}\label{anglepair}
	When $n=2$, the pair $(\alpha_1,\alpha_2)$ is called the \textit{angle pair} of any element conjugate to $E(\alpha_1,\alpha_2)$.
\end{Defi}

Let $n=2$. By Remark \ref{ellip}, the conjugacy class of an elliptic element in $\operatorname{PU}(2,1)$ is uniquely determined by its angle pair. We now describe the geometric action of an elliptic element on $\mathbb{H}_\CC^2$. First, notice that a regular elliptic element $A$ has three fixed points in $\CC\mathbb{P}^2$: one inside $\mathbb{H}_\CC^2$, and two outside. These last two are projections of vectors polar to the two orthogonal complex lines that are globally preserved by $A$, and on which it acts by rotations (see Remark \ref{ellip}). These two complex lines intersect in the fixed point that is inside $\mathbb{H}_\CC^2$.

\subsubsection{Reducibility for pairs of elliptic elements}

Reducible pairs or triples of elliptic elements will play a very important role in the resolution of Horn's problem.

\begin{Defi}\label{reducibledef}
	Let $A,B\in\pu(n,1)$ be elliptic elements. We say that $(A,B)$ is \textit{irreducible} when any pair of lifts of $(A,B)$ has no non-trivial invariant subspace (invariant subspaces do not depend on choices of lifts); otherwise, $(A,B)$ is \textit{reducible}. 
\end{Defi}

\begin{Rema}\label{center}
	\begin{itemize}
		\item In $\pu(2,1)$, $(A,B)$ is irreducible when the lifts of $A,B$ have pairwise distinct eigenspaces. However, this is too restrictive for $\pu(n,1)$ with $n\geq3$.
		\item This definition is valid for any Lie group $G$. Denoting by $Z(G)$ the center of the group $G$, we have that if the pair $(A,B)$ is irreducible, then $Z(A)\cap Z(B)=Z(G)$.
	\end{itemize}
\end{Rema}

We now specify the three possibilities for an elliptic pair to be reducible when $n=2$, borrowing the terminology from Paupert \cite{Pau}:

\begin{Defi}\label{reducibility}
	A reducible pair $(A,B)\in\pu(2,1)^2$ is called
	\begin{itemize}
		\item \textit{totally reducible} when lifts of $A,B$ are diagonalizable in the same basis (i.e.\ they commute). This means that they have the same fixed point and the same stable complex lines.
		\item \textit{spherical reducible} when lifts of $A,B$ have one common eigenvector of negative type. This means that they have the same fixed point in $\mathbb{H}_\CC^2$.
		\item \textit{hyperbolic reducible} when lifts of $A,B$ have one common eigenvector of positive type. This means that they have one stable complex line in common.
	\end{itemize}
\end{Defi}

\subsubsection{Special elliptics}\label{compref}

In this section, we fix $n=2$. We describe the geometric action of special elliptic elements on $\mathbb{H}_\CC^2$, which will be called complex reflections.

\begin{Defi}\label{reflectionformula}
	(See Pratoussevitch \cite{Pra}) Let $c\in\CC\mathbb{P}^2$ be of non-null type and $\eta\in \mathbf{S}^1$. The \textit{complex reflection} $R$ of \textit{rotation factor} $\eta$ in the projective line polar to $c$ is the projectivization of the map $\tilde{R}$ defined over $\CC^3$ by:
	\begin{equation*} \tilde{R}(\tilde{z})=\tilde{z}+(\eta-1)\frac{\langle \tilde{z},\tilde{c}\rangle}{\langle \tilde{c},\tilde{c}\rangle}\tilde{c}.
	\end{equation*}
\end{Defi}

Such an $R$ is an isometry that fixes the point $c$ and fixes point-wise the projective line $C$ polar to $c$. We divide complex reflections into two geometrical categories:
\begin{itemize}
	\item When $C$ intersects $\mathbb{H}_\CC^2$, we say that $R$ is a \textit{complex reflection in the line} $C$;
	\item When $C$ is disjoint from $\mathbb{H}_\CC^2$, we say that $R$ is a \textit{complex reflection in the point} $c$.
\end{itemize}
In both cases, the line $C$ is called the \textit{mirror} of $R$.

Complex reflections in a line and a point have lifts that are respectively conjugate in $\operatorname{U}(2,1)$ to the matrices
\begin{equation}\label{RS}
\tilde{R}(\eta)=\begin{pmatrix}
\eta & 0 & 0 \\
0 & 1 & 0 \\
0 & 0 & 1 
\end{pmatrix}, \ \tilde{S}(\eta)=\begin{pmatrix}
\eta^{-1} & 0 & 0 \\
0 & \eta^{-1} & 0 \\
0 & 0 & 1
\end{pmatrix}
\end{equation}
where $\eta=e^{i\theta}\neq1$ is the rotation factor of the reflection. We have that:
\begin{itemize}
	\item the action of $R(\eta)$ on $\mathbb{H}_\CC^2$ is given by $(z_1,z_2)\mapsto(\eta z_1,z_2)$;
	\item the action of $S(\eta)$ on $\mathbb{H}_\CC^2$ is given by $(z_1,z_2)\mapsto(\eta^{-1} z_1,\eta^{-1} z_2)$: it fixes the origin and acts by a rotation of angle $-\theta$ on each complex line through this point.
\end{itemize}

\subsubsection{Non-separability: special elliptic and parabolic classes}\label{separabilitysection}

A parabolic element $P\in\pu(2,1)$ has no diagonalizable lift. Two cases may occur (see \cite{Willsurvey} for more detail):
\begin{itemize}
	\item there exists a lift $\tilde{P}\in\U(2,1)$ which is unipotent, in which case we say that $P$ is \textit{unipotent};
	\item when this is not the case, any lift $\tilde{P}\in\U(2,1)$ has exactly two eigenvalues; one of them has null type and has multiplicity $2$, i.e.\ it corresponds to a $2$-dimensional invariant subspace of $\tilde{P}$. The fixed point of $P$ in $\partial\mathbb{H}_\CC^2$ is the projection of the null-type eigenvector. In this case, $P$ is called \textit{ellipto-parabolic} or \textit{screw-parabolic}. 
\end{itemize}

This name comes from the fact that ellipto-parabolic conjugacy classes are non-separate from classes of complex reflections in lines. Indeed, the $2$-dimensional eigenspace of $\tilde{R}(\eta)$ associated to the eigenvalue $1$ (see \eqref{RS}) may be deformed into a $2$-dimensional stable subspace by degenerating the two opposite-type eigenvectors into the same null-type eigenvector. Now, this cannot happen with complex reflections in points, as we show in the following Lemma:

\begin{Lemm}\label{separability}
	Conjugacy classes of complex reflections in points are separate from ellipto-parabolic classes.
\end{Lemm}

\begin{proof}
	Let $\eta\in\mathbf{S}^1,\eta\neq1$ and $(S_n(\eta))_n$ be a sequence of complex reflections in points (see \eqref{RS}). By means of a contradiction, assume that $(S_n(\eta))_n$ converges to an ellipto-parabolic element $P$. Let $(u_n),(v_n),(w_n)$ be the three sequences of unitary eigenvectors of $\tilde{S}_n(\eta)$, such that:
	\begin{itemize}
		\item $(u_n)$ is the sequence of unitary negative-type eigenvectors associated to the eigenvalue $1$ of $\tilde{S}_n(\eta)$;
		\item $(v_n)$ and $(w_n)$ are the two sequences of unitary positive-type eigenvectors associated to the eigenvalue $\eta^{-1}$ of $\tilde{S}_n(\eta)$.
	\end{itemize}
	Note that:
	\begin{itemize}
		\item for every $n$, $\operatorname{Vect}(u_n)=\operatorname{Vect}(v_n,w_n)^{\perp}$;
		\item up to extracting, the three sequences converge to eigenvectors of a lift $\tilde{P}$.
	\end{itemize}
	As stated above, $\tilde{P}$ is not diagonalizable, and has an eigenvalue of multiplicity $2$ for a unique unitary eigenvector $v_0$ of null type. Therefore, since $(v_n)$ and $(w_n)$ are eigenvectors for the eigenvalue of multiplicity $2$ of $\tilde{S}_n(\eta)$, both $(v_n)$ and $(w_n)$ converge to $v_0$. Now, by orthogonality, $(u_n)$ converges to a unitary vector in $\operatorname{Vect}(v_0)^{\perp}$, which has negative or null type: since $\operatorname{Vect}(v_0)^{\perp}$ contains only vectors of positive or null type, this imposes that $(u_n)$ converges to $v_0$ as well. Therefore, $v_0$ is an eigenvector of $\tilde{P}$ for the two eigenvalues $\eta^{-1}$ and $1$: this is impossible. So $(S_n(\eta))_n$ cannot converge to a parabolic element.
\end{proof}

\section{The elliptic multiplicative Horn problem in PU(n,1): qualitative aspects}\label{horn_qual}

\subsection{Introduction}

\subsubsection{Presentation of Horn's problem}\label{cadre}

Let $G$ be a Lie group. Recall that the multiplicative Horn problem in $G$ is the following question: 

\begin{Prob}
	\emph{(Horn)} For which conjugacy classes $\cC_1,\cC_2,\cC_3$ in $G$ can we find three elements $A\in\cC_1,B\in\cC_2,C\in\cC_3$ such that $ABC=\operatorname{Id}$?
\end{Prob}

The goal of this section is to give a qualitative answer to this question for the specific case where $G=\operatorname{PU}(n,1)$, and $\cC_1,\cC_2,\cC_3$ are elliptic conjugacy classes. In Sections \ref{horn}, \ref{examples} and \ref{proof}, we give an explicit description of the solution set for $n=2$.

Horn's problem may be rephrased in terms of the character variety of the group $$\Gamma=\langle \texttt{a},\texttt{b},\texttt{c} \ | \ \texttt{a}\texttt{b}\texttt{c}=1\rangle,$$
which is the fundamental group of the three-punctured sphere.

Consider the set $\operatorname{Hom}(\Gamma,G)$ of representations of $\Gamma$ in $G$. The associated character variety is the algebraic quotient
$$\chi(\Gamma,G)=\operatorname{Hom}(\Gamma,G)//G,$$
where the quotient is with respect to the action by conjugation of $G$ on $\operatorname{Hom}(\Gamma,G)$ (see \cite{LubMag}). Let $\pi:\operatorname{Hom}(\Gamma,G)\to\chi(\Gamma,G)$ be the canonical projection. 

\begin{Defi}
	We define:
	\begin{itemize}
		\item $\operatorname{Hom}_\mathrm{ell}(\Gamma,G)=\{\rho\in\operatorname{Hom}(\Gamma,G) \ | \ \rho(\texttt{a}),\rho(\texttt{b}),\rho(\texttt{c}) \ \mbox{are elliptic} \}$ as the subset of boundary elliptic representations. This set is invariant by conjugation;
		\item $\chi_\mathrm{ell}(\Gamma,G)=\pi(\operatorname{Hom}_\mathrm{ell}(\Gamma,G))\subset \chi(\Gamma,G)$;
		\item $\operatorname{Conj}_\mathrm{ell}(G)$ as the set of elliptic conjugacy classes in $G$.
	\end{itemize}
\end{Defi}

We have the following commutative diagram:
$$  \xymatrix {
	\operatorname{Hom}_\mathrm{ell}(\Gamma,G) \ar[rr]^c \ar[rd]_\pi  && \operatorname{Conj}_\mathrm{ell}(G)^3  \\
	& \chi_\mathrm{ell}(\Gamma,G) \ar[ru]_{c_{\chi}}  }$$
where:
\begin{itemize}
	\item the map $c$ sends $\rho\in\operatorname{Hom}_\mathrm{ell}(\Gamma,G)$ to the triple of conjugacy classes of $(\rho(\texttt{a}),\rho(\texttt{b}),\rho(\texttt{c}))$;
	\item $\pi$ is the canonical projection;
	\item $c_{\chi}$ is the unique map from $\chi_\mathrm{ell}(\Gamma,G)$ to $\operatorname{Conj}_\mathrm{ell}(G)^3$ such that $c_{\chi}\circ\pi=c$.
\end{itemize}

Now, Horn's problem may be expressed in terms of representations of $\Gamma$ in $G$:

\begin{Prob}
	\emph{(Horn)} What is the image of $c$ in $\operatorname{Conj}_\mathrm{ell}(G)^3$ ?
\end{Prob}

\begin{Defi}
	If $(\cC_1,\cC_2,\cC_3)\in\operatorname{Conj}_\mathrm{ell}(G)^3$, we define the associated \textit{relative character variety} $\chi_{\cC_1,\cC_2,\cC_3}(\Gamma,G)$ as the fiber of $c_{\chi}$ above $(\cC_1,\cC_2,\cC_3)$: $$\chi_{\cC_1,\cC_2,\cC_3}(\Gamma,G)=c_{\chi}^{-1}(\cC_1,\cC_2,\cC_3).$$
\end{Defi} 

We may finally express Horn's problem in terms of relative character varieties:

\begin{Prob}
	\emph{(Horn)} For which triple $(\cC_1,\cC_2,\cC_3)\in\operatorname{Conj}_\mathrm{ell}(G)^3$ is the associated relative character variety $\chi_{\cC_1,\cC_2,\cC_3}(\Gamma,G)$ non-empty?
\end{Prob}

\subsubsection{Polyhedral structure of $\operatorname{Conj}_\mathrm{ell}(G)^3$}\label{poly}

In the following sections, we will resolve Horn's problem for some particular groups $G$, all related to $\pu(n,1)$ because they enable us to describe the reducible solutions (see the Introduction). For these examples, we give now an identification between $\operatorname{Conj}_\mathrm{ell}(G)$ and a polyhedral subspace $\cT(G)\subset[0,2\pi]^k$ for some integer $k$.

\begin{Exem}\label{groupexamples}
	\begin{itemize}
		\item If $G=\operatorname{U}(2)$, $\cT(G)=\{(\alpha_1,\alpha_2)\in[0,2\pi[^2 \ | \ \alpha_1\geq \alpha_2\}$: for any $\cC\in\operatorname{Conj}_\mathrm{ell}(G)$, there exists a unique angle pair $(\alpha_1,\alpha_2)\in \cT(G)$ such that any $g\in\cC$ is conjugate to the diagonal matrix
		$$\begin{pmatrix}
		e^{i\alpha_1} & 0 \\
		0 & e^{i\alpha_2}
		\end{pmatrix}.$$
		\item If $G=\pu(1,1)$, $\cT(G)=[0,2\pi[$: for any $\cC\in\operatorname{Conj}_\mathrm{ell}(G)$, there exists a unique $\alpha\in[0,2\pi[$ such that any $g\in\cC$ has a representative conjugate to the diagonal matrix
		$$\begin{pmatrix}
		e^{i\alpha} & 0 \\
		0 & 1
		\end{pmatrix}.$$
		\item More generally, if $G=\pu(n,1)$, $$\cT(G)=\{(\alpha_1,\ldots,\alpha_n)\in[0,2\pi[^n \ | \ \alpha_1\geq\ldots\geq \alpha_n\}.$$ For any $\cC\in\operatorname{Conj}_\mathrm{ell}(G)$, there exists a unique element $(\alpha_1,\ldots,\alpha_n)\in \cT(G)$ such that any $g\in\cC$ has a representative conjugate to the diagonal matrix
		$$E(\alpha_1,\ldots,\alpha_n)=\begin{pmatrix}
		e^{i\alpha_1} & 0 & 0 & 0\\
		0 & \ddots & 0 & 0 \\
		0 & 0 & e^{i\alpha_n} & 0 \\
		0 & 0 & 0 & 1 
		\end{pmatrix}$$
		(see Remark \ref{ellip} in Section \ref{preliminaries}).
	\end{itemize}
\end{Exem}

Denote by $\cH=\operatorname{Im}(c)$ the set of solutions of Horn's problem, and let $i$ be the identification between $\operatorname{Conj}_\mathrm{ell}(G)$ and $\cT(G)$ (when $G=\pu(n,1)$, $i(\cC)=(\alpha_1,\ldots,\alpha_n)$). We have:
\begin{equation}\label{cbar}
\begin{aligned}
\operatorname{Hom}_\mathrm{ell}(\Gamma,G) &\stackrel{c}{\longrightarrow} & \operatorname{Conj}_\mathrm{ell}(G)^3 & \stackrel{i^3}{\longrightarrow}  \cT(G)^3  \\
& & \cH &\longmapsto  \cP(G)
\end{aligned}
\end{equation}
where $\cP(G)=i^3(\cH)$. We denote $\bar{c}=i^3\circ c$. 

\begin{Rema}
	Take $G=\pu(2,1)$. Note that $i$ is well defined, but not continuous. Indeed, for all $\alpha\in[0,2\pi]$, we have $i^{-1}(\alpha,0)=i^{-1}(2\pi,\alpha)$. So $i$ is continuous if it takes its values in the quotient
	$${^{\displaystyle {\overline{\cT(G)}}}}\Big/{_{\displaystyle {(\alpha,0)\sim(2\pi,\alpha)}}}$$
	which is homeomorphic to a Möbius band. Alternatively, $i$ becomes continuous when restricted to the set of regular elliptic conjugacy classes, which is equal to $i^{-1}(\mathring{\cT(G)})$ (where $\mathring{\cT(G)}$ denotes the interior of $\cT(G)$). We denote this set by $\operatorname{Conj}_\mathrm{ell}^\mathrm{reg}(\Gamma,G)$.
\end{Rema}

More generally, we have:

\begin{Lemm}
	Let $G=\pu(n,1)$. The interior $\mathring{\cT(G)}$ is homeomorphic to an open ball, and the restriction
	$$\operatorname{Conj}_\mathrm{ell}^\mathrm{reg}(G) \stackrel{i}{\longrightarrow} \mathring{\cT(G)}$$
	is a homeomorphism.
\end{Lemm}

\subsection{Qualitative result in $\pu(n,1)$}

\begin{Defi}
	A subset of $\cT(G)^3$ is called a \textit{polytope} if it is an intersection of affine half-spaces in $\cT(G)^3\subset(\RR^{n})^3$.
\end{Defi}

In particular, a polytope is a convex subset of $\cT(G)^3$. The main result of this section is the following, which describes the structure of the solution set $\cP(G)\subset\cT(G)^3$ (see \eqref{cbar}) of Horn's problem:

\begin{Theo}\label{generalsolution}
	Assume $G=\pu(n,1)$. Then, $\cP(G)$ is a finite union of polytopes in $\cT(G)^3$.
\end{Theo}

In Sections \ref{horn} to \ref{proof}, we will show:

\begin{Theo}
	Assume $G=\pu(2,1)$. Then, $\cP(G)$ is a non-disjoint union of five explicit polytopes in $\cT(G)^3$.
\end{Theo}

\subsection{Structure of the space of solutions}

In this section, we prove Theorem \ref{generalsolution}. We identify the hyperplanes that define the polytopes composing $\cP(G)$ as coming from reducible representations, and we uncover fundamental properties of the complement set of these hyperplanes (see the Introduction).

\subsubsection{Reducible subspaces and lifting properties}\label{liftingsection}

We start by defining the subsets of $\cT(G)^3$ composed of reducible and irreducible representations.

\begin{Defi}
	We define the following subsets:
	\begin{itemize}
		\item $\operatorname{Hom}_\mathrm{ell}^\mathrm{red}(\Gamma,G)$ is the set of reducible boundary elliptic representations of $\Gamma$ in $G$;
		\item $\operatorname{Hom}_\mathrm{ell}^\mathrm{irred}(\Gamma,G)$ is the set of irreducible boundary elliptic representations of $\Gamma$ in $G$;
		\item $\cP(G)^\mathrm{red}=\bar{c}(\operatorname{Hom}_\mathrm{ell}^\mathrm{red}(\Gamma,G))\subset\cT(G)^3$;
		\item $\cP(G)^\mathrm{irred}=\bar{c}(\operatorname{Hom}_\mathrm{ell}^\mathrm{irred}(\Gamma,G))\subset\cT(G)^3.$
	\end{itemize}
\end{Defi}

\begin{Rema}
	For $G=\pu(n,1)$ with $n\geq2$, $\cP(G)^\mathrm{red}\cap\cP(G)^\mathrm{irred}\neq\varnothing$, i.e.\ we may have $\rho\in\operatorname{Hom}_\mathrm{ell}^\mathrm{irred}(\Gamma,G)$ and $\bar{c}(\rho)\in\cP(G)^\mathrm{red}$; in other words, an irreducible representation and a reducible representation may have the same boundary elliptic classes. We have produced a family of such examples for $\pu(2,1)$ in \cite{Mar} (the so-called decomposable representations).
\end{Rema}

We now discuss choices of lifts of elements in $\pu(n,1)$ to $\su(n,1)$. This will greatly simplify the study of the combinatorics of the hyperplanes defining the polytopes in $\cP(G)$ (especially in Section \ref{horn}, where we will make everything explicit for $n=2$). When $G=\pu(n,1)$, there is no canonical choice of lift for an element $g\in G$ to $\su(n,1)$. We now arbitrarily fix such a choice, to which we will stick for the whole paper.

\begin{Defi}\label{standardlift}
	Let $G=\pu(n,1)$. Consider $\cC\in\operatorname{Conj}_\mathrm{ell}(G)$ and its image $(\alpha_1,\ldots,\alpha_n)=i(\cC)\in\cT(G)$. We define the \textit{standard lift of $\cC$ to $\su(n,1)$} as the conjugacy class $\widetilde{\cC}$ of the matrix
	$$\begin{pmatrix}
	e^{i\frac{n\alpha_1-(\alpha_2+\ldots+\alpha_n)}{n+1}} & 0 & 0 & 0 \\
	0 & \ddots & 0 & 0 \\
	0 & 0 & e^{i\frac{n\alpha_n-(\alpha_1+\ldots+\alpha_{n-1})}{n+1}} & 0 \\
	0 & 0 & 0 & e^{-i\frac{\alpha_1+\ldots+\alpha_n}{n+1}}
	\end{pmatrix}.$$
	Analogously, if $\rho\in\operatorname{Hom}_\mathrm{ell}(\Gamma,G)$, then for $g\in \Gamma$ such that $\rho(g)\in\cC$, we denote by $\widetilde{\rho(g)}$ the unique lift of $\rho(g)$ to $\su(n,1)$ that satisfies $\widetilde{\rho(g)}\in\widetilde{\cC}$.
\end{Defi}

\begin{Prop}
	Let $G=\pu(n,1)$. There exists a continuous map $\ell$ such that the following diagram commutes:
	$$\xymatrix{
		& \operatorname{Conj}_\mathrm{ell}^\mathrm{reg}(\Gamma,\su(n,1))^3 \ar[d]^{p} \\
		\chi_\mathrm{ell}^\mathrm{reg}(\Gamma,G) \ar@{.>}[ru]^{\ell} \ar[r]_{c_{\chi}} & \operatorname{Conj}_\mathrm{ell}^\mathrm{reg}(\Gamma,G)^3 }$$
\end{Prop}

\begin{proof}
	This comes from the fact that $\operatorname{Conj}_\mathrm{ell}^\mathrm{reg}(\Gamma,\su(n,1))^3$ is a covering of $\operatorname{Conj}_\mathrm{ell}^\mathrm{reg}(\Gamma,G)^3$, and that $\operatorname{Conj}_\mathrm{ell}^\mathrm{reg}(\Gamma,G)^3$ is homeomorphic to the product of three open balls so is simply connected. The result is a direct application of the lifting of continuous maps theorem (from the theory of covering spaces).
\end{proof}

\begin{Coro}\label{lift}
	The map $\operatorname{Hom}_\mathrm{ell}^\mathrm{reg}(\Gamma,G)\to\operatorname{Conj}_\mathrm{ell}^\mathrm{reg}(\Gamma,\su(n,1))^3$ that sends a boundary regular elliptic representation $\rho$ (see Definition \ref{regspe}) to the triple of conjugacy classes of the standard lifts $(\widetilde{\rho(\mathtt{a})},\widetilde{\rho(\mathtt{b})},\widetilde{\rho(\mathtt{c})})$ of Definition \ref{standardlift} is continuous.
\end{Coro}

\begin{proof}
	This map is the composition $\ell\circ\pi_\mathrm{ell}^\mathrm{reg}$, where $\pi_\mathrm{ell}^\mathrm{reg}:\operatorname{Hom}_\mathrm{ell}^\mathrm{reg}(\Gamma,G)\to\chi_\mathrm{ell}^\mathrm{reg}(\Gamma,G)$ is the canonical projection and is continuous.
\end{proof}

\begin{Rema}
	Take $G=\pu(n,1)$ and $\rho\in\operatorname{Hom}_\mathrm{ell}(\Gamma,G)$. Let $\widetilde{\rho(\texttt{a})},\widetilde{\rho(\texttt{b})},\widetilde{\rho(\texttt{c})}$ be the standard lifts of $\rho(\texttt{a}),\rho(\texttt{b}),\rho(\texttt{c})$ to $\su(n,1)$. Then, $$\widetilde{\rho(\texttt{a})}\widetilde{\rho(\texttt{b})}\widetilde{\rho(\texttt{c})}=u\operatorname{Id},$$
	where $u\in\mathbb{U}_{n+1}$ is an $(n+1)$-root of unity.
\end{Rema}

\begin{Defi}\label{P(G)u}
	For $G=\pu(n,1)$ and $u\in\mathbb{U}_{n+1}$, we define:
	\begin{itemize}
		\item $\operatorname{Hom}_\mathrm{ell}(\Gamma,G)_u=\{\rho\in\operatorname{Hom}_\mathrm{ell}(\Gamma,G) \ | \ \widetilde{\rho(\texttt{a})}\widetilde{\rho(\texttt{b})}\widetilde{\rho(\texttt{c})}=u\operatorname{Id} \}$;
		\item $\operatorname{Hom}_\mathrm{ell}^\mathrm{red}(\Gamma,G)_u=\{\rho\in\operatorname{Hom}_\mathrm{ell}^\mathrm{red}(\Gamma,G) \ | \ \widetilde{\rho(\texttt{a})}\widetilde{\rho(\texttt{b})}\widetilde{\rho(\texttt{c})}=u\operatorname{Id} \}$;
		\item $\cP(G)_u=\overline{c}(\operatorname{Hom}_\mathrm{ell}(\Gamma,G)_u)$;
		\item $\cP(G)^\mathrm{red}_u=\overline{c}(\operatorname{Hom}_\mathrm{ell}^\mathrm{red}(\Gamma,G)_u).$
	\end{itemize}
\end{Defi}

We have:
$$\operatorname{Hom}_\mathrm{ell}(\Gamma,G)=\bigsqcup_{u\in\mathbb{U}_{n+1}}\operatorname{Hom}_\mathrm{ell}(\Gamma,G)_u \ \mbox{and} \ \cP(G)=\bigcup_{u\in\mathbb{U}_{n+1}}\cP(G)_u,$$
where the first union is disjoint and the second union may be non-disjoint (as we will see in Section \ref{horn}).

\subsubsection{Definition of a cell and fundamental properties}

We have the following important structural result, which will be proven in Section \ref{Predproof}.

\begin{Prop}\label{Pred}
		Let $G=\pu(n,1)$. For $u\in\mathbb{U}_{n+1}$, $\cP(G)^\mathrm{red}_u$ is a family of polytopes inside hyperplanes in $\cT(G)^3$.
\end{Prop}

We may now define our notion of cell.

\begin{Defi}\label{cell}
	Let $G=\pu(n,1)$. For any $u\in\mathbb{U}_{n+1}$, we call \textit{$u$-cell} (or more generally \textit{cell}) a connected component of the complement $(\cP(G)^\mathrm{red}_u)^c$ in $\cT(G)^3$.
\end{Defi}

\begin{Rema}\label{halfspace}
	Let $G=\pu(n,1)$ and $u\in\mathbb{U}_{n+1}$. Let $C$ be a $u$-cell. Proposition \ref{Pred} implies that $C$ is an intersection of affine half-spaces in $\cT(G)^3$ (in particular, $C$ is convex).
\end{Rema}

We will describe precisely the cells for our examples of groups $G$. The following result reveals to be crucial to solve Horn's problem in our cases.

\begin{Theo}\label{emptyfulltheorem}
	Let $G=\pu(n,1)$. Take $u\in\mathbb{U}_{n+1}$ and $C$ a $u$-cell. Then, we have either $C\subset \cP(G)_u$ or $C\subset (\cP(G)_u)^c$.
\end{Theo}

\begin{Rema}
	This implies that one of the following two cases occurs:
	\begin{itemize}
		\item $\forall\tau\in C$, $\overline{c}^{-1}(\tau)\cap\operatorname{Hom}_\mathrm{ell}(\Gamma,G)_u\neq\varnothing$, i.e.\ there exists a representation $\rho$ whose standard lifts satisfy $\widetilde{\rho(\texttt{a})}\widetilde{\rho(\texttt{b})}\widetilde{\rho(\texttt{c})}=u\operatorname{Id},$ and such that the conjugacy classes of $\rho(\texttt{a}),\rho(\texttt{b}),\rho(\texttt{c})$ are given by $\tau$. In this case, we say that $C$ is \textit{full};
		\item $\forall\tau\in C$, $\overline{c}^{-1}(\tau)\cap\operatorname{Hom}_\mathrm{ell}(\Gamma,G)_u=\varnothing$, i.e.\ there exists no representation $\rho$ whose standard lifts satisfy $\widetilde{\rho(\texttt{a})}\widetilde{\rho(\texttt{b})}\widetilde{\rho(\texttt{c})}=u\operatorname{Id},$ and such that the conjugacy classes of $\rho(\texttt{a}),\rho(\texttt{b}),\rho(\texttt{c})$ are given by $\tau$. In this case, we say that $C$ is \textit{empty}.
	\end{itemize}
\end{Rema}

Theorem \ref{emptyfulltheorem} is a consequence of Propositions \ref{open} and \ref{closed} below.

\begin{Prop}\label{open}
    Let $G$ be a Lie group, and denote by $Z(G)$ the center of $G$. The map $$c_u:\operatorname{Hom}_\mathrm{ell}(\Gamma,G)_u\to\operatorname{Conj}_\mathrm{ell}(G)^3/Z(G)$$ is open at an irreducible representation.
\end{Prop}

\begin{Rema}\label{dimension}
	Note that if $Z(G)=\{\operatorname{Id}\}$, as in the case of $\pu(n,1)$, then we have $$\operatorname{Conj}_\mathrm{ell}(G)^3/Z(G)\cong\operatorname{Conj}_\mathrm{ell}(G)^3.$$ In the case of $\operatorname{U}(n)$, we have $Z(G)=\operatorname{U}(1)$, so $\operatorname{Conj}_\mathrm{ell}(G)^3/Z(G)$ is identified to a co-dimension 1 subspace of $\operatorname{Conj}_\mathrm{ell}(G)^3$. Therefore:
	\begin{itemize}
		\item If $G=\pu(n,1)$, $\dim(\cP(G)^{\mathrm{irred}})=\dim(\cT(G))=3n;$
		\item If $G=\U(n)$, $\dim(\cP(G)^{\mathrm{irred}})=\dim(\cT(G))-1=3n-1$.
	\end{itemize}
\end{Rema}

The proof of Proposition \ref{open} is classical: see for instance Proposition 4.2 in Falbel-Wentworth \cite{FalWen2}, or Proposition 2.5 in Paupert \cite{Pau}.

The second crucial result is the following (we give a proof in Section \ref{closedness} below):

\begin{Prop}\label{closed}
	Let $u\in\mathbb{U}_{n+1}$. Then, $\cP(G)_u$ is closed in the interior $\mathring{\cT(G)}^3$.
\end{Prop}

We now prove Theorem \ref{emptyfulltheorem}.

\begin{proof}
	Let $C$ be a $u$-cell. Then by Propositions \ref{open} and \ref{closed}, $\cP(G)_u\cap C$ is both open and closed in $C$. So by connectedness of $C$, $\cP(G)_u\cap C=C$ or $\cP(G)_u\cap C=\varnothing$.
\end{proof}

Therefore, $\cP(G)$ is equal to a finite union of cells. This, together with Remark \ref{halfspace}, proves Theorem \ref{generalsolution}. The following sections are dedicated to the proofs of Propositions \ref{Pred} and \ref{closed}.

\subsubsection{Hyperplane structure of $\cP(G)^\mathrm{red}_u$}\label{Predproof}

We first prove Proposition \ref{Pred}.

\begin{proof}
	Let $\rho\in\operatorname{Hom}_\mathrm{ell}^\mathrm{red}(\Gamma,G)_u$ and $A,B,C\in\U(n,1)$ be lifts of $\rho(\texttt{a}),\rho(\texttt{b}),\rho(\texttt{c})$ respectively conjugate to $E(\alpha_1,\ldots,\alpha_n)$, $E(\beta_1,\ldots,\beta_n)$, $E(\gamma_1,\ldots,\gamma_n)$ (see Example \ref{groupexamples}). We have $ABC=\lambda I_{n+1}$ for some $\lambda\in\CC$. By Definition \ref{reducibledef}, there exist two integers $d_1,d_2\in\NN$ such that $d_1+d_2=n+1$, three elements $A_1,B_1,C_1\in\U(d_1)$ and three elements $A_2,B_2,C_2\in\pu(d_2-1,1)$ such that up to global conjugation,
	$$A=\begin{pmatrix}
	A_1 & 0 \\
	0 & A_2
	\end{pmatrix}, \ B=\begin{pmatrix}
	B_1 & 0 \\
	0 & B_2
	\end{pmatrix}, \ C=\begin{pmatrix}
	C_1 & 0 \\
	0 & C_2
	\end{pmatrix}$$ and
	$$\begin{pmatrix}
	A_1 & 0 \\
	0 & A_2
	\end{pmatrix}\begin{pmatrix}
	B_1 & 0 \\
	0 & B_2
	\end{pmatrix}\begin{pmatrix}
	C_1 & 0 \\
	0 & C_2
	\end{pmatrix}=\begin{pmatrix}
	\lambda I_{d_1} & 0 \\
	0 & \lambda I_{d_2}
	\end{pmatrix}.$$
	Computing the determinants in each block, we obtain
	\begin{equation}\label{hypeq}
	d_2\arg(\det(A_1B_1C_1))-d_1\arg(\det(A_2B_2C_2))\equiv0[2\pi],
	\end{equation}
	which is a set of hyperplane equations in $\cT(G)^3$. Indeed, $$\arg(\det(A_1B_1C_1))=\sum_{i\in\cI_{d_1}}\alpha_i+\sum_{j\in\cJ_{d_1}}\beta_j+\sum_{k\in\cK_{d_1}}\gamma_k$$ where $\cI_{d_1},\cJ_{d_1},\cK_{d_1}\subset\{1,\ldots, n\}$ are three subsets of $d_1$ indices, and $$\arg(\det(A_2B_2C_2))=\sum_{\bar{i}\in\overline{\cI_{d_1}}}\alpha_{\bar{i}}+\sum_{\bar{j}\in\overline{\cJ_{d_1}}}\beta_{\bar{j}}+\sum_{\bar{k}\in\overline{\cK_{d_1}}}\gamma_{\bar{k}}$$ where $\overline{\cI_{d_1}},\overline{\cJ_{d_1}},\overline{\cK_{d_1}}\subset\{1,\ldots, n\}$ are the three complementary subsets of $d_2-1$ indices. This implies that $\bar{c}(\rho)$ belongs to one of the hyperplanes given in \eqref{hypeq} by choosing a multiple of $2\pi$.  
	
	Conversely, assume that a triple $\tau\in\cT(G)^3$ belongs to one of the hyperplanes defined in \eqref{hypeq} for some integers $d_1,d_2$ satisfying $d_1+d_2=n+1$. Let
	\begin{equation*}
	\begin{aligned}
	\tau_{d_1}&=((\alpha_i)_{i\in\cI_{d_1}},(\beta_j)_{j\in\cJ_{d_1}},(\gamma_k)_{k\in\cK_{d_1}})\in\cT(\U(d_1))^3, \\
	\tau_{d_2}&=((\alpha_{\bar{i}})_{\bar{i}\in\overline{\cI_{d_1}}},(\beta_{\bar{j}})_{\bar{j}\in\overline{\cJ_{d_1}}},(\gamma_{\bar{k}})_{\bar{k}\in\overline{\cK_{d_1}}})\in\cT(\pu(d_2-1,1))^3
	\end{aligned}
	\end{equation*}
	be the corresponding sub-triples. Assume that $\tau_{d_1}$, respectively $\tau_{d_2}$, is a solution of Horn's problem in $\U(d_1)$, respectively in $\pu(d_2-1,1)$. Then by the definition of reducibility, $\tau\in\cP(G)^\mathrm{red}_u$. By Remark \ref{dimension}, $\dim(\cP(\U(d_1)))=3d_1-1$ and $\dim(\cP(\pu(d_2-1,1)))=3(d_2-1)$ so
	$$\dim(\cP(G)^\mathrm{red}_u)=(3d_1-1)+3(d_2-1)=3(n+1)-4=3n-1.$$
	This implies that the reducibility conditions define open subsets of the hyperplanes defined in \eqref{hypeq}. This concludes the proof.
\end{proof}

\subsubsection{Closedness of the space of solutions: proof of Proposition \ref{closed}}\label{closedness}

Once again, this fact is already known (see Falbel-Wentworth \cite{FalWen}, Paupert-Will \cite{PauWil}). When $G$ is a compact group, this result is straightforward. In the case of $\pu(n,1)$, we use the geometry of $\mathbb{H}_\CC^n$ through Corollary \ref{bestpau}. Note that a slight subtlety arises because of the $u$-cells, which justifies the discussion about the lifting properties of a representation that took place in Section \ref{liftingsection}.

\begin{proof}
	Let $(\tau_k)_k\subset\cP(G)_u$ be a sequence. Then,
	$$\forall k\in\NN, \ \exists \rho_k\in\operatorname{Hom}_\mathrm{ell}(\Gamma,G)_u, \ \bar{c}(\rho_k)=\tau_k.$$ 
	Assume that this sequence converges to $\tau=i^3(\cC_1,\cC_2,\cC_3)\in\mathring{\cT(G)}^3$. For every $k$, denote $(A_k,B_k,C_k)=(\rho_k(\texttt{a}),\rho_k(\texttt{b}),\rho_k(\texttt{c})).$ Now:
	\begin{itemize}
		\item if $G=\U(n)$ or $\su(n)$ is compact, then up to extracting a sub-sequence, there exists $(A,B)\in\cC_1\times\cC_2$ such that $(A_k,B_k)$ converges to $(A,B)$;
		\item if $G=\pu(n,1)$, $A_kB_k=C_k^{-1}$ is elliptic for every $k$, so by Corollary \ref{bestpau}, up to extracting a sub-sequence and global conjugation, there exists $(A,B)\in\cC_1\times\cC_2$ such that $(A_k,B_k)$ converges to $(A,B)$.
	\end{itemize}
	Consequently, $(C_k)$ converges to $(AB)^{-1}=C$, so we have shown the existence of $(A,B,C)\in\cC_1\times\cC_2\times\cC_3$ such that $ABC=\operatorname{Id}$: there exists $\rho\in\operatorname{Hom}_\mathrm{ell}(\Gamma,G)$ such that $c(\rho)\in\cC_{1}\times\cC_{2}\times\cC_{3}$. Finally, Corollary \ref{lift} ensures that $\rho\in\operatorname{Hom}_\mathrm{ell}(\Gamma,G)_u$, which implies that $\tau=\bar{c}(\rho)\in\cP(G)_u$.
\end{proof}

\subsubsection{Tools for detecting full cells}

In practice, giving the inequations defining the cells and determining which ones are full and which ones are empty reveals to be quite technical for $n\geq2$. See Sections \ref{horn} to \ref{proof}, where we prove the explicit solution for $\pu(2,1)$. Among other tools, we will use the following result.

\begin{Prop}\label{emptyfullinterface}
	Let $C_1$ and $C_2$ be two distinct $u$-cells such that $\overline{C_1}\cap\overline{C_2}\subset\cP(G)_u^\mathrm{red}$ is non-empty and of co-dimension $1$ in $\cT(G)^3$. Then:
	\begin{itemize}
		\item At least one cell among $C_1$ and $C_2$ is full.
		\item If there exists $\rho\in\operatorname{Hom}_\mathrm{ell}^\mathrm{irred}(\Gamma,G)_u$ such that $\bar{c}(\rho)\in \overline{C_1}\cap\overline{C_2}$, then both $C_1$ and $C_2$ are full.
	\end{itemize}
\end{Prop}

\begin{proof}
	The second point is a direct consequence of Proposition \ref{open}. The first point is Corollary 4 in \cite{PauWil}, but we give here a proof for the sake of completeness. Let $(\cC_1,\cC_2,\cC_3)\in\operatorname{Conj}_\mathrm{ell}^\mathrm{reg}(G)^3$ be such that $i^3(\cC_1,\cC_2,\cC_3)\in\overline{C_1}\cap\overline{C_2}\subset\cP(G)_u^\mathrm{red}$ and $i^3(\cC_1,\cC_2,\cC_3)$ belongs to the interior of $C_1\cup C_2$. Then, there exists a reducible triple $(A,B,C)\in\cC_1\times\cC_2\times\cC_3$ such that $ABC=1$. 
	\begin{itemize}
		\item First of all, $\operatorname{Conj}_\mathrm{ell}^\mathrm{reg}(G)$ is open in $\operatorname{Conj}(G)$: indeed, $\operatorname{Conj}_\mathrm{ell}^\mathrm{reg}(G)$ is a three-fold covering of the negative locus of a continuous map (see Theorem 6.2.4 in \cite{Gol}).
		\item Furthermore, since the set of reducible pairs in $\cC_1\times\cC_2$ forms a strict sub-manifold of $\cC_1\times\cC_2$, there exists an irreducible pair $(A,B')\in\cC_1\times\cC_2$ in a neighborhood of $(A,B)$.
	\end{itemize}
	Therefore, by continuity of the product map, there exists $\cC_3'\in\operatorname{Conj}_\mathrm{ell}^\mathrm{reg}(G)$ in a neighborhood of $\cC_3$ and an irreducible triple $(A,B',(AB')^{-1})\in\cC_1\times\cC_2\times\cC_3'$. Since $i^3(\cC_1,\cC_2,\cC_3')\in C_1\cup C_2$, the result follows from Theorem \ref{emptyfulltheorem}.
\end{proof}

\section{Horn's problem in $\operatorname{U}(2)$ and $\pu(1,1)$}\label{horn_u}

\subsection{Horn's problem in $\operatorname{U}(2)$}

We now state the explicit solution in $\operatorname{U}(2)$, because we will use it to prove our result for $\pu(2,1)$ in Section \ref{horn}. This is a particular case of Theorem 2.5 in Biswas' paper \cite{Bis1}; see also Ressayre \cite{Res2}. 

We fix $G=\operatorname{U}(2)$. Note that $\operatorname{Conj}_\mathrm{ell}(G)=\operatorname{Conj}(G)$. Recall from Example \ref{groupexamples} that we have $$\cT(G)=\{(\alpha_1,\alpha_2)\in[0,2\pi[^2 \ | \ \alpha_1\geq \alpha_2\}.$$

\begin{Defi}\label{linearforms}
	We define the linear forms $S:\RR^6\to\RR$ and $\sigma_{ijk}:\RR^6\to\RR$ for $(i,j,k)\in\{1,2\}^3$ by: $\forall((\alpha_1,\alpha_2),(\beta_1,\beta_2),(\gamma_1,\gamma_2))\in \RR^6,$
	\begin{itemize}
		\item $S((\alpha_1,\alpha_2),(\beta_1,\beta_2),(\gamma_1,\gamma_2))=\alpha_1+\alpha_2+\beta_1+\beta_2+\gamma_1+\gamma_2$;
		\item $\sigma_{ijk}((\alpha_1,\alpha_2),(\beta_1,\beta_2),(\gamma_1,\gamma_2))=\alpha_i+\beta_j+\gamma_k.$
	\end{itemize}
	In practice, we will only consider their restrictions to $\cT(G)^3$. To lighten the notations, we will mostly write $S$ instead of $S((\alpha_1,\alpha_2),(\beta_1,\beta_2),(\gamma_1,\gamma_2))$, respectively $\sigma_{ijk}$ instead of $\sigma_{ijk}((\alpha_1,\alpha_2),(\beta_1,\beta_2),(\gamma_1,\gamma_2))$.
\end{Defi}

\begin{Defi}\label{ibar}
	Let $i\in\{1,2\}$ be an index. We denote by $\bar{i}\in\{1,2\}$ the complementary choice.
\end{Defi}

\begin{Rema}\label{sum}
	For $(i,j,k)\in\{1,2\}^3$, we have $\sigma_{ijk}+\sigma_{\bar{i}\bar{j}\bar{k}}=S$.
\end{Rema}

\begin{Prop}\label{U2}
	Consider the following subsets of $\cT(G)^3$:
	\begin{equation*}
	\begin{aligned}
	C_{0}&=\{S=0\}, \\
	C_{2\pi}&=\{S=2\pi, \ \sigma_{222}\leq0,\ \sigma_{112}\leq2\pi,\ \sigma_{121}\leq2\pi, \ \sigma_{211}\leq2\pi\}, \\
	C_{4\pi}&=\{S=4\pi, \ \sigma_{111}\leq4\pi, \ \sigma_{122}\leq2\pi, \ \sigma_{212}\leq2\pi, \ \sigma_{221}\leq2\pi\}, \\
	C_{6\pi}&=\{S=6\pi, \ \sigma_{222}\leq2\pi, \ \sigma_{112}\leq4\pi, \ \sigma_{121}\leq4\pi, \ \sigma_{211}\leq4\pi\},\\
	C_{8\pi}&=\{S=8\pi, \ \sigma_{111}\leq6\pi, \ \sigma_{122}\leq4\pi, \ \sigma_{212}\leq4\pi, \ \sigma_{221}\leq4\pi\}.\\
	\end{aligned}
	\end{equation*}
	The set of solutions of Horn's problem in $\U(2)$ is given by
	$$\cP(G)=C_0\cup C_{2\pi}\cup C_{4\pi}\cup C_{6\pi}\cup C_{8\pi}\subset\cT(G)^3.$$
\end{Prop}

\subsection{Horn's problem in $\operatorname{PU}(1,1)$}\label{horn_pu}

In this section, we fix $G=\pu(1,1)$. Note that in this case, $\operatorname{Conj}_\mathrm{ell}(G)\neq\operatorname{Conj}(G)$. The results of this section are already known (see for instance Paupert \cite{Pau}); they come from classical manipulations of triangles in the Poincaré disk, and the Gauss-Bonnet formula. We will also use the explicit solution for $\pu(1,1)$ to prove our result for $\pu(2,1)$ in Section \ref{horn}.

Recall from Example \ref{groupexamples} that we have $\cT(G)=[0,2\pi[$.

\begin{Prop}\label{PU11}
	For $(\alpha,\beta,\gamma)\in\cT(G)^3$, denote $\sigma(\alpha,\beta,\gamma)=\alpha+\beta+\gamma$. We have
	$$\cP(G)=\{\sigma\leq2\pi\}\cup\{\sigma\geq4\pi\},$$
	where $\cP(G)_{-1}=\{\sigma\leq2\pi\}$ and $\cP(G)_1=\{\sigma\geq4\pi\}$ (see Figure \ref{PU(1,1)}).
\end{Prop}

\begin{figure}
	\centering
	\begin{tikzpicture}[scale=1.5]
	\draw [->] (0.5,2.5) -- (0.5,3);
	\draw (0.5,2.7) node[right] {\scriptsize{$2\pi$}};
	\draw [->] (2.5,0.5) -- (3,0.5);
	\draw (2.6,0.5) node[below] {\scriptsize{$2\pi$}};
	\draw [->] (0,0) -- (-0.4,-0.4);
	\draw (0,0) node[left] {\scriptsize{$2\pi$}};
	\draw (0.5,3) node[above] {$\gamma$};
	\draw (3,0.5) node[right] {$\beta$};
	\draw (-0.55,-0.55) node {$\alpha$};
	
	\draw (2,0) -- (0,0) -- (0,2) -- (2,2) -- (2,0);
	\draw (0,2) -- (0.5,2.5) -- (2.5,2.5) -- (2.5,0.5) -- (2,0);
	\draw (2,2) -- (2.5,2.5);
	\draw [blue, very thick, dashed] (0.5,2.5) -- (0.5,0.5) -- (2.5,0.5);
	\draw [blue, very thick, dashed] (0,0) -- (0.5,0.5);
	\draw [blue, very thick, dashed, fill=blue, opacity=0.3] (0.5,2.5) -- (0,0) -- (2.5,0.5) -- (0.5,2.5);
	\draw [blue, very thick, dashed] (0.5,2.5) -- (0,0) -- (2.5,0.5) -- (0.5,2.5);
	\draw [red, very thick, fill=red, opacity=0.3] (2,0) -- (0,2) -- (2.5,2.5) -- (2,0);
	\draw [red, very thick] (2,0) -- (0,2) -- (2.5,2.5) -- (2,0);
	\draw [red, very thick] (0,2) -- (2,2) -- (2,0);
	\draw [red, very thick] (2,2) -- (2.5,2.5);
	\end{tikzpicture}
	\caption{The two components of solutions in $\operatorname{PU}(1,1)$. \label{PU(1,1)}}
\end{figure}
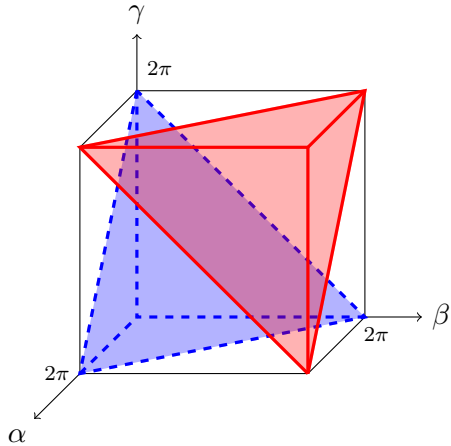

\begin{Rema}
	We have $\cP(G)^\mathrm{red}_ {-1}=\{\sigma=2\pi\}$ and $\cP(G)^\mathrm{red}_ {1}=\{\sigma=4\pi\}$. The $-1$-cells are $\{\sigma<2\pi\}$ and $\{\sigma>2\pi\}$, and the $1$-cells are $\{\sigma<4\pi\}$ and $\{\sigma>4\pi\}$.
\end{Rema}

\begin{proof}
	To prove this result, we use the action of $\operatorname{PU}(1,1)$ on the complex hyperbolic line $\HH^1_\CC$. Consider an irreducible triple
	$$A\sim\begin{pmatrix}
	e^{i\alpha} & 0 \\
	0 & 1
	\end{pmatrix}, \ B\sim\begin{pmatrix}
	e^{i\beta} & 0 \\
	0 & 1
	\end{pmatrix}, \ C\sim\begin{pmatrix}
	e^{i\gamma} & 0 \\
	0 & 1
	\end{pmatrix}.$$
	These isometries are rotations of angles $\alpha,\beta,\gamma$ around three distinct points. Now, the product of three such rotations is equal to the identity if and only if there exists a real hyperbolic triangle of angles $\alpha/2,\beta/2,\gamma/2$ or $\pi-\beta/2,\pi-\alpha/2,\pi-\gamma/2$. Indeed, when such a triangle exists, we also have reflections of order $2$ in its sides $s_1,s_2,s_3$ such that $A=s_1s_2,B=s_2s_3,C=s_3s_1$ or $B^{-1}=s_1s_2,A^{-1}=s_2s_3,C^{-1}=s_3s_1$, yielding $ABC=1$. Consequently, the representation exists if and only if we have $\alpha/2+\beta/2+\gamma/2<\pi$ or $\pi-\alpha/2+\pi-\beta/2+\pi-\gamma/2<\pi$, or equivalently, $\alpha+\beta+\gamma<2\pi$ or $\alpha+\beta+\gamma>4\pi$.
\end{proof}

\section{Explicit solution to the elliptic multiplicative Horn problem in PU(2,1): reducible subspaces and cells}\label{horn}

In this section and the following, we fix $G=\pu(2,1)$, and we solve the multiplicative Horn problem in the elliptic case (note that in $\pu(2,1)$, $\operatorname{Conj}_\mathrm{ell}(G)\neq\operatorname{Conj}(G)$). We use the definitions and notations introduced in Section \ref{horn_qual}. Our goal is to prove this explicit version of Theorem \ref{generalsolution}:

\begin{Theo}\label{solution}
	The solution set of the elliptic multiplicative Horn problem in $\pu(2,1)$ is a non-disjoint union of five explicit polytopes in $\cT(G)^3\subset(\RR^2)^3$.
\end{Theo}

We will give the explicit inequations for the five polytopes later on; but to do so, we first need to solve the reducible problem.

\subsection{Reducible subspaces}

The resolution of Horn's problem in $\operatorname{U}(2)$ and $\operatorname{PU}(1,1)$ (see Section \ref{horn_u}) gives us access to the solution for elliptic reducible triples in $\operatorname{PU}(2,1)$. We recall (see Example \ref{groupexamples}) that for $G=\pu(2,1)$, we have $$\cT(G)=\{(\alpha_1,\alpha_2)\in[0,2\pi[^2 \ | \ \alpha_1\geq \alpha_2\}.$$

In the following sections, we give the explicit equations of the pieces of hyperplanes that form $\cP(G)^\mathrm{red}$ (see Proposition \ref{Pred}).

\begin{Rema}\label{symmetry}
	We now exhibit a fundamental symmetry of the problem, which will divide by two the number of computations. Note that $ABC=\operatorname{Id}$ if and only if $C^{-1}B^{-1}A^{-1}=\operatorname{Id}$. Therefore, the involution
	\begin{equation*}
	\begin{aligned}
	\psi:\mathring{\cT(G)}^3&\to\mathring{\cT(G)}^3\\
	\Bigg(\begin{pmatrix}
	\alpha_1\\
	\alpha_2
	\end{pmatrix},\begin{pmatrix}
	\beta_1\\
	\beta_2
	\end{pmatrix},\begin{pmatrix}
	\gamma_1\\
	\gamma_2
	\end{pmatrix}\Bigg)&\mapsto\Bigg(\begin{pmatrix}
	2\pi-\gamma_2\\
	2\pi-\gamma_1
	\end{pmatrix},\begin{pmatrix}
	2\pi-\beta_2\\
	2\pi-\beta_1
	\end{pmatrix},\begin{pmatrix}
	2\pi-\alpha_2\\
	2\pi-\alpha_1
	\end{pmatrix}\Bigg)
	\end{aligned}
	\end{equation*}
	globally preserves the solution set $\cP(G)$. 
\end{Rema}

\subsubsection{Totally reducible facets}

\begin{Defi}
	Let $\rho\in\operatorname{Hom}_\mathrm{ell}(\Gamma,G)$. We say that this representation is \textit{totally reducible} if $\rho(\texttt{a}),\rho(\texttt{b}),\rho(\texttt{c})$ admit lifts to $\U(2,1)$ that are diagonalizable in the same basis (equivalently, the group $\rho(\Gamma)\subset\pu(2,1)$ is abelian). We denote by:
	\begin{itemize}
		\item $\operatorname{Hom}_\mathrm{ell}^\mathrm{TR}(\Gamma,G)$ the subset of totally reducible representations;
		\item $\cP(G)^\mathrm{TR}$ the set of triples of angle pairs that come from totally reducible representations. More precisely, $\cP(G)^\mathrm{TR}=\overline{c}\big(\operatorname{Hom}_\mathrm{ell}^\mathrm{TR}(\Gamma,G)\big)\subset \cP(G)^\mathrm{red}$ (see diagram (\ref{cbar}) for the definition of $\overline{c}$).
	\end{itemize}
\end{Defi}

\begin{Defi}\label{Tijk}
	For $m,n\in\{0,1,2,3\},$ we denote by $T_{ijk}(m,n)$, and we call \textit{totally reducible facet}, the following four-dimensional subspace of $\cT(G)^3$: 
	$$T_{ijk}(m,n)=\{\tau\in\cT(G)^3 \ | \ \sigma_{ijk}=2m\pi,\sigma_{\bar{i}\bar{j}\bar{k}}=2n\pi \}$$
	where $\sigma_{ijk}=\alpha_i+\beta_j+\gamma_k$ (see Definition \ref{linearforms}, and Definition \ref{ibar} for the complementary indices $\bar{i},\bar{j},\bar{k}$).
\end{Defi}

\begin{Rema}\label{symtotred}
	Recall the symmetry $\psi$ from Remark \ref{symmetry}. For $\tau\in\cT(G)^3$, we have
	$$\sigma_{ijk}(\psi(\tau))=6\pi-\sigma_{\bar{k}\bar{j}\bar{i}}(\tau).$$
	This implies that $\psi$ acts on the set of totally reducible facets by
	$$\psi(T_{ijk}(m,n))=T_{\bar{k}\bar{j}\bar{i}}(3-m,3-n).$$
\end{Rema}

\begin{Prop}
	The subspace $\cP(G)^\mathrm{TR}$ is the reunion of the totally reducible facets $T_{ijk}(m,n)$ for $(i,j,k)\in\{1,2\}^3$ and $m,n\in\{0,1,2,3\}$.
\end{Prop}

\subsubsection{Spherical reducible walls}

\begin{Defi}
	Let $\rho\in\operatorname{Hom}_\mathrm{ell}(\Gamma,G)$ be a representation. We say that $\rho$ is \textit{spherical reducible} if $\rho(\texttt{a}),\rho(\texttt{b}),\rho(\texttt{c})$ fix the same point inside $\HH^2_\CC$. We denote:
	\begin{itemize}
		\item $\operatorname{Hom}_\mathrm{ell}^\mathrm{SR}(\Gamma,G)$ the subset of spherical reducible representations;
		\item $\cP(G)^\mathrm{SR}$ the set of triples of angle pairs that come from spherical reducible representations. More precisely, $\cP(G)^\mathrm{SR}=\overline{c}\big(\operatorname{Hom}_\mathrm{ell}^\mathrm{SR}(\Gamma,G)\big)\subset \cP(G)^\mathrm{red}$.
	\end{itemize}
\end{Defi}

\begin{Lemm}\label{U(2)}
	Let $\rho\in\operatorname{Hom}_\mathrm{ell}^\mathrm{SR}(\Gamma,G)$ be a spherical reducible representation. Then up to global conjugation, \emph{$\rho(\texttt{a}),\rho(\texttt{b}),\rho(\texttt{c})$} admit lifts to $\U(2,1)$ of the form
	$$\begin{pmatrix}
	A' & 0 \\
	0 & 1
	\end{pmatrix}, \ \begin{pmatrix}
	B' & 0 \\
	0 & 1
	\end{pmatrix}, \ \begin{pmatrix}
	C' & 0 \\
	0 & 1
	\end{pmatrix}$$
	with $A',B',C'\in\U(2)$.
\end{Lemm}

\begin{proof}
	This is a direct consequence of the fact that the stabilizer of a point in $\HH^2_\CC$ is $\mathbb{P}(\U(2)\times\U(1))\cong\U(2)$ (see Theorem \ref{isomclassif}).
\end{proof}

With the notations of Lemma \ref{U(2)}, $\rho$ is a spherical reducible representation if and only if the triple $(A',B',C')$ is a solution to Horn's problem in $\U(2)$. Consequently, the solution for $\U(2)$ exposed in Proposition \ref{U2} can be directly transposed in this setting.

\begin{Theo}\label{sphred}
	Denote
	\begin{equation*}
	\begin{aligned}
	\Sigma_{4\pi}&=\{S=4\pi, \ \sigma_{111}\leq4\pi, \ \sigma_{122}\leq2\pi, \ \sigma_{212}\leq2\pi, \ \sigma_{221}\leq2\pi\}, \\
	\Sigma_{6\pi}&=\{S=6\pi, \ \sigma_{222}\leq2\pi, \ \sigma_{112}\leq4\pi, \ \sigma_{121}\leq4\pi, \ \sigma_{211}\leq4\pi\}, \\
	\Sigma_{8\pi}&=\{S=8\pi, \ \sigma_{111}\leq6\pi, \ \sigma_{122}\leq4\pi, \ \sigma_{212}\leq4\pi, \ \sigma_{221}\leq4\pi\}.
	\end{aligned}
	\end{equation*}
	where $S=\alpha_1+\alpha_2+\beta_1+\beta_2+\gamma_1+\gamma_2$ (see Definition \ref{linearforms}). We have
	$$\cP(G)^\mathrm{SR}=\Sigma_{4\pi}\cup \Sigma_{6\pi}\cup \Sigma_{8\pi}\subset\cT(G)^3.$$
\end{Theo}

\begin{Defi}
	The three truncated hyperplanes $\Sigma_{4\pi},\Sigma_{6\pi},\Sigma_{8\pi}$ are called the \textit{spherical reducible walls}.
\end{Defi}

The cases of equality in the inequalities above (when $\sigma_{ijk}=2m\pi$) correspond to totally reducible (i.e.\ simultaneously diagonalizable) configurations. We call the corresponding subspaces the \textit{boundary components} of $\Sigma_{2l\pi}$. Therefore, a boundary component is a totally reducible facet $T_{ijk}(m,l-m)$ (see Definition \ref{Tijk}). Note that these components may be included in the boundary of $\cT(G)^3$.

\begin{Rema}
	The cells $C_{0}$ and $C_{2\pi}$ of Proposition \ref{U2} have disappeared in Theorem \ref{sphred} because the inequality $\sigma_{222}\leq0$ implies $\alpha_2=\beta_2=\gamma_2=0$. Therefore, these two cells are entirely contained in the boundary of $\cT(G)^3$.
\end{Rema}

\begin{Rema}\label{symmetrysphwalls}
	Recall the symmetry $\psi$ from Remark \ref{symmetry}. We have
	$$\forall\tau\in\cT(G)^3, \ S(\psi(\tau))=12\pi-S(\tau).$$
	Therefore, using the action of $\psi$ on the $\sigma_{ijk}$ mentioned in Remark \ref{symtotred}, one can easily check from the formulas of Theorem \ref{sphred}  that 
	$$\psi(\Sigma_{4\pi})=\Sigma_{8\pi} \ \mbox{and} \ \psi(\Sigma_{6\pi})=\Sigma_{6\pi}.$$
\end{Rema}

\subsubsection{Hyperbolic reducible walls}

\begin{Defi}
	Let $\rho\in\operatorname{Hom}_\mathrm{ell}(\Gamma,G)$. We say that this representation is \textit{hyperbolic reducible} if $\rho(\texttt{a}),\rho(\texttt{b}),\rho(\texttt{c})$ stabilize the same complex line of $\HH^2_\CC$. We denote:
	\begin{itemize}
		\item $\operatorname{Hom}_\mathrm{ell}^\mathrm{HR}(\Gamma,G)$ the subset of hyperbolic reducible representations;
		\item $\cP(G)^\mathrm{HR}$ the set of triples of angle pairs that come from hyperbolic reducible representations. More precisely, $\cP(G)^\mathrm{HR}=\overline{c}\big(\operatorname{Hom}_\mathrm{ell}^\mathrm{HR}(\Gamma,G)\big)\subset \cP(G)^\mathrm{red}$.
	\end{itemize}
\end{Defi}

\begin{Lemm}\label{hypred}
	Let $\rho\in\operatorname{Hom}_\mathrm{ell}^\mathrm{HR}(\Gamma,G)$ be a hyperbolic reducible representation. Then up to global conjugation, \emph{$\rho(\texttt{a}),\rho(\texttt{b}),\rho(\texttt{c})$} admit lifts to $\U(2,1)$ of the form
	\begin{equation}\label{hypwalls}
	\begin{pmatrix}
	e^{i\alpha_i} & 0 \\
	0 & A'
	\end{pmatrix}, \ \begin{pmatrix}
	e^{i\beta_j} & 0 \\
	0 & B'
	\end{pmatrix}, \ \begin{pmatrix}
	e^{i\gamma_k} & 0 \\
	0 & C'
	\end{pmatrix}
	\end{equation}
	with $A',B',C'\in\U(1,1)$ respectively conjugate to
	\begin{equation}\label{hypwalls2}
	\begin{pmatrix}
	e^{i\alpha_{\bar{i}}} & 0 \\
	0 & 1
	\end{pmatrix}, \ \begin{pmatrix}
	e^{i\beta_{\bar{j}}} & 0 \\
	0 & 1
	\end{pmatrix}, \ \begin{pmatrix}
	e^{i\gamma_{\bar{k}}} & 0 \\
	0 & 1
	\end{pmatrix}.
	\end{equation}
\end{Lemm}

\begin{proof}
	This is a direct consequence of the fact that the stabilizer of a complex line in $\HH^2_\CC$ is $\mathbb{P}(\U(1)\times\U(1,1))\cong\U(1)\times\pu(1,1)$ (see Proposition \ref{CP1stab}).
\end{proof}

\begin{Rema}
	The isolated eigenvalues in (\ref{hypwalls}) do not necessarily correspond to the greatest of the angles. This is why we leave the indices $i,j,k$ free in Lemma \ref{hypred}.
\end{Rema}

Consequently, the explicit solution of the $\pu(1,1)$-Horn problem described in Proposition \ref{PU11} can be transposed in this setting. However, we will obtain more walls than in the spherical reducible case, due to the various combinations of eigenvalues that can arise.

We now make Proposition \ref{Pred} more explicit in the $n=2$ case.

\begin{Lemm}\label{Hijk}
	Let $((\alpha_1,\alpha_2),(\beta_1,\beta_2),(\gamma_1,\gamma_2))\in \cT(G)^3$. Let $\sigma_{ijk}$ be the linear form of Definition \ref{linearforms}. If this triple of angle pairs belongs to $\cP(G)^\mathrm{HR}$, then there exist $i,j,k\in\{1,2\}$ such that
	$$2\sigma_{ijk}-\sigma_{\bar{i}\bar{j}\bar{k}}\equiv0 \ [2\pi].$$ 
\end{Lemm}	

\begin{proof}
	According to Lemma \ref{hypred}, we have
	$$\begin{pmatrix}
	e^{i\alpha_i} & 0 \\
	0 & A'
	\end{pmatrix}\begin{pmatrix}
	e^{i\beta_j} & 0 \\
	0 & B'
	\end{pmatrix}\begin{pmatrix}
	e^{i\gamma_k} & 0 \\
	0 & C'
	\end{pmatrix}=\begin{pmatrix}
	\lambda & 0 \\
	0 & \lambda\operatorname{I}_2
	\end{pmatrix}$$
	for some $\lambda\in\CC$, with $A',B',C'\in\U(1,1)$ respectively conjugate to
	$$\begin{pmatrix}
	e^{i\alpha_{\bar{i}}} & 0 \\
	0 & 1
	\end{pmatrix}, \ \begin{pmatrix}
	e^{i\beta_{\bar{j}}} & 0 \\
	0 & 1
	\end{pmatrix}, \ \begin{pmatrix}
	e^{i\gamma_{\bar{k}}} & 0 \\
	0 & 1
	\end{pmatrix}.$$
	Computing the determinants in each block (compare with \eqref{hypeq}), we obtain $e^{i\sigma_{ijk}}=\lambda$ and $e^{i\sigma_{\bar{i}\bar{j}\bar{k}}}=\lambda^2$. Consequently, $2\sigma_{ijk}-\sigma_{\bar{i}\bar{j}\bar{k}}\equiv0 \ [2\pi].$
\end{proof}

\begin{Defi}
	Let $(i,j,k)\in\{1,2\}^3$. We define the linear form $H_{ijk}:(\RR^2)^3\to\RR$ by: $\forall((\alpha_1,\alpha_2),(\beta_1,\beta_2),(\gamma_1,\gamma_2))\in (\RR^2)^3$,
	\begin{equation*}
	\begin{aligned}
	H_{ijk}((\alpha_1,\alpha_2),(\beta_1,\beta_2),(\gamma_1,\gamma_2))&=2(\alpha_i+\beta_j+\gamma_k)-(\alpha_{\bar{i}}+\beta_{\bar{j}}+\gamma_{\bar{k}})\\
	&=2\sigma_{ijk}-\sigma_{\bar{i}\bar{j}\bar{k}}.
	\end{aligned}
	\end{equation*}
	In practice, we will only consider its restriction to $\cT(G)^3$. Most of the time, we will write $H_{ijk}$ instead of $H_{ijk}((\alpha_1,\alpha_2),(\beta_1,\beta_2),(\gamma_1,\gamma_2))$.
\end{Defi}

We have the following relations between the various linear forms that we defined. The proof is straightforward using the definitions of $S$ and $H_{ijk}$, together with the ordering of the angles ($\alpha_2<\alpha_1$ etc.)

\begin{Lemm}\label{Hijkproperties}
	For every $(i,j,k)\in\{1,2\}^3$, the linear forms $H_{ijk}$, $H_{\bar{i}\bar{j}\bar{k}}$ and $S$ satisfy the following properties:
	\begin{enumerate}[i)]
		\item $H_{ijk}+H_{\bar{i}\bar{j}\bar{k}}=S$;
		\item $H_{2jk}<H_{1jk},$ $H_{i2k}<H_{i1k},$ $H_{ij2}<H_{ij1}$.
	\end{enumerate}
\end{Lemm}

We now give the explicit description of the set of hyperbolic reducible walls $\cP(G)^\mathrm{HR}$. We see here that the number of reducible walls greatly increases between $\pu(1,1)$ and $\pu(2,1)$. As a consequence, we will see later on that the complexity of the arrangements between the walls increases as well. This seems to be one of the major obstacles to explicit descriptions for $\pu(n,1)$ with $n\geq2$.

\begin{Theo}\label{hyperbolicreduciblewalls}
	The set of hyperbolic reducible solutions $\cP(G)^\mathrm{HR}$ is the reunion of the following twenty-four truncated hyperplanes of $\cT(G)^3$:
	\begin{enumerate}
		\item $\{H_{ijk}=-4\pi\}$ for $(i,j,k)=(2,2,2)$;
		\item $\{H_{ijk}=0,\sigma_{\bar{i}\bar{j}\bar{k}}\geq4\pi\}$ for $(i,j,k)\in\{(2,2,2),(2,2,1),(2,1,2),(1,2,2)\}$;
		\item $\{H_{ijk}=2\pi,\sigma_{\bar{i}\bar{j}\bar{k}}\leq2\pi\}$ for \\  $(i,j,k)\in\{(2,2,1),(2,1,2),(1,2,2),(1,1,2),(1,2,1),(2,1,1),(1,1,1)\}$;
		\item $\{H_{ijk}=4\pi,\sigma_{\bar{i}\bar{j}\bar{k}}\geq4\pi\}$ for \\  $(i,j,k)\in\{(2,2,2),(2,2,1),(2,1,2),(1,2,2),(1,1,2),(1,2,1),(2,1,1)\}$;
		\item $\{H_{ijk}=6\pi,\sigma_{\bar{i}\bar{j}\bar{k}}\leq2\pi\}$ for $(i,j,k)\in\{(1,1,1),(1,1,2),(1,2,1),(2,1,1)\}$;
		\item $\{H_{ijk}=10\pi\}$ for $(i,j,k)=(1,1,1)$.
	\end{enumerate}
\end{Theo}

\begin{Defi}
	We call the twenty-four truncated hyperplanes defined in Theorem \ref{hyperbolicreduciblewalls} the \textit{hyperbolic reducible walls}.
\end{Defi}

\begin{Rema}\label{symmetryhypwalls}
	Recall the symmetry $\psi$ defined in Remark \ref{symmetry}. We have
	$$\forall\tau\in\cT(G)^3, \ H_{ijk}(\psi(\tau))=6\pi-H_{\bar{k}\bar{j}\bar{i}}(\tau).$$
	Therefore, using the action of $\psi$ on the $\sigma_{ijk}$ mentioned in Remark \ref{symtotred}, one can easily check from the formulas of Theorem \ref{hyperbolicreduciblewalls} that:
	\begin{itemize}
		\item $\psi$ exchanges the walls of Items 1 and 6;
		\item $\psi$ exchanges the walls of Items 2 and 5;
		\item $\psi$ exchanges the walls of Items 3 and 4.
	\end{itemize}
\end{Rema}

We now prove Theorem \ref{hyperbolicreduciblewalls}.

\begin{proof}
	Let $\tau\in\cT(G)^3$ be the triple of angle pairs of a hyperbolic reducible representation. According to Lemma \ref{Hijk}, we must have $H_{ijk}(\tau)\equiv0 \ [2\pi]$. Considering that the angles all belong to $[0,2\pi[$, we obtain $$H_{ijk}(\tau)\in\{-6\pi,-4\pi,-2\pi,0,2\pi,4\pi,6\pi,8\pi,10\pi \}.$$ 
	However, if $H_{ijk}(\tau)=-6\pi$, the corresponding conjugacy classes are trivial. Therefore, we only need to study the cases $$H_{ijk}(\tau)\in\{-4\pi,-2\pi,0,2\pi,4\pi,6\pi,8\pi,10\pi \}.$$
	The two key arguments are the following. They rely on the explicit solution of Horn's problem in $\pu(1,1)$ described in Proposition \ref{PU11}.
	\begin{itemize}
		\item Considering the matrices given in Lemma \ref{hypred} together with Proposition \ref{PU11}, we must have 
		$$\sigma_{\bar{i}\bar{j}\bar{k}}(\tau)\leq2\pi \ \mbox{or} \ \sigma_{\bar{i}\bar{j}\bar{k}}(\tau)\geq4\pi.$$
		\item If $\sigma_{\bar{i}\bar{j}\bar{k}}(\tau)\in\{2\pi,4\pi\}$, then by Proposition \ref{PU11} the representation is totally reducible. Consequently, $\tau$ belongs to a totally reducible facet, or equivalently, to the boundary of one of the spherical reducible walls $\Sigma_{2l\pi}$ (see Theorem \ref{sphred}). 
	\end{itemize}
	We now show that each piece of hyperplane defined in Theorem \ref{hyperbolicreduciblewalls} is an embedding in $\cT(G)^3$ of one of the two connected components $\cP(\pu(1,1))_1$ (the unique $1$-cell that is full) and $\cP(\pu(1,1))_{-1}$ (the unique $-1$-cell that is full) of $\cP(\pu(1,1))$ (see Proposition \ref{PU11}), and that there are no other such embeddings.
	
	\begin{itemize}
		\item Let us prove Item 1. Assume that $H_{ijk}(\tau)=2\sigma_{ijk}(\tau)-\sigma_{\bar{i}\bar{j}\bar{k}}(\tau)=-4\pi$. Then, the quantity $\sigma_{\bar{i}\bar{j}\bar{k}}(\tau)$ belongs to $[4\pi,6\pi]$. Indeed, by definition we have $0\leq\sigma_{\bar{i}\bar{j}\bar{k}}(\tau)\leq6\pi$; and in this case, $\sigma_{\bar{i}\bar{j}\bar{k}}(\tau)=2\sigma_{ijk}(\tau)-H_{ijk}(\tau)\geq4\pi$.
		
		Assume that $\sigma_{\bar{i}\bar{j}\bar{k}}(\tau)=4\pi$. We have $$\sigma_{ijk}(\tau)=\frac{1}{2}(\sigma_{\bar{i}\bar{j}\bar{k}}(\tau)-4\pi)=0.$$ 
		Thus, $\tau$ belongs to the totally reducible facet $T_{ijk}(0,2)$. Note that in this case, we have $S(\tau)=4\pi$. According to Theorem \ref{sphred}, $\{\sigma_{\bar{i}\bar{j}\bar{k}}=4\pi\}$ is a boundary component of $\Sigma_{4\pi}$ only for $(i,j,k)=(2,2,2)$.  
		
		Therefore, the hyperplane $\{H_{222}=-4\pi\}$ is an embedding in $\cT(G)^3$ of the connected component $\cP(\pu(1,1))_1\subset\cP(\pu(1,1))$. This yields Item 1.
		
		\item If $H_{ijk}(\tau)=-2\pi$, then by similar arguments the quantity $\sigma_{\bar{i}\bar{j}\bar{k}}(\tau)$ belongs to $[2\pi,6\pi]$. If $\sigma_{\bar{i}\bar{j}\bar{k}}(\tau)=4\pi$, we have $\sigma_{ijk}(\tau)=\pi$. Since $\sigma_{ijk}(\tau)$ is an odd multiple of $\pi$, $\tau$ does not belong to a totally reducible facet (see Definition \ref{Tijk}). Therefore the hyperplane $\{H_{ijk}=-2\pi\}$ does not contain any embedding of a connected component of $\cP(\pu(1,1))$.
		
		\item We now prove Item 2. Assume that $H_{ijk}(\tau)=0$. This implies that $\sigma_{\bar{i}\bar{j}\bar{k}}(\tau)$ belongs to $[0,6\pi]$. 
		
		Assume first that $\sigma_{\bar{i}\bar{j}\bar{k}}(\tau)=4\pi$; this implies $\sigma_{ijk}(\tau)=2\pi$. Thus, $\tau$ belongs to the totally reducible facet $T_{ijk}(1,2)$. Note that in this case, we have  $S(\tau)=6\pi$; according to Theorem \ref{sphred}, $\{\sigma_{\bar{i}\bar{j}\bar{k}}=4\pi\}$ is a boundary component of $\Sigma_{6\pi}$ only for $$(i,j,k)\in\{(2,2,2),(2,2,1),(2,1,2),(1,2,2)\}.$$  
		Therefore, each of the four truncated hyperplanes $\{H_{ijk}=0,\sigma_{\bar{i}\bar{j}\bar{k}}\geq4\pi\}$ for the above values of $(i,j,k)$ is an embedding of the connected component $\cP(\pu(1,1))_1\subset\cP(\pu(1,1))$. This yields Item 2.
		
		Assume now that $\sigma_{\bar{i}\bar{j}\bar{k}}(\tau)=2\pi$. This implies $\sigma_{ijk}(\tau)=\pi$; thus, $\tau$ does not belong to a totally reducible facet, and therefore $\tau$ cannot be the triple of angle pairs of a hyperbolic reducible representation. 
		
		\item Let us prove Item 3. Assume that $H_{ijk}(\tau)=2\pi$. This implies that $\sigma_{\bar{i}\bar{j}\bar{k}}(\tau)$ belongs to $[0,6\pi]$. 
		
		First assume that $\sigma_{\bar{i}\bar{j}\bar{k}}(\tau)=2\pi$. This implies $\sigma_{ijk}(\tau)=2\pi$; thus, $\tau$ belongs to the totally reducible facet $T_{ijk}(1,1)$. Note that in this case, we have  $S(\tau)=4\pi$; according to Theorem \ref{sphred}, $\{\sigma_{\bar{i}\bar{j}\bar{k}}=2\pi\}$ is a boundary component of $\Sigma_{4\pi}$ for $$(i,j,k)\in\{(2,2,1),(2,1,2),(1,2,2),(1,1,2),(1,2,1),(2,1,1),(1,1,1)\}.$$
		Therefore, each of the seven truncated hyperplanes $\{H_{ijk}=0,\sigma_{\bar{i}\bar{j}\bar{k}}\leq2\pi\}$ for the above values of $(i,j,k)$ is an embedding of the connected component $\cP(\pu(1,1))_{-1}\subset\cP(\pu(1,1))$. This yields Item 3.
		
		Now, assume that $\sigma_{\bar{i}\bar{j}\bar{k}}(\tau)=4\pi$. This implies $\sigma_{ijk}(\tau)=3\pi$; thus, $\tau$ does not belong to a totally reducible facet, and therefore $\tau$ cannot be the triple of angle pairs of a hyperbolic reducible representation. 
	\end{itemize}
	We obtain Items 4, 5, 6 by symmetry of Items 3, 2, 1 (see Remark \ref{symmetryhypwalls}).
\end{proof}

As a byproduct of the proof of Theorem \ref{hyperbolicreduciblewalls}, we obtain a complete description of the totally reducible facets as intersection of one spherical reducible wall and two hyperbolic reducible walls (compare with Paupert \cite{Pau}):

\begin{Prop}\label{totredintersection}
	Let $(i,j,k)\in\{1,2\}^3$, $m,n\in\{0,1,2,3\}$. We have
	$$T_{ijk}(m,n)=\Sigma_{2(m+n)\pi}\cap\{H_{ijk}=2(2m-n)\pi \}\cap\{H_{\bar{i}\bar{j}\bar{k}}=2(2n-m)\pi \}.$$
	Moreover, $T_{ijk}(m,n)$ is contained in no other reducible wall.
\end{Prop}

\subsection{Cells}

\subsubsection{Statement of the result in $\pu(2,1)$}

We now define the explicit polytopes that were announced in Theorem \ref{solution}.

\begin{Defi}\label{polytope}
	Define the following three subsets of $\cT(G)^3$:
	\begin{equation*}
	\begin{aligned}
	P_{4\pi}=&\{H_{222}<-4\pi\}\cup\\
	&\Big(\{S<4\pi\}\cap\{H_{122}<2\pi\}\cap\{H_{212}<2\pi\}\cap\{H_{221}<2\pi\}\cap\{H_{111}>2\pi\}\Big),\\
	P_{6\pi}=&\{H_{222}<0\}\cap\{H_{111}>6\pi\},\\
	P_{8\pi}=&\{H_{111}>10\pi\}\cup\\
	&\Big(\{S>8\pi\}\cap\{H_{211}>4\pi\}\cap\{H_{121}>4\pi\}\cap\{H_{112}>4\pi\}\cap\{H_{222}<4\pi\}\Big).
	\end{aligned}
	\end{equation*}
	We denote by $\overline{P_{2k\pi}}$, respectively $\partial P_{2k\pi}$, the closure, respectively the boundary, of these subsets. 
\end{Defi}

In the next sections, we prove the following explicit version of Theorem \ref{solution}:

\begin{Theo}\label{solutionpu}
	We have $\cP(G)=\overline{P_{4\pi}}\cup \overline{P_{6\pi}}\cup \overline{P_{8\pi}}$.
\end{Theo}

\begin{Rema}\label{symmetrypolytopes}
	Recall the symmetry of the problem $\psi$ mentioned in Remark \ref{symmetry}. From its action on the spherical and hyperbolic reducible walls described in Remarks \ref{symmetrysphwalls} and \ref{symmetryhypwalls}, one can easily check from the formulas of Definition \ref{polytope} that $\psi$ exchanges $P_{4\pi}$ and $P_{8\pi}$, and sends $P_{6\pi}$ to itself.
\end{Rema}

Let us relate Theorem \ref{solutionpu} to Paupert's work  in \cite{Pau} before proceeding with the proof.

\begin{Rema}\label{paupert}
	Theorem 1.2 in \cite{Pau} states that the momentum map $\mu_{\cC_1,\cC_2}$ is surjective if and only if the corresponding angle pairs $(\alpha_1,\alpha_2),(\beta_1,\beta_2)$ satisfy 
	\begin{equation}\label{musurj}
	\left \{
	\begin{array}{ccc}
	2(\alpha_1+\beta_1)-(\alpha_2+\beta_2) & \geq & 6\pi \\
	2(\alpha_2+\beta_2)-(\alpha_1+\beta_1) & \leq & -2\pi
	\end{array}
	\right.
	\end{equation}
	(we have recalled the definition of the momentum map in the Introduction). Comparing with his notations, we have set $\theta_1=2\pi-\beta_2,\theta_2=2\pi-\beta_1,\theta_3=2\pi-\alpha_2,\theta_4=2\pi-\alpha_1$ (because he considers the product $AB$ whereas we consider $C=(AB)^{-1}$). 
	
	In our setting, this result translates as follows. Fix $(\alpha_1^*,\alpha_2^*),(\beta_1^*,\beta_2^*)\in\cT(G)$, and for $(\gamma_1,\gamma_2)\in\cT(G)$ define $$\tau^*(\gamma_1,\gamma_2)=((\alpha_1^*,\alpha_2^*),(\beta_1^*,\beta_2^*),(\gamma_1,\gamma_2))\in\cT(G)^3.$$ Then, we have $\tau^*(\gamma_1,\gamma_2)\in\cP(G)$ for any $(\gamma_1,\gamma_2)$ if and only if $(\alpha_1^*,\alpha_2^*),(\beta_1^*,\beta_2^*)$ satisfy \eqref{musurj}. We now give a quick proof of this result using Theorem \ref{solutionpu}.
	
	Let $(\alpha_1^*,\alpha_2^*),(\beta_1^*,\beta_2^*)\in\cT(G)$ be fixed angle pairs that satisfy \eqref{musurj}. Note that for any $(\gamma_1,\gamma_2)\in\cT(G)$, we have $2\gamma_1-\gamma_2\geq0$ and $2\gamma_2-\gamma_1\leq\gamma_1\leq2\pi$. Therefore, if \eqref{musurj} is satisfied, then $H_{111}(\tau^*(\gamma_1,\gamma_2))\geq6\pi$ and $H_{222}(\tau^*(\gamma_1,\gamma_2))\leq0$, which implies that $\tau^*(\gamma_1,\gamma_2)\in P_{6\pi}$ for any $(\gamma_1,\gamma_2)\in\cT(G)$. 
	
	Conversely, let $(\alpha_1^*,\alpha_2^*),(\beta_1^*,\beta_2^*)\in\cT(G)$ be fixed and assume that we have $\tau^*(\gamma_1,\gamma_2)\in\cP(G)$ for any $(\gamma_1,\gamma_2)$.  If $\tau^*(\gamma_1,\gamma_2)\in P_{6\pi}$, then we have $H_{111}(\tau^*(\gamma_1,\gamma_2))\geq6\pi$ and $H_{222}(\tau^*(\gamma_1,\gamma_2))\leq0$; taking $\gamma_1=\gamma_2=0$ and $\gamma_1=\gamma_2=2\pi$ implies \eqref{musurj}. If $\tau^*(\gamma_1,\gamma_2)\in P_{4\pi}$, then $S(\tau^*(\gamma_1,\gamma_2))\leq4\pi$; taking $\gamma_1=\gamma_2=2\pi$ implies $\alpha_1^*=\alpha_2^*=\beta_1^*=\beta_2^*=0$, which is the trivial solution. Using the symmetry (Remark \ref{symmetrypolytopes}), we obtain the same result if $\tau^*(\gamma_1,\gamma_2)\in P_{8\pi}$. This concludes the proof of the converse implication. 
\end{Rema}

\subsubsection{Cells around a totally reducible facet}

As a first tool for filling cells, we will use the following result of convexity, identified by Paupert in \cite{Pau}.

\begin{Prop}\label{totredfacet}
	\begin{itemize}
		\item Let $T$ be a totally reducible facet. Then, there exist exactly three cells $C_1,C_2,C_3$ such that $T=\partial C_1\cap\partial C_2\cap \partial C_3$.
		\item Consider a totally reducible facet $T$ and let $S,H,\bar{H}$ be the three reducible walls such that $T=S\cap H\cap \bar{H}$. Call $C_1,C_2,C_3$ the three cells that contain $T$ in their boundary. Let $C\in \{C_1,C_2,C_3\}$. If $C$ meets the convex hull of $S\cup H\cup \bar{H}$, then $C$ is full.
	\end{itemize}
\end{Prop}

\begin{proof}
	The first point is a direct consequence of Proposition \ref{totredintersection}. We obtain the second point as a corollary of Proposition 2.7 in \cite{Pau}. This result states that, when $\cC_1,\cC_2$ are fixed, the image of the momentum map $\mu_{\cC_1,\cC_2}$ (see the Introduction) contains the local convex hull of the reducible walls meeting at a totally reducible vertex. This comes from the fact that the image of $\mu_{\cC_1,\cC_2}$ around a totally reducible vertex is locally a convex cone bounded by the reducible walls. In our setting, if the cell $C$ meets the convex hull of $S\cup H\cup \bar{H}$, then by choosing any classes $\cC_1,\cC_2,\cC_3$ such that $i^3(\cC_1,\cC_2,\cC_3)\in C$ and using Paupert's result, we obtain that $i(\cC_3)$ belongs to the image of $\mu_{\cC_1,\cC_2}$. This implies that $i^3(\cC_1,\cC_2,\cC_3)\in\cP(G)$, and therefore $C$ is full by Theorem \ref{emptyfulltheorem}.
\end{proof}

In order to use this argument to fill the cells, we need a more precise description of the totally reducible intersections. We define now four situations that will turn out to be the only ones that can appear around a totally reducible facet.

\begin{Defi}
	Let $T$ be a totally reducible facet and let $\Sigma_{2n_s\pi}$, $\{H=2n_h\pi\}$, and $\{\bar{H}=2n_{\bar{h}}\pi\}$ be the three reducible walls whose intersection is $T$. We will say that:
	\begin{itemize}
		\item $T$ is of \textbf{type 1} if one of the hyperbolic reducible walls is contained in the convex hull of the other two reducible walls, and the two hyperbolic reducible walls are contained in the negative half-space of $S-2n_s\pi$. This means that $\forall \tau_h\in\{H=2n_h\pi\}\cap P_{2n_s\pi}$, $\forall \tau_{\bar{h}}\in\{\bar{H}=2n_{\bar{h}}\pi\}\cap P_{2n_s\pi}$, 
		$$S(\tau_h),S(\tau_{\bar{h}})\leq2n_s\pi$$ (see Figure \ref{type1});
		\item $T$ is of \textbf{type 2} if one of the hyperbolic reducible walls is contained in the convex hull of the other two reducible walls, and the two hyperbolic reducible walls are contained in the positive half-space of $S-2n_s\pi$. This means that $\forall \tau_h\in\{H=2n_h\pi\}\cap P_{2n_s\pi}$, $\forall \tau_{\bar{h}}\in\{\bar{H}=2n_{\bar{h}}\pi\}\cap P_{2n_s\pi}$, 
		$$S(\tau_h),S(\tau_{\bar{h}})\geq2n_s\pi$$ (see Figure \ref{type2});
		\item $T$ is of \textbf{type 3} if the spherical reducible wall is contained in the convex hull of the two hyperbolic reducible walls. This leaves two possible situations: $\forall \tau_s\in\Sigma_{2n_s\pi}$, $\forall \tau_h\in\{H=2n_h\pi\}\cap P_{2n_s\pi}$, $\forall \tau_{\bar{h}}\in\{\bar{H}=2n_{\bar{h}}\pi\}\cap P_{2n_s\pi}$, either $$H(\tau_s),H(\tau_{\bar{h}})\leq2n_h\pi \ \mbox{and} \ \bar{H}(\tau_s),\bar{H}(\tau_h)\geq2n_{\bar{h}}\pi$$ or 
		$$H(\tau_s),H(\tau_{\bar{h}})\geq2n_h\pi \ \mbox{and} \  \bar{H}(\tau_s),\bar{H}(\tau_h)\leq2n_{\bar{h}}\pi$$ (see Figure \ref{type3});
		\item $T$ is of \textbf{type 4} if none of the reducible walls is contained in the convex hull of the other two. This leaves two possible situations: $\forall \tau_s\in\Sigma_{2n_s\pi}$, $\forall \tau_h\in\{H=2n_h\pi\}\cap P_{2n_s\pi}$, $\forall \tau_{\bar{h}}\in\{\bar{H}=2n_{\bar{h}}\pi\}\cap P_{2n_s\pi}$, either $$H(\tau_s)\leq2n_h\pi,H(\tau_{\bar{h}})\geq2n_h\pi \ \mbox{and} \  \bar{H}(\tau_s)\geq2n_{\bar{h}}\pi,\bar{H}(\tau_h)\leq2n_{\bar{h}}\pi$$ or $$H(\tau_s)\geq2n_h\pi,H(\tau_{\bar{h}})\leq2n_h\pi \ \mbox{and} \  \bar{H}(\tau_s)\leq2n_{\bar{h}}\pi,\bar{H}(\tau_h)\geq2n_{\bar{h}}\pi$$ (see Figure \ref{type4}).
	\end{itemize}
\end{Defi}

\begin{figure}
	\centering
	\begin{adjustbox}{minipage=1.1\linewidth,scale=0.8}
		\begin{subfigure}{0.5\textwidth}
			\centering
			\begin{tikzpicture}
			\useasboundingbox (-3,-1) rectangle (2,3);
			\draw [red](2,0) node[below]{\scriptsize$\{S=2n_s\pi\}$} ;
			\draw [blue](-1/2,0) node[below]{\scriptsize$\{\bar{H}=2n_{\bar{h}}\pi\}$} ;
			\draw [blue](-2,1) node[below]{\scriptsize$\{H=2n_h\pi\}$} ;
			\draw [red](2,0) -- (0,2);
			\draw [blue](-1/2,0) -- (0,2);
			\draw [blue](-2,1) -- (0,2);
			\end{tikzpicture}
			\caption{Totally reducible intersection of \textbf{type 1} \label{type1}}
		\end{subfigure}
		\hfill
		\begin{subfigure}{0.5\textwidth}
			\centering
			\begin{tikzpicture}
			\useasboundingbox (-1.5,0) rectangle (2,4);
			\draw [red](-1/2,1/2) node[right]{\scriptsize$\{S=2n_s\pi\}$} ;
			\draw [blue](-1/2,2.75) node[right]{\scriptsize$\{\bar{H}=2n_{\bar{h}}\pi\}$} ;
			\draw [blue](-1,4) node[right]{\scriptsize$\{H=2n_h\pi\}$} ;
			\draw [red](-1/2,1/2) -- (-2,2);
			\draw [blue](-1,4) -- (-2,2);
			\draw [blue](-1/2,2.75) -- (-2,2);
			\end{tikzpicture}
			\caption{Totally reducible intersection of \textbf{type 2} \label{type2}}
		\end{subfigure}
		\vfill
		\begin{subfigure}{0.5\textwidth}
			\centering
			\begin{tikzpicture}
			\useasboundingbox (-1.5,-1) rectangle (2,4);
			\draw [red](2,0) node[below]{\scriptsize$\{S=2n_s\pi\}$} ;
			\draw [blue](-1/2,1/4) node[below]{\scriptsize$\{\bar{H}=2n_{\bar{h}}\pi\}$} ;
			\draw [blue](1,2.25) node[right]{\scriptsize$\{H=2n_h\pi\}$} ;
			\draw [red](2,0) -- (0,2);
			\draw [blue](-1/2,1/4) -- (0,2);
			\draw [blue](3/2,11/4) -- (0,2);
			\end{tikzpicture}
			\caption{Totally reducible intersection of \textbf{type 3} \label{type3}}
		\end{subfigure}	
		\hfill
		\begin{subfigure}{0.5\textwidth}
			\centering
			\begin{tikzpicture}
			\useasboundingbox (-1.5,-1) rectangle (2,4);
			\draw [red](3/2,1/2) node[below]{\scriptsize$\{S=2n_s\pi\}$} ;
			\draw [blue](1/2+0.1,3) node[right]{\scriptsize$\{\bar{H}=2n_{\bar{h}}\pi\}$} ;
			\draw [blue](-3/2,5/4) node[below]{\scriptsize$\{H=2n_h\pi\}$} ;
			\draw [red](3/2,1/2) -- (0,2);
			\draw [blue](3/4,7/2) -- (0,2);
			\draw [blue](-3/2,5/4) -- (0,2);
			\end{tikzpicture}
			\caption{Totally reducible intersection of \textbf{type 4} \label{type4}}
		\end{subfigure}
	\end{adjustbox}
	\caption{The four combinatorial possibilities for a totally reducible intersection. \label{totred}}
\end{figure}
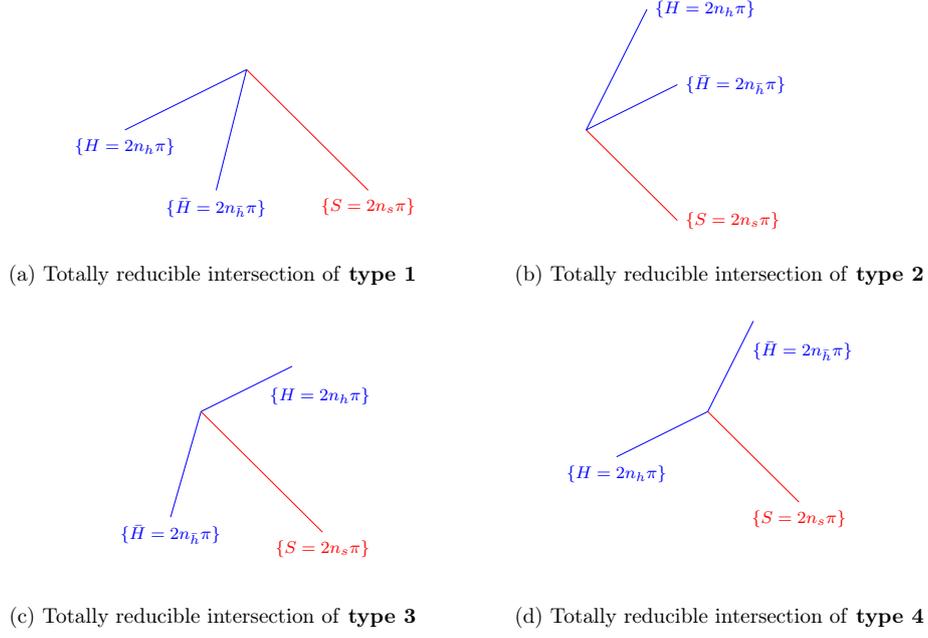

From Proposition \ref{totredfacet} we obtain directly:

\begin{Coro}\label{totallyreduciblefilling}
	Let $T$ be a totally reducible facet and let $C_1,C_2,C_3$ be the three cells such that $T=\partial C_1\cap\partial C_2\cap \partial C_3$. Then:
	\begin{itemize}
		\item if $T$ is of type 1, 2 or 3, two cells among $C_1,C_2,C_3$ are full;
		\item if $T$ is of type 4, the three cells $C_1,C_2,C_3$ are full.
	\end{itemize}
\end{Coro}

We now classify the totally reducible facets according to their type. We attract the reader's attention to the fact that we slightly change the notation for $T_{ijk}$ (compare with Proposition \ref{totredintersection}).

\begin{Prop}\label{totallyreduciblefacets}
	Let 
	$$T_{ijk}(n_h,n_{\bar{h}},n_s)=\{H_{ijk}=2n_h\pi\}\cap\{H_{\bar{i}\bar{j}\bar{k}}=2n_{\bar{h}}\pi\}\cap\Sigma_{2n_s\pi}$$ be a totally reducible facet. Then:
	\begin{itemize}
		\item $T_{ijk}(1,1,2)$ is of \textbf{type 1} for $(i,j,k)\in\{(1,2,2),(2,1,2),(2,2,1)\}$;
		\item $T_{ijk}(2,2,4)$ is of \textbf{type 2} for $(i,j,k)\in\{(2,1,1),(1,2,1),(1,1,2)\}$;
		\item $T_{ijk}(3,0,3)$ is of \textbf{type 3} for $(i,j,k)=(1,1,1)$;
		\item $T_{ijk}(3,0,3)$ is of \textbf{type 4} for $(i,j,k)\in\{(2,1,1),(1,2,1),(1,1,2)\}$.
	\end{itemize}
\end{Prop}

\begin{proof}
	We prove this result by a precise case-by-case analysis. Recall that $H_{ijk}=2\sigma_{ijk}-\sigma_{\bar{i}\bar{j}\bar{k}}$, so $H_{ijk}=2n_h\pi$ implies
	\begin{equation}\label{ineg}
	\sigma_{ijk}=\frac{1}{2}(2n_h\pi+\sigma_{\bar{i}\bar{j}\bar{k}}).
	\end{equation}
	Similarly, $S=\sigma_{ijk}+\sigma_{\bar{i}\bar{j}\bar{k}}$, and thus $S=2n_s\pi$ implies
	\begin{equation}\label{ineg2}
	\sigma_{ijk}=2n_s\pi-\sigma_{\bar{i}\bar{j}\bar{k}}.
	\end{equation}
	\begin{itemize}
		\item Assume first that $(i,j,k)\in\{(1,2,2),(2,1,2),(2,2,1)\}$. We are going to show that $$T_{ijk}(1,1,2)=\{H_{ijk}=2\pi\}\cap\{H_{\bar{i}\bar{j}\bar{k}}=2\pi\}\cap\Sigma_{4\pi}$$ is of type 1. Let $\tau_h\in\{H_{ijk}=2\pi,\sigma_{\bar{i}\bar{j}\bar{k}}\leq2\pi\}$. By (\ref{ineg}) we have $$\sigma_{ijk}(\tau_h)\leq\frac{1}{2}(2\pi+2\pi)=2\pi,$$
		so $S(\tau_h)\leq4\pi$. Let $\tau_{\bar{h}}\in\{H_{\bar{i}\bar{j}\bar{k}}=2\pi,\sigma_{ijk}\leq2\pi\}$. By (\ref{ineg}) we have $$\sigma_{\bar{i}\bar{j}\bar{k}}(\tau_{\bar{h}})\leq\frac{1}{2}(2\pi+2\pi)=2\pi,$$
		so $S(\tau_{\bar{h}})\leq4\pi$.
		
		\item Consider the case where $(i,j,k)\in\{(2,1,1),(1,2,1),(1,1,2)\}$. The fact that $$T_{ijk}(2,2,4)=\{H_{ijk}=4\pi\}\cap\{H_{\bar{i}\bar{j}\bar{k}}=4\pi\}\cap\Sigma_{8\pi}$$ is of type 2 is obtained by applying the symmetry $\psi$ to the totally reducible facets of type 1 (see Remarks \ref{symtotred}, \ref{symmetrysphwalls}, \ref{symmetryhypwalls}).
		
		\item We show that $$T_{111}(3,0,3)=\{H_{111}=6\pi\}\cap\{H_{222}=0\}\cap\Sigma_{6\pi}$$ is of type 3. Let $$\tau_s\in\Sigma_{6\pi}=\{S=6\pi,\sigma_{222}\leq2\pi,\sigma_{112}\leq4\pi, \sigma_{121}\leq4\pi,\sigma_{211}\leq4\pi\},$$ $$\tau_h\in\{H_{111}=6\pi,\sigma_{222}\leq2\pi\}, \  \tau_{\bar{h}}\in\{H_{222}=0,\sigma_{111}\geq4\pi\}.$$ By (\ref{ineg}) and (\ref{ineg2}) we have
		\begin{equation*}
		\begin{aligned}
		H_{111}(\tau_s)&=2(6\pi-\sigma_{222}(\tau_s))-\sigma_{222}(\tau_s)=12\pi-3\sigma_{222}(\tau_s)\geq 12\pi-3\cdot2\pi=6\pi, \\
		H_{111}(\tau_{\bar{h}})&=2\sigma_{111}(\tau_{\bar{h}})-\frac{1}{2}\sigma_{111}(\tau_{\bar{h}})=\frac{3}{2}\sigma_{111}(\tau_{\bar{h}})\geq\frac{3}{2}\cdot4\pi=6\pi,\\
		H_{222}(\tau_s)&=2\sigma_{222}(\tau_s)-(6\pi-\sigma_{222}(\tau_s))=-6\pi+3\sigma_{222}(\tau_s)\leq-6\pi+3\cdot2\pi=0,\\
		H_{222}(\tau_h)&=2\sigma_{222}(\tau_h)-(3\pi+\frac{1}{2}\sigma_{222}(\tau_h))=\frac{3}{2}\sigma_{222}(\tau_h)-3\pi\leq\frac{3}{2}\cdot2\pi-3\pi=0.
		\end{aligned}
		\end{equation*}
		Note that the symmetry $\psi$ globally preserves this totally reducible facet.
		
		\item Let $(i,j,k)\in\{(2,1,1),(1,2,1),(1,1,2)\}$. We show that $$T_{ijk}(3,0,3)=\{H_{ijk}=6\pi\}\cap\{H_{\bar{i}\bar{j}\bar{k}}=0\}\cap\Sigma_{6\pi}$$ is of type 4. We can already notice that the symmetry $\psi$ leaves invariant this set of totally reducible facets. Let $$\tau_s\in\Sigma_{6\pi}=\{S=6\pi,\sigma_{222}\leq2\pi,\sigma_{112}\leq4\pi, \sigma_{121}\leq4\pi,\sigma_{211}\leq4\pi\},$$ $$\tau_h\in\{H_{ijk}=6\pi,\sigma_{\bar{i}\bar{j}\bar{k}}\leq2\pi\}, \ \tau_{\bar{h}}\in\{H_{\bar{i}\bar{j}\bar{k}}=0,\sigma_{ijk}\geq4\pi\}.$$ By (\ref{ineg}) and (\ref{ineg2}) we have
		\begin{equation*}
		\begin{aligned}
		H_{ijk}(\tau_s)&=2\sigma_{ijk}(\tau_s)-(6\pi-\sigma_{ijk}(\tau_s))=-6\pi+3\sigma_{ijk}(\tau_s)\leq-6\pi+3\cdot4\pi=6\pi,\\
		H_{ijk}(\tau_{\bar{h}})&=2\sigma_{ijk}(\tau_{\bar{h}})-\frac{1}{2}\sigma_{ijk}(\tau_{\bar{h}})=\frac{3}{2}\sigma_{ijk}(\tau_{\bar{h}})\geq\frac{3}{2}\cdot4\pi=6\pi,\\
		H_{\bar{i}\bar{j}\bar{k}}(\tau_s)&=2(6\pi-\sigma_{ijk}(\tau_s))-\sigma_{ijk}(\tau_s)=12\pi-3\sigma_{ijk}(\tau_s)\geq 12\pi-3\cdot4\pi=0,\\
		H_{\bar{i}\bar{j}\bar{k}}(\tau_h)&=2\sigma_{\bar{i}\bar{j}\bar{k}}(\tau_h)-(3\pi+\frac{1}{2}\sigma_{\bar{i}\bar{j}\bar{k}}(\tau_h))=\frac{3}{2}\sigma_{\bar{i}\bar{j}\bar{k}}(\tau_h)-3\pi\leq\frac{3}{2}\cdot2\pi-3\pi=0.
		\end{aligned}
		\end{equation*}
	\end{itemize}
	The proof is now complete.
\end{proof}

\subsubsection{Adjacent cells}

We now describe the five polytopes that compose $\cP(G)$ (see Theorem \ref{solution}).

\begin{Lemm}\label{connectedcomponents}
	The subsets $P_{2k\pi}$ (see Definition \ref{polytope}) enjoy the following properties of connectedness:
	\begin{enumerate}[i)]
		\item $P_{4\pi}$ has two connected components: one of them is $\{H_{222}<-4\pi\}$, and the other is $\{S<4\pi\}\cap\{H_{122}<2\pi\}\cap\{H_{212}<2\pi\}\cap\{H_{221}<2\pi\}\cap\{H_{111}>2\pi\}$;
		\item $P_{6\pi}$ is connected;
		\item $P_{8\pi}$ has two connected components: one of them is $\{H_{111}>10\pi\}$, and the other is 
		$\{S>8\pi\}\cap\{H_{211}>4\pi\}\cap\{H_{121}>4\pi\}\cap\{H_{112}>4\pi\}\cap\{H_{222}<4\pi\}$.
	\end{enumerate}
\end{Lemm}

\begin{proof}
	First notice that the symmetry $\psi$ exchanges the objects of Items i) and iii) (see Remarks \ref{symmetrysphwalls}, \ref{symmetryhypwalls}, \ref{symmetrypolytopes}), which is why we only prove one Item among the two.
	\begin{enumerate}[i)]
		\item If $H_{222}<-4\pi$, then $\sigma_{111}=2\sigma_{222}-H_{222}>2\sigma_{222}+4\pi>4\pi$, and therefore $S>\sigma_{111}>4\pi$. This implies that the two subsets of $P_{4\pi}$ described in Item i) are disconnected from each other. Besides, the second subspace is an intersection of half-spaces, so is connected.
		\item $P_{6\pi}$ is an intersection of two half-spaces, so is connected.
	\end{enumerate}
	This concludes the proof.
\end{proof}

\begin{Defi}
	We define:
	\begin{equation*}
	\begin{aligned}
	P_{4\pi}^{\mathrm{red}} \ &= \ \Sigma_{4\pi}\cup\{H_{222}=-4\pi\}
	\cup\Bigg(\bigcup_{(i,j,k)\in\{1,2\}^3\setminus\{(2,2,2)\}}\{H_{ijk}=2\pi,\sigma_{\bar{i}\bar{j}\bar{k}}\leq 2\pi\}\Bigg)\\
	P_{6\pi}^{\mathrm{red}} \ &= \ \Sigma_{6\pi}\cup\Bigg(\bigcup_{\substack{(i,j,k)\in\{(2,2,2),(2,2,1), \\ (2,1,2),(1,2,2)\}}}\{H_{ijk}=0,\sigma_{\bar{i}\bar{j}\bar{k}}\geq 4\pi\}\Bigg)\\
	&\cup\Bigg(\bigcup_{\substack{(i,j,k)\in\{(1,1,1),(1,1,2),\\ (1,2,1),(2,1,1)\}}}\{H_{ijk}=6\pi,\sigma_{\bar{i}\bar{j}\bar{k}}\leq 2\pi\}\Bigg)\\
	P_{8\pi}^{\mathrm{red}} \ &= \ \Sigma_{8\pi}\cup \{H_{111}=10\pi\}\cup\Bigg(\bigcup_{(i,j,k)\in\{1,2\}^3\setminus\{(1,1,1)\}}\{H_{ijk}=4\pi,\sigma_{\bar{i}\bar{j}\bar{k}}\geq 4\pi\}\Bigg)
	\end{aligned}
	\end{equation*}
	We say that $W\subset P_{2k\pi}^{\mathrm{red}}$ is an \textit{exterior wall}, respectively \textit{interior wall} of $P_{2k\pi}$, if $W\subset\partial P_{2k\pi}$, respectively if $W\subset P_{2k\pi}$ (see Definition \ref{polytope}).
\end{Defi}

\begin{Rema}
	Notice again that the symmetry $\psi$ exchanges $P_{4\pi}^{\mathrm{red}}$ and $P_{8\pi}^{\mathrm{red}}$, and globally preserves $P_{6\pi}^{\mathrm{red}}$.
\end{Rema}

We can now relate $P_{2k\pi}^{\mathrm{red}}$ to the lifting properties of a representation; see Definition \ref{P(G)u}.

\begin{Theo}\label{polytopelayers}
	Let $\omega=e^{2i\pi/3}$. We have:
	$$P_{4\pi}^{\mathrm{red}}=\cP(G)_\omega^{\mathrm{red}}, \ P_{6\pi}^{\mathrm{red}}=\cP(G)_1^{\mathrm{red}}, \ P_{8\pi}^{\mathrm{red}}=\cP(G)_{\omega^2}^{\mathrm{red}}.$$
\end{Theo}

\begin{proof}
	\begin{itemize}
		\item We have $\tau\in P_{4\pi}^{\mathrm{red}}$ if and only if $S(\tau)=4\pi$ or $H_{ijk}(\tau)=2\pi$ or $H_{ijk}(\tau)=-4\pi$ for some $(i,j,k)\in\{1,2\}^3$. This is equivalent to $-S(\tau)/3=-4\pi/3$ or $H_{ijk}(\tau)/3=2\pi/3$ or $H_{ijk}(\tau)/3=-4\pi/3$, i.e.\ the standard lifts of Definition \ref{standardlift} satisfy the relation $\widetilde{\rho(\mathtt{a})}\widetilde{\rho(\mathtt{b})}\widetilde{\rho(\mathtt{c})}=\omega \operatorname{Id}$. Therefore, we have shown $P_{4\pi}^{\mathrm{red}}=\cP(G)_\omega^{\mathrm{red}}$.
		\item We have $\tau\in P_{6\pi}^{\mathrm{red}}$ if and only if $S(\tau)=6\pi$ or $H_{ijk}(\tau)=0$ or $H_{ijk}(\tau)=6\pi$ for some $(i,j,k)\in\{1,2\}^3$. This is equivalent to $-S(\tau)/3=-2\pi/3$ or $H_{ijk}(\tau)/3=0$ or $H_{ijk}/3=2\pi$, i.e.\ the standard lifts of Definition \ref{standardlift} satisfy $\widetilde{\rho(\mathtt{a})}\widetilde{\rho(\mathtt{b})}\widetilde{\rho(\mathtt{c})}= \operatorname{Id}$. Therefore, we have shown $P_{6\pi}^{\mathrm{red}}=\cP(G)_1^{\mathrm{red}}$.
		\item We have $\tau\in P_{8\pi}^{\mathrm{red}}$ if and only if $S(\tau)=8\pi$ or $H_{ijk}(\tau)=4\pi$ or $H_{ijk}(\tau)=10\pi$ for some $(i,j,k)\in\{1,2\}^3$. This is equivalent to $-S(\tau)/3=-8\pi/3$ or $H_{ijk}(\tau)/3=4\pi/3$ or $H_{ijk}(\tau)/3=10\pi/3$, i.e.\ the standard lifts of Definition \ref{standardlift} satisfy $\widetilde{\rho(\mathtt{a})}\widetilde{\rho(\mathtt{b})}\widetilde{\rho(\mathtt{c})}=\omega^2 \operatorname{Id}$. Therefore, we have shown $P_{8\pi}^{\mathrm{red}}=\cP(G)_{\omega^2}^{\mathrm{red}}$.
	\end{itemize}
\end{proof}

The goal of the upcoming result is to identify which walls are interior.

\begin{Prop}\label{internalwalls}
	The following reducible walls are interior:
	\begin{enumerate}[i)]
		\item The hyperbolic reducible walls $\{H_{ijk}=2\pi,\sigma_{\bar{i}\bar{j}\bar{k}}<2\pi\}$ are interior walls of the connected component $P_{4\pi}\backslash\{H_{222}<-4\pi\}$ for $(i,j,k)\in\{(2,1,1),(1,2,1),(1,1,2)\}$;
		\item The spherical reducible wall $\Sigma_{6\pi}$ is an interior wall of $P_{6\pi}$;
		\item The hyperbolic reducible walls $\{H_{ijk}=0,\sigma_{\bar{i}\bar{j}\bar{k}}>4\pi\}$ are interior walls of the half-polytope $P_{6\pi}\cap\{S>6\pi\}$ for $(i,j,k)\in\{(2,2,1),(2,1,2),(1,2,2)\}$;
		\item The hyperbolic reducible walls $\{H_{ijk}=6\pi,\sigma_{\bar{i}\bar{j}\bar{k}}<2\pi\}$ are interior walls of the half-polytope $P_{6\pi}\cap\{S<6\pi\}$ for $(i,j,k)\in\{(1,1,2),(1,2,1),(2,1,1)\}$;
		\item The hyperbolic reducible walls $\{H_{ijk}=4\pi,\sigma_{\bar{i}\bar{j}\bar{k}}>4\pi\}$ are interior walls of the connected component $P_{8\pi}\backslash\{H_{111}>10\pi\}$ for $(i,j,k)\in\{(1,2,2),(2,1,2),(2,2,1)\}$.
	\end{enumerate}
\end{Prop}

\begin{proof}
	First of all, notice that the symmetry $\psi$ exchanges Items i) and v), and Items iii) and iv), which is why we only give a proof for the first three Items.
	\begin{enumerate}[i)]
		\item We show that $\{H_{211}=2\pi,\sigma_{122}<2\pi\}\subset P_{4\pi}$; the other cases are similar. Let $\tau$ be a triple of angle pairs in this reducible wall. First, we have $$\sigma_{211}=\frac{1}{2}(2\pi+\sigma_{122})<\frac{1}{2}(2\pi+2\pi)=2\pi,$$
		so $S<4\pi$. Then,
		\begin{equation*}
		\begin{aligned}
		H_{122}=2\sigma_{122}-\sigma_{211}=2\sigma_{122}-\frac{1}{2}(2\pi+\sigma_{122})=\frac{3}{2}\sigma_{122}-\pi<\frac{3}{2}\cdot 2\pi-\pi=2\pi.
		\end{aligned}
		\end{equation*}
		Next, by Lemma \ref{Hijkproperties} we have $H_{212},H_{221}<H_{211}=2\pi$, and finally $H_{111}>H_{211}=2\pi$. So $\tau\in P_{4\pi}$.
		
		\item Let $\tau\in\Sigma_{6\pi}$. We have $\sigma_{111}=S-\sigma_{222}>6\pi-2\pi=4\pi$, so
		$H_{111}>2\cdot 4\pi-2\pi=6\pi$ and $H_{222}<2\cdot 2\pi-4\pi=0$. Therefore $\tau\in P_{6\pi}$.
		
		\item We show that $\{H_{221}=0,\sigma_{112}>4\pi\}\subset P_{6\pi}\cap\{S>6\pi\}$; the other cases are similar. Let $\tau$ be a triple of angle pairs in this reducible wall. First, we have $$\sigma_{221}=\frac{1}{2}\sigma_{112}>\frac{1}{2}\cdot4\pi=2\pi,$$
		so $S>6\pi$.
		Then, by Lemma \ref{Hijkproperties}, we have $H_{222}<H_{221}=0$ and $H_{111}=S-H_{222}>6\pi$. So $\tau\in P_{6\pi}\cap\{S>6\pi\}$.
	\end{enumerate}
	The proof is now complete.
\end{proof}

From Theorem \ref{polytopelayers} and Proposition \ref{internalwalls} we deduce:

\begin{Coro}
	Let $C$ be a $\omega$-cell, respectively $1$-cell, respectively $\omega^2$-cell (see Definition
	\ref{cell}). Then, $C$ is entirely contained in one of the two subspaces $\overline{P_{4\pi}}$ or $\overline{P_{4\pi}^{c}},$ respectively $\overline{P_{6\pi}}$ or $\overline{P_{6\pi}^{c}},$ respectively $\overline{P_{8\pi}}$ or $\overline{P_{8\pi}^{c}}$.
\end{Coro}

Note that two hyperbolic reducible walls of $P_{2k\pi}^\mathrm{red}$ may intersect in a subspace which is not a totally reducible facet. This is what we describe in the following result, which will be useful to identify the cells:

\begin{Prop}\label{internalintersections}
	We have the following non-totally reducible intersections between the interior reducible walls:
	\begin{enumerate}[i)]
		\item The hyperbolic reducible walls $\{H_{ijk}=2\pi,\sigma_{\bar{i}\bar{j}\bar{k}}<2\pi\}$ intersect pairwise inside of $P_{4\pi}\backslash\{H_{222}<-4\pi\}$ for $(i,j,k)\in\{(2,1,1),(1,2,1),(1,1,2)\}$;
		\item The hyperbolic reducible walls $\{H_{ijk}=0,\sigma_{\bar{i}\bar{j}\bar{k}}>4\pi\}$ do not intersect inside of $P_{6\pi}$ for $(i,j,k)\in\{(2,2,1),(2,1,2),(1,2,2)\}$;
		\item The hyperbolic reducible walls $\{H_{ijk}=6\pi,\sigma_{\bar{i}\bar{j}\bar{k}}<2\pi\}$ do not intersect inside of $P_{6\pi}$ for $(i,j,k)\in\{(1,1,2),(1,2,1),(2,1,1)\}$;
		\item The hyperbolic reducible walls $\{H_{ijk}=4\pi,\sigma_{\bar{i}\bar{j}\bar{k}}>4\pi\}$ intersect pairwise inside of $P_{8\pi}\backslash\{H_{111}>10\pi\}$ for $(i,j,k)\in\{(1,2,2),(2,1,2),(2,2,1)\}$.
	\end{enumerate}
\end{Prop}

\begin{proof}
	First note that applying the symmetry $\psi$ exchanges Items i) and iv), and Items ii) and iii). For this reason, we only give a proof for the first two Items.
	\begin{enumerate}[i)]
		\item The two hyperplanes $\{H_{112}=2\pi\}$ and $\{H_{121}=2\pi\}$ are distinct and not parallel, so they intersect along a 4-dimensional subspace of $\RR^6$. We have seen in Proposition \ref{internalwalls} i) that the two walls $\{H_{112}=2\pi,\sigma_{221}<2\pi\}$ and $\{H_{121}=2\pi,\sigma_{212}<2\pi\}$ are contained in $P_{4\pi}$. So to prove the result, it suffices to find an element of $\cT(G)^3$ in the intersection. Consider the triple $$\tau=((3\pi/4,\pi/2),(2\pi/3,\pi/3),(2\pi/3,\pi/3))\in\cT(G)^3.$$ A straightforward computation implies that $$\tau\in\{H_{112}=2\pi,\sigma_{221}<2\pi\}\cap\{H_{121}=2\pi,\sigma_{212}<2\pi\}.$$ Therefore the two walls intersect inside of $\cT(G)^3$, so by the first Item of Proposition \ref{internalwalls}, they intersect inside of $P_{4\pi}\backslash\{H_{222}<-4\pi\}$. Besides, we have $\tau\notin\{H_{211}=2\pi\}$, so $$\{H_{112}=2\pi\}\cap\{H_{121}=2\pi\}\neq\{H_{112}=2\pi\}\cap\{H_{211}=2\pi\}.$$ We obtain an analogous result for each value of $(i,j,k)\in\{(2,1,1),(1,2,1),(1,1,2)\}$ by permuting $\alpha_i,\beta_i,\gamma_i$.
		
		\item Assume that there exists a triple $$\tau\in\{H_{221}=0,\sigma_{112}>4\pi\}\cap\{H_{212}=0,\sigma_{121}>4\pi\}.$$ Then,
		\begin{equation*}
		\begin{aligned}
		0&=H_{221}+H_{212}=4\alpha_2-2\alpha_1+\beta_1+\gamma_2+\beta_2+\gamma_1\\
		&>4\alpha_2-2\alpha_1+8\pi-2\alpha_1=8\pi+4(\alpha_2-\alpha_1).
		\end{aligned}
		\end{equation*}
		This implies $\alpha_1-\alpha_2>2\pi$, which is impossible in $\cT(G)^3$. So these two walls do not intersect in $\cT(G)^3$. Permuting the indices, we obtain the same result for the other two intersections.
	\end{enumerate}
	The proof is now complete.
\end{proof}

\section{Examples}\label{examples}

Before proving Theorem \ref{solutionpu}, we describe in this section a few examples.

\subsection{Slices of the polytopes}\label{figures}

We begin by considering certain two-dimensional slices of $\cT(G)^3$, to try giving the reader an intuition of the situations we have described in the previous sections. We use the following color code:

\begin{itemize}
	\item \textbf{Red bold segments}: exterior walls of $P_{4\pi}$;
	
	\textbf{Red dotted segments}: interior walls of $P_{4\pi}$.
	
	\item \textbf{Blue bold segments}: exterior walls of $P_{6\pi}$;
	
	\textbf{Blue dotted segments}: interior walls of $P_{6\pi}$.
	
	\item \textbf{Green bold segments}: exterior walls of $P_{8\pi}$;
	
	\textbf{Green dotted segments}: interior walls of $P_{8\pi}$.
\end{itemize}
We have colored $P_{4\pi}$ in red, $P_{6\pi}$ in blue, and $P_{8\pi}$ in green.

\begin{Exem}\label{symslice}
	We define the symmetric slice as the case where $\cC_1=\cC_2=\cC_3$. In $\cT(G)^3$ coordinates, we have
	\begin{equation*}
	\begin{aligned}
	\operatorname{Sym}&=\{(\alpha_1,\alpha_2),(\beta_1,\beta_2),(\gamma_1,\gamma_2)\in\cT(G)^3 \ | \ (\alpha_1,\alpha_2)=(\beta_1,\beta_2)=(\gamma_1,\gamma_2)\} \\
	& \cong \cT(G).
	\end{aligned}
	\end{equation*}
	This is a two-dimensional subspace of $\cT(G)^3$. We have depicted in Figure \ref{tranche_sym} the intersection of the slice Sym with the various reducible walls of $\cP(G)^\mathrm{red}$.
	\begin{figure}
		\centering
		\scalebox{0.48}{\includegraphics{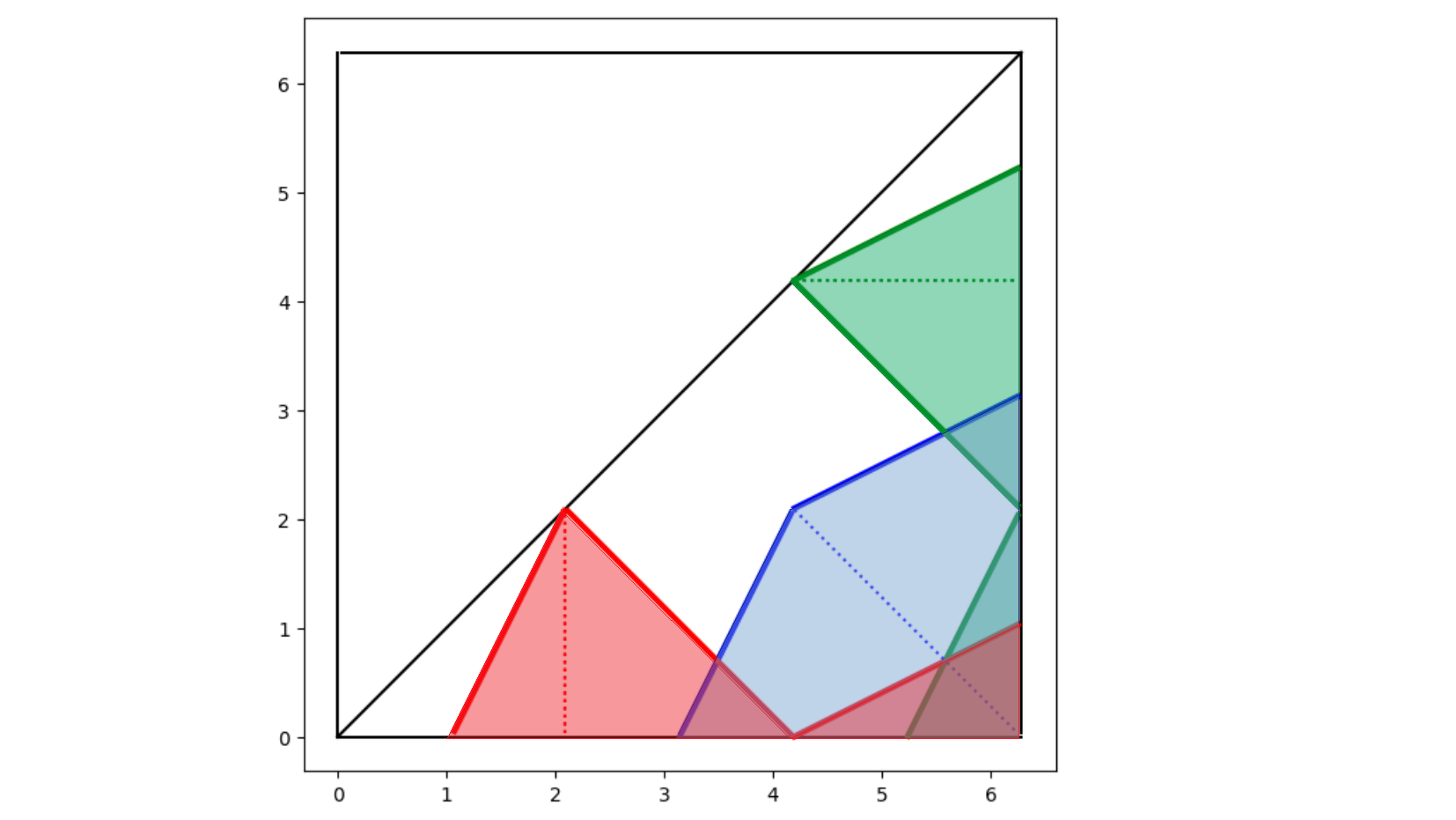}}
		\caption{The symmetric slice $\operatorname{Sym}$. \label{tranche_sym}}
	\end{figure}
	Note that the action of the involution $\psi$, which exchanges $P_{4\pi}$ (red) and $P_{8\pi}$ (green) and leaves $P_{6\pi}$ (blue) invariant, is apparent on Figure \ref{tranche_sym}. Indeed, the restriction of $\psi$ to Sym is given by $\psi(\alpha_1,\alpha_2)=(2\pi-\alpha_2,2\pi-\alpha_1)$, which is the symmetry with respect to the anti-diagonal of the square.
	\begin{itemize}
		\item The red (respectively, green) bold and dotted segments are the reducible walls of $P_{4\pi}^{\mathrm{red}}$ (respectively, $P_{8\pi}^{\mathrm{red}}$) that appear in this slice. More precisely, the red (respectively, green) segment of slope $-1$ is the spherical reducible wall $\Sigma_{4\pi}$ (respectively, $\Sigma_{8\pi}$); the red (respectively, green) segment of slope $1/2$ is the hyperbolic reducible wall $\{H_{222}=-4\pi\}$ (respectively, $\{H_{222}=4\pi\}$); and the red (respectively, green) segment of slope $2$ is the hyperbolic reducible wall $\{H_{111}=2\pi\}$ (respectively, $\{H_{111}=10\pi\}$). Those are all exterior walls of $P_{4\pi}$ (respectively, $P_{8\pi}$). The red (respectively, green) dotted segment is the reunion of the hyperbolic reducible walls $\{H_{ijk}=2\pi,\sigma_{\bar{i}\bar{j}\bar{k}}\leq2\pi\}$ for $(i,j,k)\in\{(2,1,1),(1,2,1),(1,1,2)\}$ (respectively, $\{H_{ijk}=4\pi,\sigma_{\bar{i}\bar{j}\bar{k}}\geq4\pi\}$ for $(i,j,k)\in\{(1,2,2),(2,1,2),(2,2,1)\}$), which are equal to each other in this slice.  Those are all interior walls of $P_{4\pi}$ (respectively, $P_{8\pi}$). The (exterior) hyperbolic reducible walls $\{H_{ijk}=2\pi,\sigma_{\bar{i}\bar{j}\bar{k}}\leq2\pi\}$ for $(i,j,k)\in\{(1,2,2),(2,1,2),(2,2,1)\}$ (respectively, $\{H_{ijk}=4\pi,\sigma_{\bar{i}\bar{j}\bar{k}}\geq4\pi\}$ for $(i,j,k)\in\{(2,1,1),(1,2,1),(1,1,2)\}$) are equal to the point $(2\pi/3,2\pi/3)$ (respectively, $(4\pi/3,4\pi/3)$). Since $P_{4\pi}^{\mathrm{red}}=\cP(G)_\omega^{\mathrm{red}}$ (respectively, $P_{8\pi}^{\mathrm{red}}=\cP(G)_{\omega^2}^{\mathrm{red}}$) by Theorem \ref{polytopelayers}, we see five $\omega$-cells (respectively, $\omega^2$-cells) in this slice: the complement of the red (respectively, green) segments has five connected components.
		
		\item The blue bold and dotted segments are the reducible walls of $P_{6\pi}^{\mathrm{red}}$ that appear in this slice. More precisely, the blue segment of slope $2$ is the hyperbolic reducible wall $\{H_{111}=6\pi,\sigma_{222}\leq2\pi \}$, and the blue segment of slope $1/2$ is the hyperbolic reducible wall $\{H_{222}=0,\sigma_{111}\geq4\pi \}$.  Those are all exterior walls of $P_{6\pi}$. The blue dotted segment is the spherical reducible wall $\Sigma_{6\pi}$, which is an interior wall of $P_{6\pi}$. The (interior) hyperbolic reducible walls $\{H_{ijk}=0,\sigma_{\bar{i}\bar{j}\bar{k}}\geq4\pi\}$ for $(i,j,k)\in\{(1,2,2),(2,1,2),(2,2,1)\}$ and $\{H_{ijk}=6\pi,\sigma_{\bar{i}\bar{j}\bar{k}}\leq2\pi\}$ for $(i,j,k)\in\{(2,1,1),(1,2,1),(1,1,2)\}$ are all equal to the point $(2\pi,0)$. Since $P_{6\pi}^{\mathrm{red}}=\cP(G)_1^{\mathrm{red}}$ by Theorem \ref{polytopelayers}, we see three $1$-cells in this slice: the complement of the blue segments has three connected components.
	\end{itemize}
\end{Exem}

\begin{Exem}
	Fix $(\beta_1^*,\beta_2^*),(\gamma_1^*,\gamma_2^*)\in\cT(G)$, and define
	$$\operatorname{Slice}((\beta_1^*,\beta_2^*),(\gamma_1^*,\gamma_2^*))=\cT(G)\times\{(\beta_1^*,\beta_2^*)\}\times\{(\gamma_1^*,\gamma_2^*)\}.$$
	This is a two-dimensional subspace of $\cT(G)^3$. We have represented examples of such slices in Figures \ref{tranche_paupert_1}, \ref{tranche_paupert_2}, \ref{tranche_paupert_3}, \ref{tranche_paupert_4} with the color code mentioned at the beginning of this section.
	
	On these figures, we observe the four possible types of totally reducible intersections. We also observe non-totally reducible intersections between internal walls of $P_{8\pi}$ (Figures \ref{tranche_paupert_3} and \ref{tranche_paupert_4}).
	
	These slices correspond to Paupert's drawings of the image of the momentum map in \cite{Pau} (see the Introduction for the definition of the momentum map). Note that our explicit description of the reducible walls enables us to implement an algorithm whose output is such a drawing, giving as an input the sole data of the two angle pairs $(\beta_1^*,\beta_2^*),(\gamma_1^*,\gamma_2^*).$
\end{Exem}

\begin{figure}
	\centering
	\scalebox{0.48}{\includegraphics{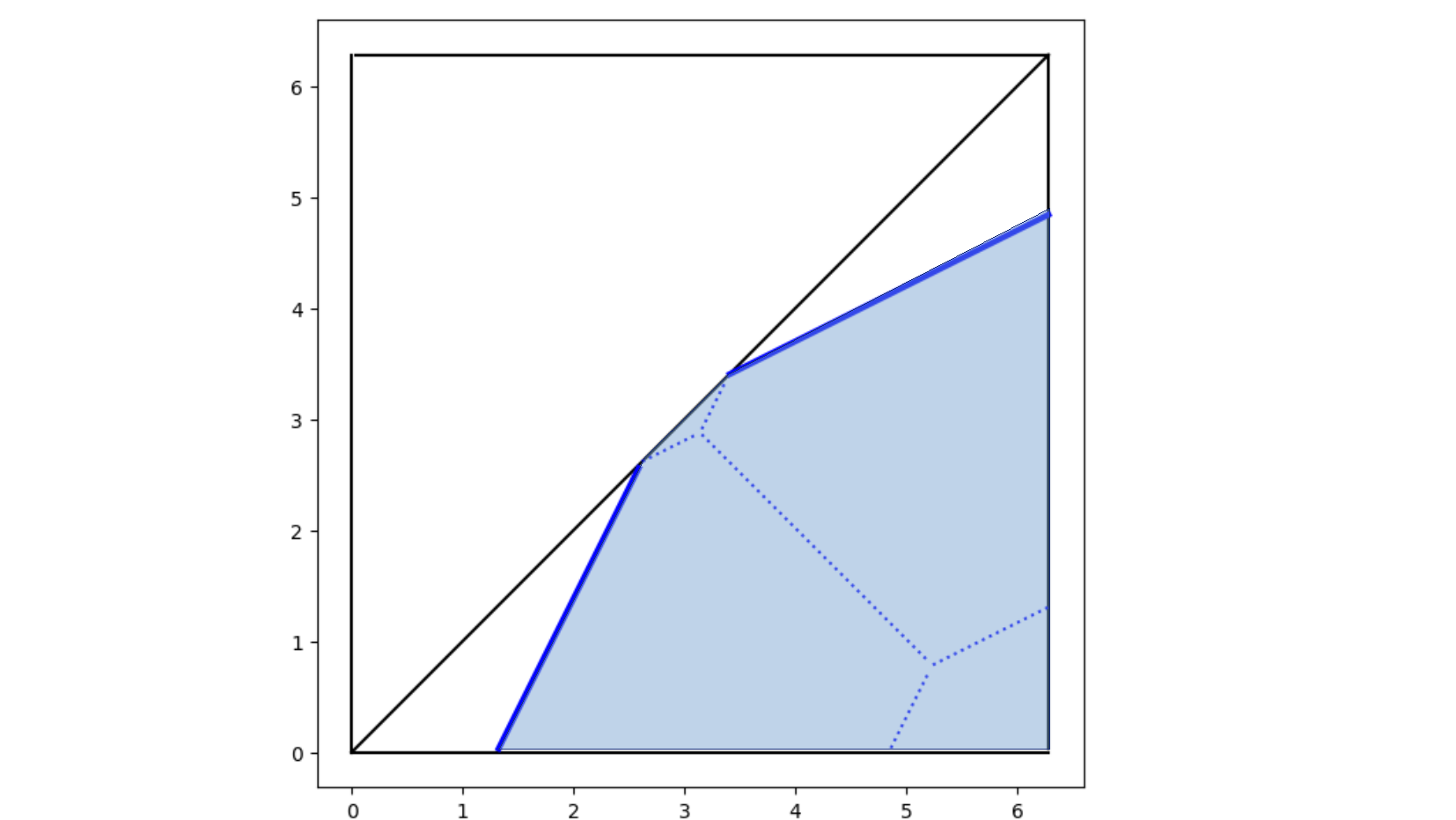}}
	\caption{The slice $\operatorname{Slice}((5\pi/4,\pi/2),(11\pi/6,\pi/2))$. The only polytope that appears is $P_{6\pi}.$ The two dotted blue triple intersections are totally reducible facets of type 4. We see six $1$-cells. \label{tranche_paupert_1}}
\end{figure}

\begin{figure}
	\centering
	\scalebox{0.48}{\includegraphics{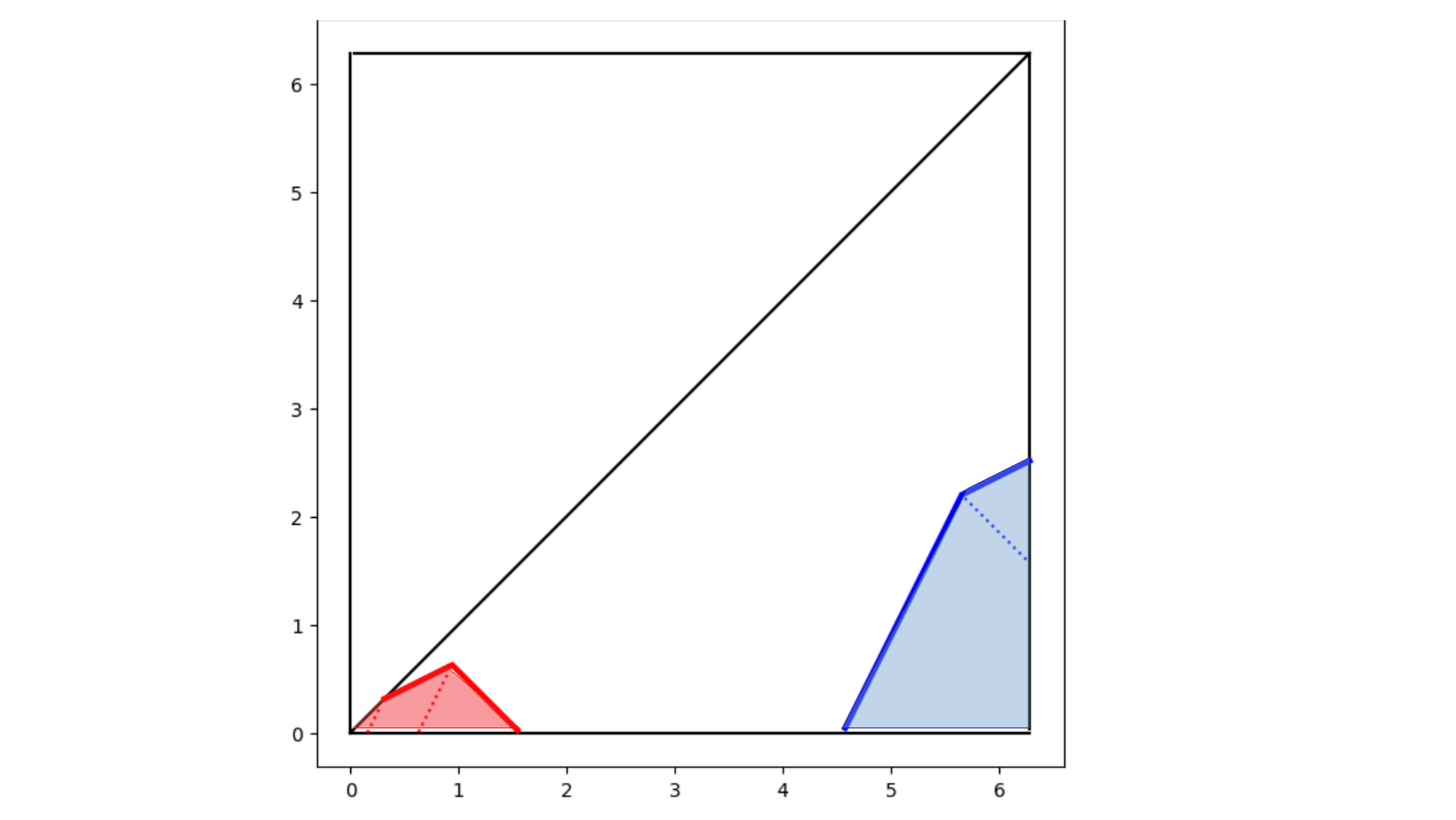}}
	\caption{The slice $\operatorname{Slice}((6\pi/5,4\pi/5),(\pi,\pi/2))$. The two polytopes that appear are $P_{4\pi}$ (red) and $P_{6\pi}$ (blue). The red triple intersection is a totally reducible facet of type 1, and the blue triple intersection is a totally reducible facet of type 3. We see four $\omega$-cells and three $1$-cells. \label{tranche_paupert_2}}
\end{figure}

\begin{figure}
	\centering
	\scalebox{0.48}{\includegraphics{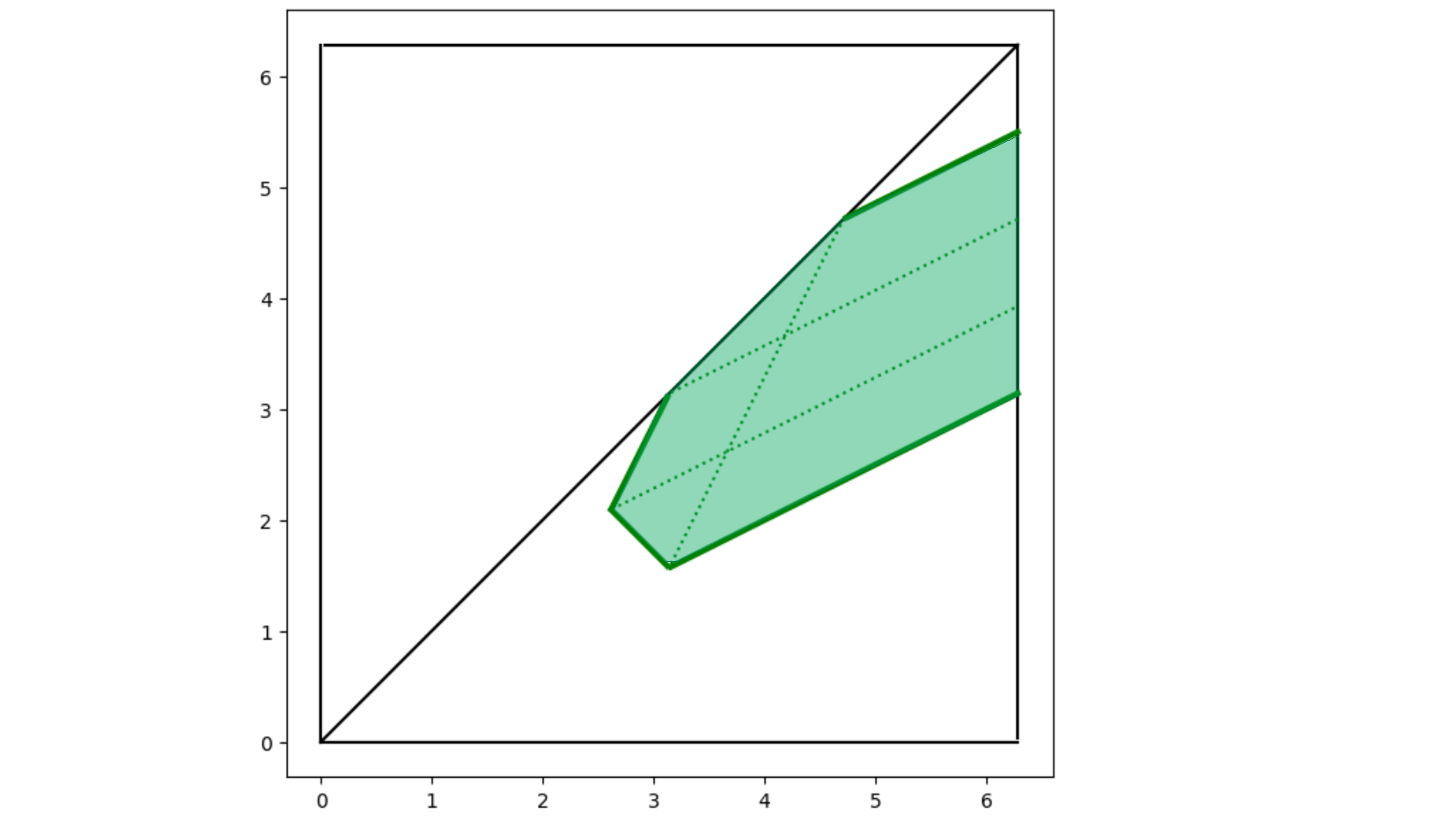}}
	\caption{The slice $\operatorname{Slice}((11\pi/6,5\pi/3),(5\pi/3,4\pi/3))$. The only polytope that appears is $P_{8\pi}.$ The two green bold and dotted triple intersections are totally reducible facets of type 2. The two dotted intersections are non-totally reducible intersections between internal hyperbolic reducible walls. We see eight $\omega^2$-cells. \label{tranche_paupert_3}}
\end{figure}

\begin{figure}
	\centering
	\scalebox{0.48}{\includegraphics{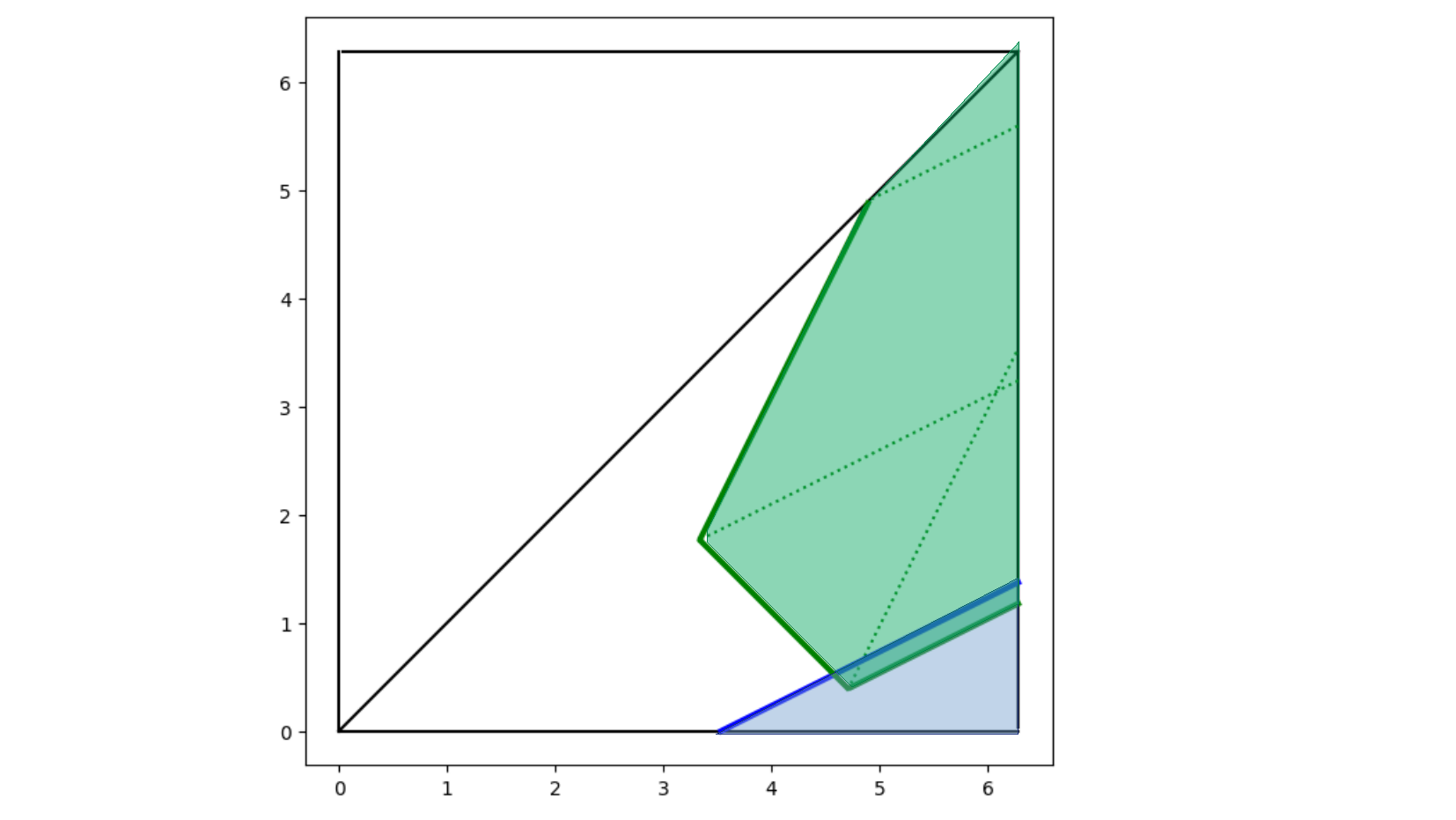}}
	\caption{The slice $\operatorname{Slice}((31\pi/16,3\pi/2),(31\pi/16,\pi))$. The two polytopes that appear are $P_{8\pi}$ (green) and $P_{6\pi}$ (blue). We see that they overlap. We see two $1$-cells and six $\omega^2$-cells. \label{tranche_paupert_4}}
\end{figure}

\subsection{Explicit constructions of representations}

Constructing explicitly a representation is hard, which is why we need arguments of various types to fill or empty the cells (recall Theorem \ref{emptyfulltheorem}). However, in some particular cases, we are able to say that a triple of angle pairs cannot belong to $\cP(G)$. Besides, \cite{Mar} yields an explicit method of construction, based on the so-called decomposability property.

\begin{Lemm}\label{diago}
	Define 
	$$\cD=\{\tau\in\overline{\cT(G)}^3 \ | \ \alpha_1=\alpha_2,\beta_1=\beta_2,\gamma_1=\gamma_2\}\subset\partial\cT(G)^3.$$
	Then, $\cD\cap\cP(G)\subset\cP(G)^\mathrm{red}$. 
\end{Lemm}

\begin{proof}
	Let $\rho\in\operatorname{Hom}_\mathrm{ell}(\Gamma,G)$ such that $\bar{c}(\rho)\in \cD$. Then, $\rho(\texttt{a}),\rho(\texttt{b}),\rho(\texttt{c})$ are three complex reflections in points, i.e.\ isometries that fix point-wise a projective line (see Section \ref{compref}). Therefore, the two fixed projective lines of $\rho(\texttt{a})$ and $\rho(\texttt{b})$ have an intersection point which is fixed by the three elements $\rho(\texttt{a}),\rho(\texttt{b}),\rho(\texttt{ab})^{-1}=\rho(\texttt{c})$. Consequently, they form a reducible triple.
\end{proof}	

\begin{Lemm}\label{reflpoint}
	$\cP(G)$ is closed in $\mathring{\cT(G)}^3\cup\mathring{\cD}$ (compare with Proposition \ref{closed}).
\end{Lemm}

\begin{proof}
	Let $(\tau_n)_n\subset\cP(G)$ be a sequence that converges to $\tau\in\mathring{\cD}$. First of all, the same arguments as in the proof of Proposition \ref{closed} show that there exists $\rho\in\operatorname{Hom}(\Gamma,G)$ such that $\tau=\bar{c}(\rho)$ (Corollary \ref{bestpau}). Now, we need to be a little careful, because special elliptic conjugacy classes are not separate from parabolic classes in general (see Section \ref{separabilitysection}). However, this default of separation occurs only for classes of complex reflections in lines, and not for classes of complex reflections in points which are the classes present on $\mathring{\cD}$ (see Definition \ref{reflectionformula}): this is the result proven in Lemma \ref{separability}. Therefore, $\rho\in\operatorname{Hom}_\mathrm{ell}(\Gamma,G)$.
\end{proof}

\begin{Prop}\label{diagonalinboundary}
	Let $C$ be a cell that contains an open subset of $\cD$ in its boundary. Then, $C$ is empty.
\end{Prop}

\begin{proof}
	We show the contraposition. Assume that $C$ is full and contains a subset $S\subset \cD$ in its boundary. Then by Lemma \ref{reflpoint}, we have $S\subset\cP(G)$. By Lemma \ref{diago}, $S$ is also contained in $\cP(G)^\mathrm{red}$. Now, $\cD\cap\Sigma_{2k\pi}$ and $\cD\cap \{H_{ijk}=2l\pi\}$ are two-dimensional subsets of $\cD$, whose dimension is three. So $S$ cannot be open in $\cD$. This shows the result.
\end{proof}

We now construct explicitly an irreducible representation with fixed angle pairs, which we will use later on to fill some cells. We do not explain here where it comes from, because this would involve a lot of complex hyperbolic geometry; to find it, we have actually used the results of \cite{Mar}.

\begin{Prop}\label{decompfamily2}
	There exists an irreducible solution to the Horn problem for the three $\su(2,1)$-conjugacy classes respectively represented by the matrices 
	$$\begin{pmatrix}
	1 & 0 & 0 \\
	0 & e^{i\pi/3} & 0 \\
	0 & 0 & e^{-i\pi/3} \\
	\end{pmatrix}, \ \begin{pmatrix}
	e^{i\pi/3} & 0 & 0 \\
	0 & 1 & 0 \\
	0 & 0 & e^{-i\pi/3} \\
	\end{pmatrix}, \ \begin{pmatrix}
	e^{-i\pi/3} & 0 & 0 \\
	0 & e^{-2i\pi/3} & 0 \\
	0 & 0 & -1 \\
	\end{pmatrix},$$
	where the first two eigenvectors of these matrices are of positive type and the third one is of negative type.
\end{Prop}

\begin{proof}
	Let $r_1=\sqrt{\sqrt{3}-1},r_2=r_3=\sqrt{\sqrt{3}+1}$. Let $\theta=11\pi/6$ and $\alpha=2\pi/3$. Let $\eta_1=\eta_2=e^{i\theta},\eta_3=e^{i(\theta+\pi)}$ and $u=e^{i\alpha/3}$. Then:
	\begin{enumerate}[i)]
		\item the matrix
		$$H=\begin{pmatrix}
		-1 & -r_3u^{-1} & r_2u \\
		-r_3u & 1 & -r_1u^{-1} \\
		r_2u^{-1} & -r_1u & -1 
		\end{pmatrix}$$
		defines a Hermitian form of signature $(2,1)$, and the three matrices
		\begin{equation*}
		\begin{aligned}
		R_1&=\eta_1^{-1/3}\begin{pmatrix}
		\eta_1 & r_3(\eta_1-1)u & -r_2(\eta_1-1)u^{-1} \\
		0 & 1 & 0 \\
		0 & 0 & 1
		\end{pmatrix},\\
		R_2&=\eta_2^{-1/3}\begin{pmatrix}
		1 & 0 & 0 \\
		-r_3(\eta_2-1)u^{-1} & \eta_2 & -r_1(\eta_2-1)u \\
		0 & 0 & 1
		\end{pmatrix},\\
		R_3&=\eta_3^{-1/3}\begin{pmatrix}
		1 & 0 & 0 \\
		0 & 1 & 0 \\
		-r_2(\eta_3-1)u & r_1(\eta_3-1)u^{-1} & \eta_3
		\end{pmatrix}.
		\end{aligned}
		\end{equation*}
		belong to $\su(H)$;
		\item the three matrices 
		$$A=R_1R_2^{-1}, \ B=R_2R_3^{-1}, \ C=R_3R_1^{-1}$$
		form an irreducible triple, and are respectively conjugate to
		$$\begin{pmatrix}
		1 & 0 & 0 \\
		0 & e^{i\pi/3} & 0 \\
		0 & 0 & e^{-i\pi/3} \\
		\end{pmatrix}, \ \begin{pmatrix}
		e^{i\pi/3} & 0 & 0 \\
		0 & 1 & 0 \\
		0 & 0 & e^{-i\pi/3} \\
		\end{pmatrix}, \ \begin{pmatrix}
		e^{-i\pi/3} & 0 & 0 \\
		0 & e^{-2i\pi/3} & 0 \\
		0 & 0 & -1 \\
		\end{pmatrix}.$$
	\end{enumerate}

	Indeed, for point ii) it is a direct computation to check that:
	\begin{itemize}
		\item the vectors
		$$\begin{pmatrix}
		-\frac{1}{r_2}(r_1^2u^{-2}-r_1^2u) \\
		(1-r_1^2)u^{-1}+r_1^2u^2 \\
		1
		\end{pmatrix},\begin{pmatrix}
		\frac{1-i}{4}(r_1^2-1-i)r_3u^{-1} \\
		1 \\
		0
		\end{pmatrix},\begin{pmatrix}
		1 \\
		\frac{-1+i}{4}(r_1^2-1-i)r_3u \\
		0
		\end{pmatrix}$$
		are eigenvectors of $A$ for the eigenvalues $1,e^{i\pi/3},e^{-i\pi/3}$, and that they are respectively of positive, positive and negative type for $H$;
		
		\item the vectors
		$$\begin{pmatrix}
		1 \\
		r_2\sqrt{3}(r_1u^{-2}+u)	 \\
		-r_2\sqrt{3}(r_1u^{2}-u^{-1})
		\end{pmatrix},\begin{pmatrix}
		0 \\
		1 \\
		-\frac{1+i}{4}(r_2^2+1-i)r_1u
		\end{pmatrix},\begin{pmatrix}
		0 \\
		-\frac{1+i}{4}(r_1^2-1-i)r_1u^{-1} \\
		1
		\end{pmatrix}$$
		are eigenvectors of $B$ for the eigenvalues $e^{i\pi/3},1,e^{-i\pi/3}$, and that they are respectively of positive, positive and negative type for $H$;
		
		\item the vectors
		$$\begin{pmatrix}
		r_2(1+(r_2^2+1)i)(r_1u^2+u^{-1}) \\
		r_2^2-1+(2r_2^2+1)i \\
		((1+(r_2^2+1)i)r_1u+(r_2^2+(3r_2^2+2)i)u^{-2})
		\end{pmatrix},\begin{pmatrix}
		r_2u(-1+(r_1^2-1)i) \\
		0 \\
		-r_2^2+ir_1^2
		\end{pmatrix},$$ $$\begin{pmatrix}
		-r_2^2+ir_1^2 \\
		0 \\
		r_2u^{-1}(-1+(r_2^2+1)i)
		\end{pmatrix}$$
		are eigenvectors of $C$ for the eigenvalues $e^{-i\pi/3},e^{-2i\pi/3},-1$, and that they are respectively of positive, positive and negative type for $H$.
	\end{itemize}
	 Besides, it is clear from these formulas that $A,B,C$ have no common eigenvector, which implies that the triple is indeed irreducible.
\end{proof}

\section{Proof of Theorem \ref{solutionpu}: solution in $\pu(2,1)$}\label{proof}

We now prove Theorem \ref{solutionpu}, which describes the solution set of the elliptic multiplicative Horn problem in $\pu(2,1)$ as a union of five explicit polytopes. Namely, we show that $\cP(G)=\overline{P_{4\pi}}\cup \overline{P_{6\pi}}\cup \overline{P_{8\pi}}$.

\subsection{Empty cells}

We first examine the complements of $P_{4\pi}$, $P_{6\pi}$ and $P_{8\pi}$ in $\cT(G)^3$ and prove, cell by cell, that they are empty.

\begin{Prop}\label{complementcells}
	\begin{itemize}
		\item $P_{4\pi}^c$ is the disjoint union of two cells: $P_{4\pi}^c=C_{4\pi}^{c,1}\sqcup C_{4\pi}^{c,2}$, where
		\begin{equation*}
		\begin{aligned}
		C_{4\pi}^{c,1}=&\{H_{222}>-4\pi\}\cap\Big(\{S>4\pi\}\cup\{H_{122}>2\pi\}\cup\{H_{212}>2\pi\}\cup\{H_{221}>2\pi\}\Big), \\
		C_{4\pi}^{c,2}=&\{H_{111}<2\pi\}.
		\end{aligned}
		\end{equation*}
		\item $P_{6\pi}^c$ is equal to a unique cell: $P_{6\pi}^c=C_{6\pi}^{c}$, where
		$$C_{6\pi}^{c}=\{H_{222}>0\}\cup\{H_{111}<6\pi\}.$$
		\item $P_{8\pi}^c$ is the disjoint union of two cells: $P_{8\pi}^c=C_{8\pi}^{c,1}\sqcup C_{8\pi}^{c,2}$, where
		\begin{equation*}
		\begin{aligned}
		C_{8\pi}^{c,1}=&\{H_{111}<10\pi\}\cap\Big(\{S<8\pi\}\cup\{H_{211}<4\pi\}\cup\{H_{121}<4\pi\}\cup\{H_{112}<4\pi\}\Big), \\
		C_{8\pi}^{c,2}=&\{H_{222}>4\pi\}.
		\end{aligned}
		\end{equation*}
	\end{itemize}
\end{Prop}

We advise the reader to take a second look at Figure \ref{tranche_sym}, on which this result becomes apparent.

\begin{proof}
	First note that the symmetry $\psi$ exchanges the objects of the first and third points (see Remark \ref{symmetry}). We therefore only prove the first one among the two.
	\begin{itemize}
		\item Let us prove the first point. Taking the complement in the definition of $P_{4\pi}$, we obtain 
		\begin{equation*}
		\begin{aligned}
		&P_{4\pi}^c=\{H_{222}>-4\pi\}\cap\\
		&\Big(\{S>4\pi\}\cup\{H_{122}>2\pi\}\cup\{H_{212}>2\pi\}\cup\{H_{221}>2\pi\}\cup\{H_{111}<2\pi\}\Big).
		\end{aligned}
		\end{equation*}
		Distributing the intersection over the reunion yields
		$$P_{4\pi}^c=C_{4\pi}^{c,1}\cup (\{H_{222}>-4\pi\}\cap\{H_{111}<2\pi\}).$$
		Now, $H_{111}<2\pi$ implies $H_{222}=S-H_{111}>-2\pi>-4\pi$, so $$\{H_{222}>-4\pi\}\cap\{H_{111}<2\pi\}=C_{4\pi}^{c,2}.$$
		We now show that the subspaces $C_{4\pi}^{c,1}$ and $C_{4\pi}^{c,2}$ are disjoint. To see this, observe that the inequality $\sigma_{111}>\sigma_{222}$ directly implies
		\begin{equation*}
		\begin{aligned}
		H_{111}&=2\sigma_{111}-\sigma_{222}>2\sigma_{111}-\sigma_{111}=\sigma_{111}, \\
		H_{111}&=2\sigma_{111}-\sigma_{222}>2\sigma_{222}-\sigma_{222}=\sigma_{222}.
		\end{aligned}
		\end{equation*}  Therefore, if $H_{111}<2\pi$, then $S=\sigma_{111}+\sigma_{222}<2H_{111}<4\pi$. Besides, we have $H_{122},H_{212},H_{221}<H_{111}<2\pi$ by Lemma \ref{Hijkproperties}. So $C_{4\pi}^{c,1}\cap C_{4\pi}^{c,2}=\varnothing$.
		
		Clearly, $C_{4\pi}^{c,2}$ is connected. We now show that $C_{4\pi}^{c,1}$ is connected. To that end, we are going to prove that it is star-shaped with respect to the point $$p_0=((\pi,\pi),(\pi,\pi),(\pi,\pi))$$ (see Figure \ref{tranche_sym}). Fix $\tau^*$ a triple in $C_{4\pi}^{c,1}$, and parameterize the segment $[p_0,\tau^*]$ by $p_t=t\tau^*+(1-t)p_0$. We now show that $\forall t\in[0,1],p_t\in C_{4\pi}^{c,1}$. We have
		$$H_{ijk}(p_0)=3\pi,\ S(p_0)=6\pi,$$
		$$H_{ijk}(p_t)=tH_{ijk}(\tau^*)+3(1-t)\pi,\ S(p_t)=tS(\tau^*)+6(1-t)\pi.$$
		Note that if $H_{ijk}(\tau^*)>a$ with $a<3\pi$, then $H_{ijk}(p_t)>a$. Indeed, under these hypotheses we obtain
		$$H_{ijk}(p_t)>ta+3(1-t)\pi=3\pi+(a-3\pi)t>3\pi+a-3\pi=a.$$
		Similarly, if $S(\tau^*)>b$ with $b<6\pi$, then $S(p_t)>b$. Indeed, under these hypotheses we obtain
		$$S(p_t)>tb+6(1-t)\pi=6\pi+(b-6\pi)t>6\pi+b-6\pi=b.$$
		Applying this with $a=-4\pi,2\pi$ and $b=4\pi$, this shows that the segment $[p_0,\tau^*]$ is contained in $C_{4\pi}^{c,1}$, and consequently, $C_{4\pi}^{c,1}$ is connected.
		
		According to Theorem \ref{polytopelayers} and Proposition \ref{internalwalls}, no other reducible walls may form sides of cells inside $P_{4\pi}^c$. Therefore, $C_{4\pi}^{c,1}$ and $C_{4\pi}^{c,2}$ are the only cells of $P_{4\pi}^c$.
		
		\item We now prove the second point. We have $P_{6\pi}^{c}=\{H_{222}>0\}\cup\{H_{111}<6\pi\}.$ Since this union is not disjoint, $P_{6\pi}^{c}$ is connected. According to Theorem \ref{polytopelayers} and Proposition \ref{internalwalls}, no other reducible walls may form sides of cells inside $P_{6\pi}^c$. Therefore, $P_{6\pi}^{c}$ is composed of a unique cell.
	\end{itemize}
	The proof is complete.
\end{proof}

\begin{Theo}\label{emptycells}
	The five cells described in Proposition \ref{complementcells} are empty. In other words,
	$$P_{4\pi}^c\subset(\cP(G)_\omega)^c, \ P_{6\pi}^c\subset(\cP(G)_1)^c, \ P_{8\pi}^c\subset(\cP(G)_{\omega^2})^c.$$
\end{Theo}

\begin{proof}
	We show that they all contain an open subset of $$\cD=\{\tau\in\overline{\cT(G)}^3 \ | \ \alpha_1=\alpha_2,\beta_1=\beta_2,\gamma_1=\gamma_2\}$$ in their boundary, and use Proposition \ref{diagonalinboundary}. Remark that on $\cD$, we have $H_{ijk}=\sigma_{111}$ for any $(i,j,k)$. We obtain (see Figure \ref{tranche_sym}):
	\begin{itemize}
		\item $\{\sigma_{111}>2\pi\}\cap \cD\subset \partial C_{4\pi}^{c,1}$. Indeed, $\tau\in\{\sigma_{111}>2\pi\}\cap \cD$ implies $H_{ijk}(\tau)=\sigma_{111}>2\pi$ for any $i,j,k$, and $S(\tau)=2\sigma_{111}>4\pi$.
		
		\item $\{\sigma_{111}<2\pi\}\cap \cD\subset \partial C_{4\pi}^{c,2}$. Indeed, $\tau\in\{\sigma_{111}<2\pi\}\cap \cD$ implies $H_{111}(\tau)=\sigma_{111}<2\pi$.
		\item $\cD\subset\partial C_{6\pi}^{c}$. Indeed, $\tau\in\cD$ implies $H_{222}(\tau)=\sigma_{111}>0$ and $H_{111}(\tau)=\sigma_{111}<6\pi$.
		\item $\{\sigma_{111}<4\pi\}\cap \cD\subset \partial C_{8\pi}^{c,1}$ and $\{\sigma_{111}>4\pi\}\cap \cD\subset \partial C_{8\pi}^{c,2}$, by symmetry of the first two points.
	\end{itemize}
	Those are all open subsets of $\cD$, which shows the result.
\end{proof}

\subsection{Full cells}

We now describe the cells composing $P_{4\pi}$, $P_{6\pi}$ and $P_{8\pi}$, and prove that they are full.

\begin{Prop}\label{fullcells}
	\begin{itemize}
		\item $P_{4\pi}$ is the reunion of the nine cells defined below: 
		\begin{equation*}
		\begin{aligned}
		C_{4\pi}^{---}&=P_{4\pi}\cap\{H_{211}<2\pi,H_{121}<2\pi,H_{112}<2\pi\}, \\
		C_{4\pi}^{--+}&=P_{4\pi}\cap\{H_{211}<2\pi,H_{121}<2\pi,H_{112}>2\pi\}, \\
		C_{4\pi}^{-+-}&=P_{4\pi}\cap\{H_{211}<2\pi,H_{121}>2\pi,H_{112}<2\pi\}, \\
		C_{4\pi}^{+--}&=P_{4\pi}\cap\{H_{211}>2\pi,H_{121}<2\pi,H_{112}<2\pi\}, \\
		C_{4\pi}^{-++}&=P_{4\pi}\cap\{H_{211}<2\pi,H_{121}>2\pi,H_{112}>2\pi\}, \\
		C_{4\pi}^{+-+}&=P_{4\pi}\cap\{H_{211}>2\pi,H_{121}<2\pi,H_{112}>2\pi\}, \\
		C_{4\pi}^{++-}&=P_{4\pi}\cap\{H_{211}>2\pi,H_{121}>2\pi,H_{112}<2\pi\}, \\
		C_{4\pi}^{+++}&=P_{4\pi}\cap\{H_{211}>2\pi,H_{121}>2\pi,H_{112}>2\pi\},\\
		C_{4\pi}^{-}&=\{H_{222}<-4\pi\}.
		\end{aligned}
		\end{equation*}
		\item $P_{6\pi}$ is the reunion of the five cells defined below:
		\begin{equation*}
		\begin{aligned}
		C_{6\pi}^{221}&=P_{6\pi}\cap\{H_{221}<0,H_{112}>6\pi\}, \\
		C_{6\pi}^{212}&=P_{6\pi}\cap\{H_{212}<0,H_{121}>6\pi\}, \\
		C_{6\pi}^{122}&=P_{6\pi}\cap\{H_{122}<0,H_{211}>6\pi\}, \\
		C_{6\pi}^{+}&=P_{6\pi}\cap\{S>6\pi,H_{221}>0,H_{212}>0,H_{122}>0\}, \\
		C_{6\pi}^{-}&=P_{6\pi}\cap\{S<6\pi,H_{112}<6\pi,H_{121}<6\pi,H_{211}<6\pi\}.
		\end{aligned}
		\end{equation*}
		\item $P_{8\pi}$ is the reunion of the nine cells defined below:
		\begin{equation*}
		\begin{aligned}
		C_{8\pi}^{---}&=P_{8\pi}\cap\{H_{122}<4\pi,H_{212}<4\pi,H_{221}<4\pi\}, \\
		C_{8\pi}^{--+}&=P_{8\pi}\cap\{H_{122}<4\pi,H_{212}<4\pi,H_{221}>4\pi\}, \\
		C_{8\pi}^{-+-}&=P_{8\pi}\cap\{H_{122}<4\pi,H_{212}>4\pi,H_{221}<4\pi\}, \\
		C_{8\pi}^{+--}&=P_{8\pi}\cap\{H_{122}>4\pi,H_{212}<4\pi,H_{221}<4\pi\}, \\
		C_{8\pi}^{-++}&=P_{8\pi}\cap\{H_{122}<4\pi,H_{212}>4\pi,H_{221}>4\pi\}, \\
		C_{8\pi}^{+-+}&=P_{8\pi}\cap\{H_{122}>4\pi,H_{212}<4\pi,H_{221}>4\pi\}, \\
		C_{8\pi}^{++-}&=P_{8\pi}\cap\{H_{122}>4\pi,H_{212}>4\pi,H_{221}<4\pi\}, \\
		C_{8\pi}^{+++}&=P_{8\pi}\cap\{H_{122}>4\pi,H_{212}>4\pi,H_{221}>4\pi\}, \\
		C_{8\pi}^{+}&=\{H_{111}>10\pi\}.
		\end{aligned}
		\end{equation*}
	\end{itemize}
\end{Prop}

Again, we suggest to the reader to look at the figures in Section \ref{figures}, in order to understand this description for some particular cases.

\begin{proof}
	As usual, we obtain the third point by applying the symmetry $\psi$ to the first point (see Remark \ref{symmetry}), which allows us to prove only the first point among the two.
	\begin{itemize}
		\item Let us show the first point. According to Lemma \ref{connectedcomponents}, $P_{4\pi}$ has two connected components. By Proposition \ref{internalwalls}, the hyperbolic reducible walls $\{H_{ijk}=2\pi,\sigma_{\bar{i}\bar{j}\bar{k}}<2\pi\}$ are contained in $P_{4\pi}\backslash\{H_{222}<-4\pi\}$ for $(i,j,k)\in\{(2,1,1),(1,2,1),(1,1,2)\}$. Proposition \ref{internalintersections} states that they intersect pairwise in this subspace. According to Theorem \ref{polytopelayers} and Proposition \ref{internalwalls}, no other reducible walls may form sides of cells inside $P_{4\pi}$. Therefore, the cells that are contained in $P_{4\pi}$ may only be bounded by:
		\begin{itemize}
			\item reducible walls contained in $\partial P_{4\pi}$;
			\item hyperbolic reducible walls among $\{H_{ijk}=2\pi,\sigma_{\bar{i}\bar{j}\bar{k}}<2\pi\}$ for $$(i,j,k)\in\{(2,1,1),(1,2,1),(1,1,2)\}.$$
		\end{itemize}
		This leaves the $2^3=8$ cells that are listed in the statement. They correspond to the possible signs of $H_{ijk}-2\pi$ for the above values of $(i,j,k)$.
		
		\item We now prove the second point. According to Lemma \ref{connectedcomponents}, $P_{6\pi}$ is connected. By Proposition \ref{internalwalls}, the walls $\{H_{ijk}=0,\sigma_{\bar{i}\bar{j}\bar{k}}>4\pi\}$ are contained in $P_{6\pi}\cap\{S>6\pi\}$ for $(i,j,k)\in\{(2,2,1),(2,1,2),(1,2,2)\}$; and the walls $\{H_{ijk}=6\pi,\sigma_{\bar{i}\bar{j}\bar{k}}<2\pi\}$ are contained in $P_{6\pi}\cap\{S<6\pi\}$ for $(i,j,k)\in\{(1,1,2),(1,2,1),(2,1,1)\}$. Besides, the spherical reducible wall $\Sigma_{6\pi}$ is also contained in $P_{6\pi}$.
		According to Proposition \ref{internalintersections}, the only intersections between these walls are totally reducible facets. By Theorem \ref{polytopelayers} and Proposition \ref{internalwalls}, no other reducible walls may form sides of cells inside $P_{6\pi}$. Therefore, each of these walls is a side of exactly two cells.
		
		Now, the subspaces $\{H_{ijk}<0,H_{\bar{i}\bar{j}\bar{k}}>6\pi\}$ do not contain any of the other reducible walls for $(i,j,k)\in\{(2,2,1),(2,1,2),(1,2,2)\}$. Therefore, they form inside $P_{6\pi}$ the cells $C_{6\pi}^{221},C_{6\pi}^{212},C_{6\pi}^{122}$.
		
		Since $\Sigma_{6\pi}$ intersects all the walls mentioned above, the other cells must contain $\Sigma_{6\pi}$ among their sides. So we obtain the cells $C_{6\pi}^{+}$ and $C_{6\pi}^{-}$.
	\end{itemize}
	This concludes the proof.
\end{proof}

\begin{Theo}\label{fulltheorem}
	The twenty-three cells described in Proposition \ref{fullcells} are full. In other words,
	$$P_{4\pi}\subset\cP(G)_\omega, \ P_{6\pi}\subset\cP(G)_1, \ P_{8\pi}\subset\cP(G)_{\omega^2}.$$
\end{Theo}

\begin{proof}
	As usual, the proof of the third point is symmetric to the proof of the first point (see Remark \ref{symmetry}), which is why we only prove the first one among the two.
	\begin{itemize}
		\item Let us prove the first point. The cell $C_{4\pi}^{-}$ has a side contained in $\partial P_{4\pi}$: the wall $\{H_{222}=-4\pi\}$. According to Proposition \ref{complementcells} and Theorem \ref{emptycells}, this wall is a side of $C_{4\pi}^{c,1}$, which is an empty cell in $P_{4\pi}^c$. So by Proposition \ref{emptyfullinterface}, the cell $C_{4\pi}^{-}$ is full.
		
		To fill the eight other cells, recall Proposition \ref{decompfamily2}: we constructed an irreducible representation with angle pairs $(2\pi/3,\pi/3),$ $(2\pi/3,\pi/3),$ $(2\pi/3,\pi/3)$. Since this triple belongs to the hyperbolic reducible walls $\{H_{ijk}=2\pi,\sigma_{\bar{i}\bar{j}\bar{k}}<2\pi\}$ for $(i,j,k)\in\{(2,1,1),(1,2,1),(1,1,2)\},$ this means that the eight cells $C_{4\pi}^{---}$, $C_{4\pi}^{--+}$, $C_{4\pi}^{-+-}$, $C_{4\pi}^{+--}$, $C_{4\pi}^{-++}$, $C_{4\pi}^{+-+}$, $C_{4\pi}^{++-}$, $C_{4\pi}^{+++}$ have an irreducible representation lying on one of their sides. By Proposition \ref{emptyfullinterface}, they are all full.
		
		\item We now prove the second point. According to Proposition \ref{totallyreduciblefacets}, the cells inside $P_{6\pi}$ all contain in their boundary a totally reducible facet of type 4. So by Corollary \ref{totallyreduciblefilling}, they are full.
	\end{itemize}
	This concludes the proof.
\end{proof}

Putting together Theorems \ref{emptycells} and \ref{fulltheorem} yields Theorem \ref{solutionpu}.

\bibliography{biblio}
\bibliographystyle{plain}
	
\end{document}